\documentclass[a4paper,11pt]{article}
\usepackage{amsmath,amsthm,amssymb,enumitem,xcolor}
\usepackage[hang]{footmisc}

\usepackage[nosort,nocompress,noadjust]{cite}

\usepackage[bookmarks=false,hyperfootnotes=false,colorlinks,linktoc=all,
    linkcolor={red!60!black},
    citecolor={blue!50!black},
    urlcolor={blue!80!black}]{hyperref}

\renewcommand{\eqref}[1]{\hyperref[#1]{(\ref{#1})}}

\newlist{enumlist}{enumerate}{2}
\setlist[enumlist,1]{labelindent=0cm,label=(\roman*),ref=(\roman*),labelwidth=4.5ex,labelsep=1ex,leftmargin=5.5ex,align=right,topsep=0.5ex,itemsep=1ex,parsep=1ex}
\setlist[enumlist,2]{labelindent=0cm,label=\alph*),ref=\arabic*,labelwidth=5ex,labelsep=0.5ex,leftmargin=5.5ex,align=left,topsep=0.5ex,itemsep=1ex,parsep=1ex}

\newlist{enumlistprime}{enumerate}{1}
\setlist[enumlistprime]{labelindent=0cm,label=(\roman*)',ref=(\roman*)',labelwidth=4.5ex,labelsep=1ex,leftmargin=5.5ex,align=right,topsep=0.5ex,itemsep=1ex,parsep=1ex}

\newlist{enumlistprimeprime}{enumerate}{1}
\setlist[enumlistprimeprime]{labelindent=0cm,label=(\roman*)'',ref=(\roman*)'',labelwidth=4.5ex,labelsep=1ex,leftmargin=5.5ex,align=right,topsep=0.5ex,itemsep=1ex,parsep=1ex}

\newlist{itemlist}{itemize}{2}
\setlist[itemlist,1]{labelindent=0cm,label=$\bullet$,labelwidth=2.5ex,labelsep=0.5ex,leftmargin=3ex,align=left,topsep=0.5ex,itemsep=1ex,parsep=1ex}
\setlist[itemlist,2]{labelindent=0cm,label=$\circ$,labelwidth=2.5ex,labelsep=0.5ex,leftmargin=3ex,align=left,topsep=0.5ex,itemsep=1ex,parsep=1ex}

\numberwithin{equation}{section}

{\theoremstyle{definition}\newtheorem{definitiona}{Definition}[section]
\newtheorem{remarka}[definitiona]{Remark}
\newtheorem{examplea}[definitiona]{Example}}

\newtheorem{propositiona}[definitiona]{Proposition}
\newtheorem{lemmaa}[definitiona]{Lemma}
\newtheorem{theorema}[definitiona]{Theorem}
\newtheorem{corollarya}[definitiona]{Corollary}

\newtheorem{letterthma}{Theorem}
\renewcommand{\theletterthma}{\Alph{letterthma}}

{\theoremstyle{definition}
\newtheorem{letterremarka}[letterthma]{Remark}

}

\newenvironment{definition}[1][]{\begin{definitiona}[#1]\setlist*[enumlist,1]{label=(\roman*),ref=\thedefinitiona(\roman*)}}{\end{definitiona}}
\newenvironment{remark}[1][]{\begin{remarka}[#1]\setlist*[enumlist,1]{label=(\roman*),ref=\theremarka(\roman*)}}{\end{remarka}}
\newenvironment{example}[1][]{\begin{examplea}[#1]\setlist*[enumlist,1]{label=(\roman*),ref=\theexamplea(\roman*)}}{\end{examplea}}
\newenvironment{proposition}[1][]{\begin{propositiona}[#1]\setlist*[enumlist,1]{label=(\roman*),ref=\thepropositiona(\roman*)}}{\end{propositiona}}
\newenvironment{lemma}[1][]{\begin{lemmaa}[#1]\setlist*[enumlist,1]{label=(\roman*),ref=\thelemmaa(\roman*)}}{\end{lemmaa}}
\newenvironment{theorem}[1][]{\begin{theorema}[#1]\setlist*[enumlist,1]{label=(\roman*),ref=\thetheorema(\roman*)}}{\end{theorema}}
\newenvironment{corollary}[1][]{\begin{corollarya}[#1]\setlist*[enumlist,1]{label=(\roman*),ref=\thecorollarya(\roman*)}}{\end{corollarya}}
\newenvironment{letterthm}[1][]{\begin{letterthma}[#1]\setlist*[enumlist,1]{label=(\roman*),ref=\theletterthma(\roman*)}}{\end{letterthma}}

\newcommand{\C}{\mathbb{C}}
\newcommand{\F}{\mathbb{F}}
\newcommand{\bG}{\mathbb{G}}
\newcommand{\N}{\mathbb{N}}
\newcommand{\R}{\mathbb{R}}
\newcommand{\T}{\mathbb{T}}
\newcommand{\Q}{\mathbb{Q}}
\newcommand{\Z}{\mathbb{Z}}

\newcommand{\cA}{\mathcal{A}}
\newcommand{\cB}{\mathcal{B}}

\newcommand{\cF}{\mathcal{F}}
\newcommand{\cG}{\mathcal{G}}
\newcommand{\cH}{\mathcal{H}}

\newcommand{\cK}{\mathcal{K}}

\newcommand{\cM}{\mathcal{M}}
\newcommand{\cN}{\mathcal{N}}
\newcommand{\cO}{\mathcal{O}}
\newcommand{\cP}{\mathcal{P}}

\newcommand{\cU}{\mathcal{U}}

\newcommand{\cZ}{\mathcal{Z}}

\newcommand{\al}{\alpha}
\newcommand{\be}{\beta}
\newcommand{\eps}{\varepsilon}
\newcommand{\vphi}{\varphi}
\newcommand{\om}{\omega}
\newcommand{\Om}{\Omega}
\newcommand{\si}{\sigma}

\newcommand{\Ad}{\operatorname{Ad}}

\newcommand{\Aut}{\operatorname{Aut}}

\newcommand{\inv}{\operatorname{inv}}

\newcommand{\lspan}{\operatorname{span}}

\newcommand{\Tr}{\operatorname{Tr}}

\newcommand{\betil}{\widetilde{\beta}}

\newcommand{\Mtil}{\widetilde{M}}

\newcommand{\Ntil}{\widetilde{N}}

\newcommand{\Stil}{\widetilde{S}}

\newcommand{\vbar}{\overline{v}}

\newcommand{\ot}{\otimes}
\newcommand{\id}{\mathord{\text{\rm id}}}
\newcommand{\ovt}{\mathbin{\overline{\otimes}}}
\newcommand{\actson}{\curvearrowright}

\newcommand{\dpr}{^{\prime\prime}}
\newcommand{\op}{^\text{\rm op}}

\newcommand{\Wred}{W_{\text{\rm red}}}
\newcommand{\otalg}{\otimes_{\text{\rm alg}}}
\newcommand{\otmin}{\otimes_{\text{\rm min}}}
\newcommand{\free}{\operatornamewithlimits{{\text{\LARGE $\ast$}}}}
\newcommand{\link}{\operatorname{link}}
\newcommand{\ster}{\operatorname{star}}

\newcommand{\alg}{_{\text{\rm alg}}}
\newcommand{\cb}{_{\text{\rm cb}}}

\newcommand{\bim}[3]{\mathord{\raisebox{-0.4ex}[0ex][0ex]{\scriptsize $#1$}\hspace{0.1ex}{#2}\hspace{-0.1ex}\raisebox{-0.4ex}[0ex][0ex]{\scriptsize $#3$}}}

\newcommand{\rightmod}[2]{{#1}\hspace{-0.1ex}\raisebox{-0.4ex}[0ex][0ex]{\scriptsize $#2$}}

\renewcommand{\leq}{\leqslant}
\renewcommand{\geq}{\geqslant}

\begin{document}

\begin{center}
{\boldmath\LARGE\bf W$^*$-correlations of II$_1$ factors and rigidity of\vspace{0.5ex}\\ tensor products and graph products}

\vspace{1ex}

{\sc by Daniel Drimbe\footnote{University of Iowa, Department of Mathematics, Iowa City (United States), daniel-drimbe@uiowa.edu\\ Supported by Engineering and Physical Sciences Research Council grant EP/X026647/1 and NSF Grant DMS-2452525.} and Stefaan Vaes\footnote{KU~Leuven, Department of Mathematics, Leuven (Belgium), stefaan.vaes@kuleuven.be\\ Supported by FWO research project G016325N of the Research Foundation Flanders and by Methusalem grant METH/21/03 –- long term structural funding of the Flemish Government.}}
\end{center}

\begin{abstract}\noindent
A variant of Gromov's notion of measure equivalence for groups has been introduced for II$_1$ factors under different names. We propose the terminology of W$^*$-correlated II$_1$ factors. We prove rigidity results up to W$^*$-correlations for tensor products and graph products of II$_1$ factors. As a consequence, we construct the first uncountable family of discrete groups $\Gamma$ that are not von Neumann equivalent, which means that their group von Neumann algebras $L(\Gamma)$ are not W$^*$-correlated, and which implies that these groups are neither measure equivalent, nor have isomorphic or virtually isomorphic group von Neumann algebras.
\end{abstract}
%
%

\section{Introduction and main results}

Both in group theory and in operator algebras, one considers several weak forms of isomorphism or equivalence between two discrete groups $\Gamma$ and $\Lambda$, or between two II$_1$ factors $A$ and $B$. These notions are often closely related when $A = L(\Gamma)$ and $B = L(\Lambda)$ are group von Neumann algebras. Most notably, two discrete groups $\Gamma$ and $\Lambda$ are called \emph{measure equivalent} (see \cite{Gro91}) if they admit commuting, free, measure preserving actions that each admit a finite measure fundamental domain. The prototype examples are lattices in the same ambient locally compact group. Next, two discrete groups $\Gamma$ and $\Lambda$ are called \emph{W$^*$-equivalent} if their group von Neumann algebras $L(\Gamma)$ and $L(\Lambda)$ are isomorphic. This then leads to natural rigidity questions: for which classes of groups, does measure equivalence, resp.\ W$^*$-equivalence imply (virtual) isomorphism of the groups?

In \cite{IPR19}, a new and even coarser equivalence between discrete groups $\Gamma$ and $\Lambda$ was found and coined \emph{von Neumann equivalence}, which we discuss below. They prove that both measure equivalence and W$^*$-equivalence imply von Neumann equivalence. It then becomes a difficult and challenging problem to distinguish between discrete groups up to this very coarse notion of von Neumann equivalence. So far, only qualitative notions, such as amenability, the Haagerup approximation property, Kazhdan's property~(T), weak amenability, etc., were shown to be stable under von Neumann equivalence, see \cite{IPR19,Ish21}.

Our first main result is a more quantitative rigidity statement for von Neumann equivalence, resulting in the first uncountable family of groups that are mutually not von Neumann equivalent. Note that these groups are at the same time not measure equivalent and not W$^*$-equivalent.

\begin{letterthm}[{See Theorem \ref{thm.uncountable-family-of-non-correlated}}]\label{thm.main-A}
For every subset $\cF \subset \{2,3,\ldots\}$, define the group
$$G_\cF = \free_{n \in \cF} \bigl(\underbrace{\F_2 \times \cdots \times \F_2}_{n \; \text{times}}\bigr) \; .$$
If $\cF \neq \cF'$, then $G_\cF$ and $G_{\cF'}$ are not von Neumann equivalent.
\end{letterthm}

We prove Theorem \ref{thm.main-A} as a consequence of more general von Neumann equivalence rigidity results for tensor products and graph products. Before presenting these, we introduce the following shift of terminology.

Already in \cite[Theorem 1.5]{IPR19}, it was shown that von Neumann equivalence of two discrete groups $\Gamma$ and $\Lambda$ actually is a property of their group von Neumann algebras $L(\Gamma)$ and $L(\Lambda)$. Thus, in \cite[Definition 1.4]{IPR19} (see Definition \ref{def.vNequiv} below), the notion of \emph{von Neumann equivalence} of two finite von Neumann algebras is defined. A variant of this notion was considered in \cite[Definition 5.1]{BV22} and coined \emph{measure equivalence}. In hindsight, we believe that both terminologies may lead to confusion. Such a measure equivalence (in the sense of \cite{BV22}) between finite von Neumann algebras $A$ and $B$ is, by definition, given by a special type of Hilbert $A$-$B$-bimodule. In Definition \ref{def.Wstar-corr}, we call such special $A$-$B$-bimodules \emph{W$^*$-correlations} between $A$ and $B$. We then say that $A$ and $B$ are \emph{W$^*$-correlated} if there exists a faithful W$^*$-correlation.

In Section \ref{sec.Wstar-corr-vs-vN-equiv}, we discuss the relation between being von Neumann equivalent and being W$^*$-correlated. In Proposition \ref{prop.Wstar-corr-vs-vN-equiv}, we prove that both notions are the same for II$_1$ factors, and that in general, two von Neumann equivalent finite von Neumann algebras are always W$^*$-correlated. We believe that the stability results in Proposition \ref{prop.direct-sums-etc} show that W$^*$-correlations provide a more natural and flexible setting when considering nonfactorial von Neumann algebras. This is confirmed by Proposition \ref{prop.Wstar-corr-amenable}, in which we prove that there are exactly three W$^*$-correlation equivalence classes of amenable finite von Neumann algebras with separable predual: the diffuse ones, the discrete ones and their direct sums. Section \ref{sec.corr} contains several other general results on W$^*$-correlations, and the one-sided notion of $P$-embeddings. In particular in Theorem \ref{thm.no-embed-inner-amenable}, we prove that a group von Neumann algebra $L(\Lambda)$ of a nonamenable, inner amenable group has no $P$-embedding into a free group factor.

As mentioned above, we deduce Theorem \ref{thm.main-A} from two results: a rigidity theorem for W$^*$-correlated tensor product factors and a rigidity theorem for W$^*$-correlated graph product II$_1$ factors.

By \cite{Oza03}, the group von Neumann algebra $L(\Gamma)$ of a hyperbolic group with infinite conjugacy classes (icc) is solid and, in particular, prime: it cannot be written as the tensor product of two II$_1$ factors. Moreover, in \cite[Theorem 1]{OP03}, the following unique prime factorization property was obtained: if $\Gamma_1,\ldots,\Gamma_n$ are such groups and if $L(\Gamma_1) \ovt \cdots \ovt L(\Gamma_n)$ is isomorphic to any tensor product $B_1 \ovt \cdots \ovt B_r$ of II$_1$ factors, we can partition $\{1,\ldots,n\}$ into nonempty subsets $S_1 \sqcup \cdots \sqcup S_r$ such that for every $k \in \{1,\ldots,r\}$, the tensor product $\ovt_{i \in S_k} L(\Gamma_i)$ is stably isomorphic to $B_k$. Recall here that two II$_1$ factors $A$ and $B$ are called \emph{stably isomorphic} if there exist nonzero projections $p \in A$ and $q \in B$ such that $pAp \cong qBq$.

Our next result is a W$^*$-correlation version of \cite[Theorem 1]{OP03}. We prove this result for tensor products of stably solid II$_1$ factors. We introduce this concept of stable solidity in detail Section \ref{sec.rel-solid}. Note that $L(\Gamma)$ is stably solid for every hyperbolic group $\Gamma$. In Section \ref{sec.rel-solid}, we also adapt some of the existing solidity theorems to the stable setting, and prove that the $q$-Gaussian II$_1$ factors, as well as the II$_1$ factors given by certain free quantum groups are stably solid.

\begin{letterthm}[{See Theorem \ref{thm.W-star-corr-tensor-product-any-tensor-product}}]\label{thm.main-B}
Let $A = A_1 \ovt \cdots \ovt A_n$ be a tensor product of nonamenable, stably solid II$_1$ factors. Let $B = B_1 \ovt \cdots \ovt B_r$ be an arbitrary tensor product of II$_1$ factors. Then $A$ is W$^*$-correlated to $B$ if and only if we can partition $\{1,\ldots,n\}$ into nonempty subsets $S_1 \sqcup \cdots \sqcup S_r$ such that for every $k \in \{1,\ldots,r\}$, the tensor product $A_{S_k} := \ovt_{i \in S_k} A_i$ is W$^*$-correlated to $B_k$.
\end{letterthm}

Note here that a special case of Theorem \ref{thm.main-B} says that if two tensor products $A_1 \ovt \cdots \ovt A_n$ and $B_1 \ovt \cdots \ovt B_m$ of nonamenable, stably solid II$_1$ factors are W$^*$-correlated, then $n=m$ and after permuting the indices, $A_i$ is W$^*$-correlated to $B_i$ for every $i$. We state this special case as Theorem \ref{thm.corr-tensor-product}, since it serves as an ingredient in the proof of Theorem \ref{thm.main-B}.

Since two measure equivalent groups have W$^*$-correlated group von Neumann algebras, we also note that Theorem \ref{thm.corr-tensor-product} can be seen as a W$^*$-correlation version of the following result from \cite{MS02}: if two direct products $\Gamma_1 \times \cdots \times \Gamma_n$ and $\Lambda_1 \times \cdots \times \Lambda_{m}$ of nonelementary torsion free hyperbolic groups are measure equivalent, then $n=m$, and after permuting the indices, $\Gamma_i$ is measure equivalent to $\Lambda_i$.

%

We then turn to rigidity results for \emph{graph product} II$_1$ factors up to W$^*$-correlations. Given a simple graph $\Gamma$, which is nothing but a symmetric relation $\sim$ on a vertex set $S$ such that for all $s \in S$, we have $s \not\sim s$, and given a labeling $(A_s)_{s \in S}$ of the vertices by tracial von Neumann algebras, the graph product $A_\Gamma$ was defined in \cite{CF14} and is a tracial von Neumann algebra generated by $(A_s)_{s \in S}$, subject to the relations that $A_s$ commutes with $A_t$ whenever $s \sim t$. The precise definition is slightly more subtle, and recalled in Section \ref{sec.graph-products}.

In \cite{BCC24}, the following isomorphism rigidity theorem was proven for graph products defined by \emph{rigid graphs}. In \cite{BCC24}, a simple graph $\Gamma$ with vertex set $S$ is called rigid if $\link(\link s) = \{s\}$ for all $s \in S$. Here, $\link s = \{t \in S \mid t \sim s\}$ and $\link(\cU) = \bigcap_{u \in \cU} \link u$. So a graph is rigid if for every vertex $s$, we have that $s$ is the only vertex that is connected to all vertices in $\link s$.

Let $\Gamma$ and $\Lambda$ be rigid, simple graphs with vertex sets $S$ and $T$. Label the vertices by ``allowed'' II$_1$ factors $(A_s)_{s \in S}$ and $(B_t)_{t \in T}$. By \cite[Theorem~A]{BCC24}, if the graph products $A_\Gamma$ and $B_\Lambda$ are isomorphic, the graphs $\Gamma$ and $\Lambda$ are isomorphic, and there exists a graph isomorphism $\si : S \to T$ such that $A_s$ is stably isomorphic to $B_{\si(s)}$ for all $s \in S$.

In Section \ref{sec.Wstar-corr-graph-products}, we analyze what can be said if two such graph products $A_\Gamma$ and $B_\Lambda$ are merely W$^*$-correlated. If $\Gamma$ and $\Lambda$ are rigid graphs as above, and if the vertices are labeled by nonamenable, stably solid II$_1$ factors $(A_s)_{s \in S}$ and $(B_t)_{t \in T}$, we prove in Theorem \ref{thm.corr-graph-product} the following result: if $A_\Gamma$ is W$^*$-correlated to $B_\Lambda$, then for every $s \in S$, there exists a $t \in T$ such that $A_{\link s}$ is W$^*$-correlated to $B_{\link t}$.

While this result in combination with Theorem \ref{thm.main-B}, is enough to deduce Theorem \ref{thm.main-A}, the result by no means implies that the graphs $\Gamma$ and $\Lambda$ are isomorphic. In Proposition \ref{prop.counterex}, we prove that one should not expect this: even when $\Gamma$ and $\Lambda$ are finite rigid graphs and the vertices are labeled by free group factors, it may happen that $A_\Gamma$ is W$^*$-correlated to $B_\Lambda$ without $\Gamma$ and $\Lambda$ being isomorphic.

When however $A_\Gamma$ is stably isomorphic to $B_\Lambda$, more can be said. We thus also prove the following variant of \cite[Theorem~A]{BCC24}.

\begin{letterthm}[{See Theorem \ref{thm.iso-graph-product}}]\label{thm.main-C}
If $A_\Gamma$ and $B_\Lambda$ are graph products given by rigid simple graphs $\Gamma$, $\Lambda$, labeled by nonamenable, stably solid II$_1$ factors $(A_s)_{s \in S}$ and $(B_t)_{t \in T}$, and if $A_\Gamma$ is stably isomorphic to $B_\Lambda$, then $\Gamma$ is isomorphic to $\Lambda$ and there exists a graph isomorphism $\si : S \to T$ such that $A_s$ is stably isomorphic to $B_{\si(s)}$ for all $s \in S$.
\end{letterthm}

To compare Theorem \ref{thm.main-C} to \cite[Theorem~A]{BCC24}, note that the first version of \cite{BCC24} only covered finite graphs. Then our paper generalized the result to infinite graphs, after which \cite{BCC24} was updated so as to cover infinite graphs as well. Next, one has to be careful about which II$_1$ factors are allowed as labeling of the vertices. In \cite{BCC24}, the authors consider nonamenable II$_1$ factors satisfying the \emph{strong (AO) condition} (see \cite[Definition 2.6]{HI15}). As we discuss in Section \ref{sec.rel-solid}, conjecturally, every tracial von Neumann algebra $A$ satisfying the strong (AO) condition is stably solid, because it is biexact by \cite[Theorem 7.17]{DP23}. The converse is not true because of the examples in \cite[Section 8]{DP23}, as we explain just before Theorem \ref{thm.stably-solid-and-q-Gaussian}. Most importantly, both stable solidity and the strong (AO) condition hold for the usual natural families of examples, such as group von Neumann algebras of hyperbolic groups.

\section{Preliminaries}

\subsection{Popa's intertwining-by-bimodules}

Throughout this paper, we make use of \emph{intertwining-by-bimodules} as introduced by Popa in \cite[Section 2]{Pop03}. Recall that for a tracial von Neumann algebra $(M,\tau)$ and von Neumann subalgebras $A \subset p(M_n(\C) \ot M)p$ and $B \subset M$, one writes $A \prec_M B$ if the Hilbert $A$-$B$-bimodule $p(\C^n \ot L^2(M))$ admits a nonzero closed $A$-$B$-subbimodule that is finitely generated as a right Hilbert $B$-module. This is equivalent to the existence of a nonzero positive element
$
d\in A'\cap p\bigl(M_n(\mathbb C)\,\bar\otimes\,\langle M,e_B\rangle\bigr)p
\quad\text{such that}\quad
(\operatorname{Tr}_n\otimes\operatorname{Tr})(d)<\infty,
$
where \(\operatorname{Tr}_n\) is the usual matrix trace and \(\operatorname{Tr}\) is the canonical semifinite trace on \(\langle M,e_B\rangle\).
Also recall that one defines $A \prec^f_M B$ if $A q \prec_M B$ for every nonzero projection $q \in A' \cap p(M_n(\C) \ot M)p$. Finally recall that there exists a unique projection $q$, potentially zero, in the center of the normalizer of $A$ inside $p (M_n(\C) \ot M)p$, such that $A q \prec^f_M B$ and $A(p-q) \not\prec_M B$.

The next lemma is a well-known result. Since we could not find a proof in the literature, we include one here for the convenience of the reader.

\begin{lemma}\label{lemma.intertwining.known}
Let $(M,\tau)$ be a tracial von Neumann algebra and let $A\subset aMa$ and $B\subset M$ be von Neumann subalgebras. Then $A\prec_M^f B$ if and only if there exists an $A$-central normal positive functional $\Omega: a\langle M,e_B\rangle a\to\mathbb C$ such that $\Omega_{|aMa}=\tau$.
\end{lemma}

\begin{proof}
Write $\mathcal N=\langle M,e_B\rangle$ and denote by ${\rm Tr}$ its canonical faithful normal semifinite trace.

Assume that $A\prec_M^f B$. By a standard maximality argument, it is enough to show that for every nonzero projection $p\in A'\cap aMa$, there exist a nonzero projection $q\in (Ap)'\cap pMp$ and a normal positive $Aq$-central functional on $q\mathcal Nq$ whose restriction to $qMq$ is $\tau$.

Hence, we let $p\in A'\cap aMa$ be an arbitrary nonzero projection. Since $Ap\prec_M B$, there is a nonzero projection $e\in (Ap)'\cap p\mathcal Np$ with ${\rm Tr}(e)<\infty$. Note that $\varphi: p\mathcal Np\to\mathbb C$ defined by $\varphi(x)={\rm Tr}(exe)$ is normal, positive and $Ap$-central. By letting $k\in L^1(pMp,\tau)_{+}$ be the density of $\varphi_{|pMp}$, we  get
$\varphi(x)=\tau(kx)$ for any $ x\in pMp.$
By $Ap$-centrality and uniqueness of the density, $k$ commutes with $Ap$. Since $\tau(k)={\rm Tr}(e)>0$, we may choose $\varepsilon>0$ such that $q:=1_{[\varepsilon,\infty)}(k)\neq0\in (Ap)'\cap pMp$. Moreover, $b=k^{-1/2}q$ is a bounded element of $pMp$ commuting with $Ap$. Thus, the functional
$\psi: q\mathcal Nq\to\mathbb C, \psi(x)=\varphi(bxb),$
is normal, positive and $Aq$-central. Since  $bkb=q$, it follows that $\psi(x)=\tau(x)$  for any $x\in qMq$.
This proves the direct implication.

Conversely, suppose that such a functional $\Omega$ exists and let $p\in A'\cap aMa$ be an arbitrary nonzero projection. The restriction $\Omega_p=\Omega_{|p\mathcal Np}$ is normal, positive and $Ap$-central. Let $h_p\in L^1(p\mathcal Np,{\rm Tr})_{+}$ be its density, so that $
\Omega_p(x)={\rm Tr}(h_px), x\in p\mathcal Np.$ As before, by $Ap$-centrality and uniqueness of the density, we derive that $uh_pu^*=h_p$ for every $u\in\mathcal U(Ap)$. By considering a spectral projection of $h_p$, we obtain a nonzero projection $e\in (Ap)'\cap p\mathcal N p$ with ${\rm Tr}(e)<\infty$, proving the lemma.
\end{proof}

The following lemma is a straightforward generalization of \cite[Corollary 2.3]{Pop03}. For completeness, we give a detailed proof.

\begin{lemma}\label{lem.Popa-intertwining-for-bimodules}
Let $(A,\tau)$ and $(B,\tau)$ be tracial von Neumann algebras and $\bim{A}{H}{B}$ a Hilbert $A$-$B$-bimodule. Then the following are equivalent.
\begin{enumlist}
\item $\{0\}$ is the only closed $A$-$B$-subbimodule of $H$ that is finitely generated as a right Hilbert $B$-module.
\item There exists a net of unitaries $(u_i)_{i \in I}$ in $A$ such that $\lim_{i \in I} \|P_K(u_i \xi)\| = 0$ for every $\xi \in H$ and every finitely generated right Hilbert $B$-submodule $K \subset H$.
\end{enumlist}
\end{lemma}
\begin{proof}
(ii) $\Rightarrow$ (i). Let $K \subset H$ be a closed $A$-$B$-subbimodule that is finitely generated as a right Hilbert $B$-module. Choose an arbitrary $\xi \in H$. Since $K$ is a left $A$-module, we have that $P_K(u_i \xi) = u_i P_K(\xi)$. Since every $u_i$ is a unitary, it follows that
$$\|P_K(\xi)\| = \lim_{i \in I} \|u_i P_K(\xi)\| = \lim_{i \in I} \|P_K(u_i \xi)\| = 0 \; .$$
Since this holds for every $\xi \in H$, we conclude that $K = \{0\}$.

(i) $\Rightarrow$ (ii). Recall that a vector $\xi \in H$ is called right bounded if there exists a $C > 0$ such that $\|\xi b\| \leq C \|b\|_2$ for all $b \in B$. For every right bounded $\xi \in H$, we denote by $L_\xi$ the bounded, right $B$-linear operator $L_\xi : L^2(B) \to H : L_\xi(b) = \xi b$, for all $b \in B$. We denote by $\cM \subset B(H)$ the commutant of the right $B$-action on $H$. Note that for all right bounded $\xi,\eta \in H$, the operator $L_\eta L_\xi^*$ belongs to $\cM$, while the operator $L_\xi^* L_\eta$ belongs to $B$, viewed as left multiplication operators on $L^2(B)$. Recall that $\cM$ has a unique normal faithful semifinite trace $\Tr$ satisfying $\Tr(L_\eta L_\xi^*) = \tau(L_\xi^* L_\eta)$ for all right bounded $\xi,\eta \in H$.

We claim that there exists a net of unitaries $(u_i)_{i \in I}$ in $A$ such that $\lim_{i \in I} \|L_\xi^* u_i L_\eta\|_2 = 0$ for all right bounded vectors $\xi,\eta \in H$. Assume the contrary. We then find $\eps > 0$ and right bounded vectors $\xi_1,\ldots,\xi_k \in H$ such that
\begin{equation}\label{eq.big-eps}
\sum_{s,t=1}^k \|L_{\xi_s}^* u L_{\xi_t}\|_2^2 \geq \eps \quad\text{for all $u \in \cU(A)$.}
\end{equation}
Define $T \in \cM^+$ by $T = \sum_{t=1}^k L_{\xi_t} L_{\xi_t}^*$. Then \eqref{eq.big-eps} says that
\begin{equation}\label{eq.T-eps}
\Tr(T u T u^*) \geq \eps \quad\text{for all $u \in \cU(A)$.}
\end{equation}
Define $S \in \cM^+$ as the unique element of minimal $\|\cdot\|_{2,\Tr}$ in the weakly closed convex hull of $\{u T u^* \mid u \in \cU(A)\}$ in $\cM$. By \eqref{eq.T-eps}, we get that $\Tr(T S) \geq \eps$, so that $S \neq 0$. By uniqueness, we find that $S = u S u^*$ for all $u \in \cU(A)$. Choose $\delta > 0$ such that the spectral projection $P = 1_{[\delta,+\infty)}(S)$ of $S$ is nonzero. Then $P$ is a nonzero projection in $\cM$ that commutes with $A$. Since $S \in \cM^+$ and $\Tr(S) < +\infty$, also $\Tr(P) < +\infty$.

Denote by $K \subset H$ the image of $P$. Then $K$ is a right Hilbert $B$-submodule of $H$. Since $\Tr(P) < +\infty$, this right Hilbert $B$-module $K$ has finite $B$-dimension w.r.t.\ $\tau$. This means that $K \cong p(\ell^2(\N) \ot L^2(B))$ as right Hilbert $B$-module for some projection $p \in B(\ell^2(\N)) \ovt B$ with $(\Tr \ot \tau)(p) < +\infty$. Take a central projection $z_0 \in \cZ(B)$ such that $(1 \ot z_0) p \neq 0$ and such that $(1 \ot z_0) p$ has bounded center-valued trace. Denoting by $\rho : \cZ(B) \to B(H)$ the $*$-homomorphism given by the right action of $\cZ(B)$, it follows that $\rho(z_0)K$ is a nonzero closed $A$-$B$-subbimodule of $H$ that is finitely generated as a right Hilbert $B$-module. This contradicts (i), so that the claim is proven.

We conclude by proving that the net $(u_i)_{i \in I}$ given by the claim satisfies (ii). Fix an arbitrary finitely generated right Hilbert $B$-submodule $K \subset H$. Take a projection $p \in M_n(\C) \ot B$ and a right $B$-linear unitary $V : p(\C^n \ot L^2(B)) \to K$. Define $\eta_s \in K$ by $\eta_s = V(p(e_s \ot 1))$. Then every $\eta_s$ is a right bounded vector with $\|L_{\eta_s}\| \leq 1$, and $P_K = \sum_{s=1}^n L_{\eta_s} L_{\eta_s}^*$. So when $\xi_0 \in H$ is any right bounded vector, we find that
$$\|P_K(u_i \xi_0)\| \leq \sum_{s =1}^n \|L_{\eta_s} L_{\eta_s}^* (u_i \xi_0)\| \leq \sum_{s =1}^n \|L_{\eta_s}^* (u_i \xi_0)\| = \sum_{s=1}^n \|L_{\eta_s}^* u_i L_{\xi_0}\|_2$$
for all $i \in I$. The right hand side tends to zero for $i \in I$, so that $\lim_{i \in I} \|P_K(u_i \xi_0)\| = 0$. Since the right bounded vectors are dense in $H$, it follows that $\lim_{i \in I} \|P_K(u_i \xi)\| = 0$ for all $\xi \in H$, so that (ii) is proven.
\end{proof}

The following is another variant of \cite[Corollary 2.3]{Pop03} and of the preceding Lemma \ref{lem.Popa-intertwining-for-bimodules}. Again we give a detailed proof for the convenience of the reader.

\begin{lemma}\label{lem.intertwining-weak-normalizer}
Let $(M,\tau)$ be a tracial von Neumann algebra with von Neumann subalgebras $A,B \subset M$. Denote by $E_M(A,B)$ the set of all $x \in M$ for which there exist $n \in \N$ and $x_1,\ldots,x_n \in M$ with $A x \subset x_1 B + \cdots + x_n B$. Note that $E_M(A,B)$ is an $A$-$B$-subbimodule of $M$.
\begin{enumlist}
\item\label{lem.intertwining-weak-normalizer.1} If $K \subset L^2(M)$ is a closed $A$-$B$-subbimodule that is finitely generated as a right Hilbert $B$-module, then $K$ is contained in the closure of $E_M(A,B)$ in $L^2(M)$.
\item\label{lem.intertwining-weak-normalizer.2} There exists a net of unitaries $(u_i)_{i \in I}$ in $\cU(A)$ such that $\lim_{i \in I} \|P_K(u_i \xi)\|_2=0$ for every $\xi \in L^2(M) \cap E_M(A,B)^\perp$ and every finitely generated right Hilbert $B$-submodule $K \subset L^2(M)$.
\end{enumlist}
\end{lemma}
\begin{proof}
(i) Take a bimodule $K$ as in (i). Choose $n \in \N$, a projection $p \in M_n(\C) \ot B$ and a unital normal $*$-homomorphism $\vphi : A \to p(M_n(\C) \ot B)p$ such that $\bim{A}{K}{B} \cong \bim{\vphi(A)}{p(\C^n \ot L^2(B))}{B}$. Choose an $A$-$B$-bimodular unitary $V : p(\C^n \ot L^2(B)) \to K$. Define $\eta \in (\C^n)^* \ot K$ by $\eta = \sum_{s=1}^n e_s^* \ot V(p(e_s \ot 1))$. Note that for all $a \in A$, because $V$ is $A$-$B$-bimodular,
\begin{align*}
a \eta &= \sum_{s=1}^n e_s^* \ot V(\vphi(a)(e_s \ot 1)) = \sum_{s,t=1}^n e_s^* \ot V(p(e_t \ot (\vphi(a))_{ts})) \\
&= \sum_{s,t=1}^n e_s^* \ot \bigl(V(p(e_t \ot 1)) \cdot (\vphi(a))_{ts}\bigr) = \eta \vphi(a) \; .
\end{align*}
So, $a \eta = \eta \vphi(a)$ for all $a \in A$. In particular, $\eta = \eta p$.

Since $K \subset L^2(M)$, we can view $\eta \eta^*$ as an element of $L^1(M)$. Since $a \eta = \eta \vphi(a)$ for all $a \in A$, we get that $\eta \eta^*$ commutes with $A$. We can then choose a sequence of spectral projections $q_k$ of $\eta \eta^*$ such that $q_k \to 1$ strongly and such that $q_k \eta \eta^*$ is bounded for every $k$. Note that $q_k \in A' \cap M$ for all $k$. Write $\eta_k = q_k \eta$ and fix $k$. We have that $\eta_k \in (\C^n)^* \ot M$ and $a \eta_k = \eta_k \vphi(a)$ for all $a \in A$. So, $\eta_k \in (\C^n)^* \ot E_M(A,B)$.

Denote by $H \subset L^2(M)$ the closure of $E_M(A,B)$ inside $L^2(M)$. Since $q_k \to 1$ strongly, we conclude that $\eta \in (\C^n)^* \ot H$. By definition of $\eta$, it follows that the image of $V$ sits inside $H$. So, $K \subset H$ and (i) is proven.

(ii) By (i) and Lemma \ref{lem.Popa-intertwining-for-bimodules}, take a net of unitaries $(u_i)_{i \in I}$ in $\cU(A)$ such that $\lim_{i \in I} \|P_{K'}(u_i \xi)\|_2 = 0$ for every $\xi \in L^2(M) \cap E_M(A,B)^\perp$ and every closed right $B$-submodule $K' \subset L^2(M) \cap E_M(A,B)^\perp$ that is finitely generated as a right Hilbert $B$-module. Let $K \subset L^2(M)$ be a finitely generated right Hilbert $B$-module. Define $K' \subset L^2(M) \cap E_M(A,B)^\perp$ as the cokernel of the restriction of $P_K$ to $L^2(M) \cap E_M(A,B)^\perp$. As a right Hilbert $B$-module, $K'$ is isomorphic with a Hilbert $B$-submodule of $K$, and thus finitely generated. Since $P_K(u_i \xi) = P_K(P_{K'}(u_i \xi))$, we get that $\limsup_{i \in I} \|P_{K}(u_i \xi)\|_2 \leq \limsup_{i \in I} \|P_{K'}(u_i \xi)\|_2 = 0$.
\end{proof}

\subsection{Graph products of tracial von Neumann algebras}\label{sec.graph-products}

In \cite{CF14}, the concept of a \emph{graph product} of von Neumann algebras was defined. The initial data consists of a simple graph $\Gamma = (S,E)$ and a labeling of the vertices by tracial von Neumann algebras $(B_s,\tau_s)_{s \in S}$. Recall here that we call $(S,E)$ a simple graph when $S$ is a set and $E \subset S \times S$ a subset such that the following holds: for all $(s,t) \in E$, also $(t,s) \in E$; and for all $s \in S$, we have $(s,s) \not\in E$.

In analogy with the classical construction of a graph product of groups, in \cite{CF14}, a canonical tracial von Neumann algebra $(B_\Gamma,\tau)$ is defined such that $M$ is generated by the subalgebras $B_s \subset B_\Gamma$, $s \in S$, and such that for all $(s,t) \in E$, we have that $B_s$ commutes with $B_t$. To rigorously define $(B_\Gamma,\tau)$, we however need to be more careful, proceeding as follows.

Fix a simple graph $\Gamma = (S,E)$ and fix tracial von Neumann algebras $(B_s,\tau_s)$, $s \in S$. For $s,t \in S$, we write $s \sim t$ iff $(s,t) \in E$. Recall that the right-angled Coxeter group $C_\Gamma$ associated to such a simple graph $\Gamma$ is the group with generating set $S$ and relations $s^2 = e$ for all $s \in S$ and $s t = t s$ whenever $s \sim t$. As a group, $C_\Gamma$ is the graph product given by $\Gamma$ and labeling all the vertices with the group $\Z/2\Z$.

Recall from \cite[Theorem 3.4.2]{Dav08} that the word problem for $C_\Gamma$ has the following easy solution. One says that a finite word $w$ in the alphabet $S$ is \emph{reduced} if $w$ does not contain a subword of the form $s v s$ with $s \in S$ and $s \sim t$ for every letter $t \in v$. By convention, also the empty word $e$ is called reduced. One checks that a word of the form $v_1 s t v_2$, with $v_1$, $v_2$ words and $s,t \in S$ with $s \sim t$, is reduced if and only if $v_1 t s v_2$ is reduced. One denotes by $\Wred$ the set of reduced words and denotes by $\approx$ the smallest equivalence relation on $\Wred$ generated by $v_1 s t v_2 \approx v_1 t s v_2$. One then proves that two reduced words $w_1$, $w_2$ represent the same element of $C_\Gamma$ if and only if $w_1 \approx w_2$.

From this solution of the word problem, it also follows that for every subset $T \subset S$, the subgroup of $C_\Gamma$ generated by $T$ precisely is the right-angled Coxeter group associated with the subgraph $(T,E \cap (T \times T))$.

For every $s \in S$, define the Hilbert space $H_s = L^2(B_s,\tau_s)$ and the closed subspace $H_s^\circ = L^2(B_s) \ominus \C 1$. For every reduced word $w = s_1 \cdots s_k$, define
$$H_w^\circ = H_{s_1}^\circ \ot \cdots \ot H_{s_k}^\circ \; .$$
When two reduced words $w = s_1 \cdots s_k$ and $w'=s'_1 \cdots s'_{k'}$ represent the same element of $C_\Gamma$, we have that $w \approx w'$, so that in particular $k = k'$, the letters appearing in $w$ are the same as the letters appearing in $w'$ and these letters all appear the same number of times. There is thus a unique bijection $\si : \{1,\ldots,k\} \to \{1,\ldots,k'\}$ such that $s'_{\si(i)} = s_i$ for all $i \in \{1,\ldots,k\}$ and such that $\si(i) < \si(j)$ whenever $i \leq j$ and $s_i = s_j$. This bijection induces a unitary $H_w^\circ \to H_{w'}^\circ$ that permutes the tensor factors.

Using these permutation unitaries to identify $H_w^\circ \cong H_{w'}^\circ$ when $w$, $w'$ are reduced words that represent the same element of $C_\Gamma$, one defines $H_g^\circ = H_w^\circ$ for every $g \in C_\Gamma$ represented by a reduced word $w$.

One denotes by $|g|$ the word length of an element $g \in C_\Gamma$ w.r.t.\ the generating set $S$. By the discussion above, $|g|$ equals the number of letters in any representation $g = w$ as a reduced word $w$. One then gets a canonical identification $H_{gh}^\circ = H_g^\circ \ot H_h^\circ$ whenever $g,h \in C_\Gamma$ and $|gh| = |g| + |h|$. One defines
$$H = \C \Omega \oplus \bigoplus_{g \in C_\Gamma \setminus \{e\}} H_g^\circ \; .$$
For every $s \in S$, define $L_s = \{g \in C_\Gamma \mid |sg| = 1+|g|\}$. Then $C_\Gamma = s L_s \sqcup L_s$. Writing $K_s = \C \oplus \bigoplus_{g \in L_s \setminus \{e\}} H_g^\circ$, we get the canonical identification
$$H = \C \Omega \oplus (H_s^\circ \ot K_s) \oplus (K_s \ominus \C) = (H_s^\circ \ot K_s) \oplus K_s = H_s \ot K_s \; .$$
Using the representation of $B_s$ on $H_s = L^2(B_s)$ by left multiplication, one represents each $B_s$ on the Hilbert space $H$. One defines $B_\Gamma \subset B(H)$ as the von Neumann algebra generated by these representations of $B_s$, $s \in S$. Then, $\Omega$ is a cyclic and separating vector for $B_\Gamma$ that defines a faithful normal tracial state $\tau$ on $B_\Gamma$.

For every $s \in S$, denote $B_s^\circ = \{b \in B_s \mid \tau_s(b)=0\}$. For every $g \in C_\Gamma \setminus \{e\}$, represented by a reduced word $g = s_1 \cdots s_n$ in the alphabet $S$, we define $B_g^\circ \subset B_\Gamma$ as the linear subspace spanned by $B_{s_1}^\circ \cdots B_{s_n}^\circ$. One checks that this definition of $B_g^\circ$ does not depend on the choice of the reduced expression for $g$. By construction, the closure of $B_g^\circ \Omega$ equals $H_g^\circ$. So if $g,h \in C_\Gamma \setminus \{e\}$ are distinct elements, $B_g^\circ \perp B_h^\circ$ in $L^2(B_\Gamma)$.

If $S_0 \subset S$, we denote by $B_{S_0} \subset B_\Gamma$ the von Neumann subalgebra generated by $B_s$, $s \in S_0$. We similarly define the subgroup $C_{S_0} \subset C_\Gamma$ generated by $S_0$. Then $L^2(B_{S_0})$ is densely spanned by $1$ and $B_g^\circ$, $g \in C_{S_0} \setminus \{e\}$. It follows that for $S_1,S_2 \subset S$ and $S_0 = S_1 \cap S_2$, we have
$$E_{B_{S_1}} \circ E_{B_{S_2}} = E_{B_{S_0}} = E_{B_{S_2}} \circ E_{B_{S_1}} \; ,$$
which means that $B_{S_1}$ and $B_{S_2}$ are in a commuting square position inside $B_\Gamma$.

\subsection{\boldmath The $q$-Gaussian II$_1$ factors}\label{sec.q-Gaussian}

Recall from \cite{BS90} the construction of the tracial von Neumann algebras $(M_q(H_\R),\tau)$ associated with a real Hilbert space $H_\R$ and a parameter $-1 \leq q \leq 1$. When $q=-1$, this corresponds to the CAR-algebra of $H_\R$, and thus the hyperfinite II$_1$ factor if $H_\R$ is separable and infinite dimensional. When $q=0$, one recovers Voiculescu's free Gaussian functor, so that $M_0(H_\R) \cong L(\F_{\dim H_\R})$. When $q=1$, the algebra $M_q(H_\R)$ is commutative and corresponds to the CCR-algebra of $H_\R$. It is canonically isomorphic with $L^\infty(X,\mu)$, where $(X,\mu)$ is the Gaussian probability space associated with $H_\R$.

Fix $-1 \leq q \leq 1$. Denote by $H = H_\R + i H_\R$ the canonical complexification of $H_\R$, with complex conjugation operation $\xi \mapsto \overline{\xi}$. Fix $n \in \N$ and write $[n] = \{1,\ldots,n\}$. Denote by $S_n$ the group of permutations of $[n]$ and for every $\si \in S_n$, denote by $\inv \si = \#\{(i,j) \mid 1 \leq i < j \leq n \; , \; \si(i) > \si(j)\}$ the number of inversions of $\si$. We consider the $n$-fold tensor product $H^{\ot^n}$ and denote for $\si \in S_n$ by $\pi_\si$ the unitary that permutes the tensor factors by $\si$. By \cite[Proposition 1]{BS90}, the map $S_n \to \R : \si \mapsto q^{\inv \si}$ is positive definite, so that the operator $T_n \in B(H^{\ot^n})$ defined by $T_n = \sum_{\si \in S_n} q^{\inv \si} \pi_\si$ is positive. When $-1 < q < 1$, the operator $T_n$ is also invertible.

We then define the new Hilbert space $H^{\ot_q^n}$, formally by separation and completion of the scalar product $\langle \xi,\eta \rangle_q = \langle T_n \xi,\eta \rangle$ on $H^{\ot^n}$. When $-1 < q < 1$, the vector spaces $H^{\ot_q^n}$ and $H^{\ot^n}$ are equal, but have a different scalar product. When $q=-1$, we can rather identify $H^{\ot_{-1}^n}$ with the closed subspace of $H^{\ot^n}$ of antisymmetric vectors, while $H^{\ot_1^n}$ is identified with the closed subspace of $H^{\ot^n}$ of symmetric vectors, both with a rescaled scalar product.

The $q$-Fock space is then defined by $\cF_q(H) = \C \Om \oplus \bigoplus_{n=1}^\infty H^{\ot_q^n}$. By convention, $\C \Om = H^{\ot_q^0}$. For every $\xi \in H$ and $n \geq 0$, we consider the left creation operator $\ell_q(\xi) : H^{\ot_q^n} \to H^{\ot_q^{(n+1)}} : \ell_q(\xi)(\mu) = \xi \ot \mu$. W.r.t.\ the $q$-deformed scalar product, we have that
$$\ell_q(\xi)^*(\mu_1 \ot \cdots \ot \mu_n) = \sum_{k=1}^n q^{k-1} \langle \mu_k,\xi\rangle \, (\mu_1 \ot \cdots \widehat{\mu_k} \cdots \ot \mu_n) \; .$$
When $-1 \leq q < 1$, the operators $\ell_q(\xi)$ are bounded operators on $\cF_q(H)$. One writes $s_q(\xi) = \ell_q(\xi) + \ell_q(\overline{\xi})^*$ and defines $M_q(H_\R) := \{s_q(\xi) \mid \xi \in H_\R\}\dpr$. Then, the vector $\Om$ is cyclic and separating for $M_q(H_\R)$ and the corresponding vector state $\tau$ is a faithful normal tracial state on $M_q(H_\R)$.

When $q=1$, the operators $s_q(\xi)$, $\xi \in H_\R$, are unbounded and closable, with self-adjoint closure, defined on a common domain given by the algebraic direct sum of the subspaces $H^{\ot^n_{-1}}$. We can still define the abelian von Neumann algebra $M_1(H_\R)$ generated by the spectral projections of the closures of $s_q(\xi)$. Of course, there are more convenient ways of defining the Gaussian probability space $(X,\mu)$ of a real Hilbert space $H_\R$ such that $M_1(H_\R) = L^\infty(X,\mu)$.

Note the every orthogonal transformation $v \in \cO(H_\R)$ gives rise to a trace preserving automorphism $\al \in \Aut M_q(H_\R)$ satisfying $\al(s_q(\xi)) = s_q(v \xi)$ for all $\xi \in H_\R$.

For every closed real subspace $K_\R \subset H_\R$, we can naturally identify $M_q(K_\R)$ with a von Neumann subalgebra of $M_q(H_\R)$. In particular, we embed $M = M_q(H_\R) \subset \Mtil = M_q(H_\R \oplus H_\R)$ by identifying $H_\R$ with $H_\R \oplus \{0\}$. As in \cite{Avs11}, we then consider the malleable deformation in the sense of Popa of $M$, given by the automorphisms $\al_t \in \Aut \Mtil$ that are associated with the rotations $R_t \in \cO(H_\R \oplus H_\R)$ defined by $R_t(\xi \oplus \mu) = (\cos t \, \xi - \sin t \, \mu) \oplus (\sin t \, \xi + \cos t \, \mu)$.

For the rest of this section, assume that $-1 < q < 1$ and $\dim H_\R \geq 2$. By \cite[Theorem 2]{Nou03} and \cite[Corollary 1]{Ric04}, $M_q(H_\R)$ is a nonamenable II$_1$ factor. For every $n \in \N$ and $\xi$ in the algebraic tensor product $H^{\ot^n\alg}$, there is a unique element $W_q(\xi) \in M_q(H_\R)$ such that $W_q(\xi) \Om = \xi$. We denote by $\|\cdot\|_2$ the $2$-norm on $M_q(H_\R)$ given by the trace. By definition, $\|W_q(\xi)\|_2 = \|\xi\|_q$, where $\|\cdot\|_q$ denotes $q$-deformed norm of $H^{\ot^n_q}$, for every $\xi \in H^{\ot^n\alg}$.

The following result is mentioned as \cite[Lemma 4.4]{Avs11}, but the proof in \cite{Avs11} is incomplete. Using \cite[Proposition 4.9]{CIW19}, one can give a complete proof. This has been done in detail in \cite[Proposition 6.6]{Wil20}.

\begin{proposition}[{Lemma 4.4 in \cite{Avs11}, Proposition 4.9 in \cite{CIW19} and Proposition 6.6 in \cite{Wil20}}]\label{prop.good-estimate-q-Gaussian}
Let $H_\R$ and $K_\R$ be real Hilbert spaces and put $L_\R = H_\R \oplus K_\R$. Consider the von Neumann algebra $M_q(L_\R)$ with von Neumann subalgebra $M_q(H_\R)$. Assume that $-1 < q < 1$.

For every $1 \leq k \leq n$, define $F^n_k \subset L^{\ot^n}\alg$ as the linear span of the vectors of the form $\xi_1 \ot \cdots \ot \xi_n$ where all $\xi_i$ belong to either $H_\R \oplus \{0\}$, or $\{0\} \oplus K_\R$, with at least $k$ of the vectors $\xi_i$ belonging to $\{0\} \oplus K_\R$.

Let $1 \leq k_1 \leq n_1$ and $1 \leq k_2 \leq n_2$. Let $x \in W_q(F^{n_1}_{k_1})$ and $y \in W_q(F^{n_2}_{k_2})$. Define the normal completely bounded map
$$\Phi_{x,y} : M_q(H_\R) \to M_q(H_\R) : \Phi_{x,y}(a) = E_{M_q(H_\R)}(x a y) \; .$$
\begin{enumlist}
\item We have $\Phi_{x,y}\bigl(W_q(H^{\ot^n}\alg)\bigr) \subset \lspan \bigl\{ W_q(H^{\ot^r\alg}) \bigm| |r-n| \leq n_1+n_2\bigr\}$.
\item Put $k = \min\{k_1,k_2\}$. There exists a $C_{x,y} > 0$ such that $\|\Phi_{x,y}(a)\|_2 \leq C_{x,y} \, |q|^{kn} \, \|a\|_2$ for all $n \in \N$ and $a \in W_q(H^{\ot^n\alg})$.
\end{enumlist}
\end{proposition}

In the context of Proposition \ref{prop.good-estimate-q-Gaussian}, write $M = M_q(H_\R)$ and $\Mtil = M_q(L_\R)$. Consider the $M$-$M$-bimodule $L^2(\Mtil \ominus M)$. Since the linear span of $F^n_1$, $n \geq 1$, is a dense subspace of $L^2(\Mtil \ominus M)$, Proposition \ref{prop.good-estimate-q-Gaussian} thus provides a dense subspace $V \subset L^2(\Mtil \ominus M)$ of left and right bounded vectors and, for all $\xi,\mu \in V$, a constant $C_{\xi,\mu} > 0$ such that the associated completely bounded map $\Phi_{\xi,\eta} : M \to M$ has the properties stated in the proposition.

This does not imply that the $M$-$M$-bimodule $L^2(\Mtil \ominus M)$ is weakly contained in the coarse $M$-$M$-bimodule. We nevertheless have the following result, as in \cite{Avs11}. Let $H_{0,\R} \subset H_\R$ be any finite dimensional real subspace. Take an integer $\kappa \geq 1$ such that $|q|^{2 \kappa} \dim H_0 < 1$. Define $M_0 = M_q(H_{0,\R})$. We claim that the $M_0$-$M_0$-bimodule $L^2(\Mtil \ominus M)^{\ot^\kappa_M}$ is coarse. Indeed, take coefficient maps $\Phi_{\xi,\eta} : M_0 \to M_0$ corresponding to $\xi = \xi_1 \ot_M \cdots \ot_M \xi_\kappa$ and $\eta = \eta_1 \ot_M \cdots \ot_M \eta_\kappa$ with $\xi_i,\eta_j \in V$. We find $C_{\xi,\eta}$ such that $\|\Phi_{\xi,\eta}(a)\|_2 \leq C_{\xi,\eta} \, |q|^{\kappa n} \, \|a\|_2$ for all $a \in W_q(H_0^{\ot^n})$. Since $H_0^{\ot^n}$ has dimension $(\dim H_0)^n$, it follows that each $\Phi_{\xi,\eta}$ defines a Hilbert-Schmidt operator on $L^2(M_0) = \cF_q(H_0)$. This implies the coarseness of the bimodule.

\subsection{Relative amenability}

We also make use of the notion of \emph{relative amenability} introduced in \cite[Section 2.2]{OP07}. For a tracial von Neumann algebra $(M,\tau)$ and von Neumann subalgebras $A \subset p(M_n(\C) \ot M)p$ and $B \subset M$, one says that $A$ is amenable relative to $B$ inside $M$ if the von Neumann algebra $p (M_n(\C) \ot \langle M,e_B \rangle) p$ given by Jones's basic construction admits a positive functional that is $A$-central and that is equal to $\Tr \ot \tau$ when restricted to $p(M_n(\C) \ot M)p$. One says that $A$ is strongly non-amenable relative to $B$ inside $M$ if for every non-zero projection $p'\in A'\cap p(M_n(\mathbb C)\otimes \langle M, e_B \rangle)p$, $Ap'$ is not amenable relative to $B$ inside $M$.

The following result provides an abstract variant of the essence of \cite[Proposition 2.7]{PV11} and \cite[Theorem 5.6]{BCC24}. The result gives a sufficient condition to deduce from amenability relative to $B_1$ and amenability relative to $B_2$, amenability relative to $B_1 \cap B_2$.

\begin{proposition}\label{prop.rel-amen-intersection}
Let $(M,\tau)$ be a tracial von Neumann algebra with von Neumann subalgebras $B_0,B_1,B_2 \subset M$. Let $\bim{B_0}{K}{M}$ be a Hilbert bimodule such that
$$\bim{M}{\bigl(L^2(M) \ot_{B_1} L^2(M) \ot_{B_2} L^2(M)\bigr)}{M} \quad\text{is weakly contained in}\quad \bim{M}{\bigl(L^2(M) \ot_{B_0} K\bigr)}{M} \; .$$
If $p \in M$ is a projection and $A \subset pMp$ a von Neumann subalgebra that is amenable relative to $B_1$ and amenable relative to $B_2$ inside $M$, then $A$ is amenable relative to $B_0$ inside $M$.
\end{proposition}
\begin{proof}
We have the following chain of weak containments of bimodules, resp.\ using that $A$ is amenable relative to $B_1$, relative to $B_2$ and the given weak containment.
\begin{align*}
\bim{M}{L^2(Mp)}{A} &\prec \bim{M}{\bigl( L^2(M) \ot_{B_1} L^2(Mp) \bigr)}{A} \prec \bim{M}{\bigl( L^2(M) \ot_{B_1} L^2(M) \ot_{B_2} L^2(Mp) \bigr)}{A}\\
&\prec \bim{M}{\bigl( L^2(M) \ot_{B_0} K p \bigr)}{A} \; .
\end{align*}
It then follows from \cite[Proposition 2.4(5)]{PV11} that $A$ is amenable relative to $B_0$.
\end{proof}

We also prove the following approximation result. It roughly says that relative amenability can already be detected by an increasing and generating net of von Neumann subalgebras. We will apply this technical result in two different settings: right away in Corollary \ref{cor.rel-amen-in-the-limit} and later on, in the proof of Theorem \ref{thm.stably-solid-and-q-Gaussian}.

\begin{proposition}\label{prop.rel-amen-approximation}
Let $(M,\tau)$ be a tracial von Neumann algebra and $p \in M$ a projection. Let $A \subset pMp$ and $B \subset M$ be von Neumann subalgebras. Let $(M_i)_{i \in I}$ be an increasing net of von Neumann subalgebras of $M$ such that $\bigcup_{i \in I} M_i$ is weakly dense in $M$ and $B \subset M_i$ for all $i \in I$.

Assume that for every $i \in I$, we have a Hilbert $M$-$M$-bimodule $\bim{M}{(\cH_i)}{M}$ with the following properties.
\begin{enumlist}
\item For every $i \in I$, there exists a net $(\xi_j)_{j \in J}$ of vectors in $\cH_i$ such that
$$\lim_{j \in J} \langle x \xi_j,\xi_j\rangle = \tau(pxp) = \lim_{j \in J} \langle \xi_j x , \xi_j \rangle \quad\text{and}\quad \lim_{j \in J} \|a \xi_j - \xi_j a \| = 0 \quad\text{for all $x \in M$, $a \in A$.}$$
\item For every $i \in I$, the bimodule $\bim{M_i}{(\cH_i)}{M_i}$ is weakly contained in $\bim{M_i}{L^2(M_i) \ot_B L^2(M_i)}{M_i}$.
\end{enumlist}
Then $A$ is amenable relative to $B$ inside $M$.
\end{proposition}
\begin{proof}
Fix $i \in I$. Choose a net $(\xi_j)_{j \in J}$ in $\cH_i$ as in (i). Then we get in particular that
\begin{align*}
& \lim_{j \in J} \langle x \xi_j,\xi_j\rangle = \tau(pxp) = \lim_{j \in J} \langle \xi_j x , \xi_j \rangle \quad\text{for all $x \in M_i$,} \\
& \limsup_{j \in J} \|E_{M_i}(a) \xi_j - \xi_j E_{M_i}(a) \| \leq 2 \|a - E_{M_i}(a)\|_2 \quad\text{for all $a \in A$.}
\end{align*}
Since by (ii), the bimodule $\bim{M_i}{(\cH_i)}{M_i}$ is weakly contained in $\bim{M_i}{L^2(M_i) \ot_B L^2(M_i)}{M_i}$, we can choose a net of vectors $(\eta_j)_{j \in J}$ in $(L^2(M_i) \ot_B L^2(M_i))^{\oplus \infty}$ such that also
\begin{align*}
& \lim_{j \in J} \langle x \eta_j,\eta_j\rangle = \tau(pxp) = \lim_{j \in J} \langle \eta_j x , \eta_j \rangle \quad\text{for all $x \in M_i$,} \\
& \limsup_{j \in J} \|E_{M_i}(a) \eta_j - \eta_j E_{M_i}(a) \| \leq 2 \|a - E_{M_i}(a)\|_2 \quad\text{for all $a \in A$.}
\end{align*}
We now view each $\eta_j$ as a vector in $(L^2(M) \ot_B L^2(M))^{\oplus \infty}$. Then $\langle x \eta_j , \eta_j \rangle = \langle E_{M_i}(x) \eta_j , \eta_j\rangle$ for all $x \in M$ and $j \in J$, and similarly for $\langle \eta_j x , \eta_j \rangle$. Define the normal positive functional $\om_i$ on $M$ by $\om_i(x) = \tau(p E_{M_i}(x) p)$. Note that $\om_i \leq \tau$. We get that
\begin{equation}\label{eq.prop.eta-j}
\begin{split}
& \lim_{j \in J} \langle x \eta_j,\eta_j\rangle = \om_i(x) = \lim_{j \in J} \langle \eta_j x , \eta_j \rangle \quad\text{for all $x \in M$,} \\
& \limsup_{j \in J} \|a \eta_j - \eta_j a \| \leq 4 \|a - E_{M_i}(a)\|_2 \quad\text{for all $a \in A$.}
\end{split}
\end{equation}
Also note that $T \mapsto (T \ot_B 1)^{\oplus \infty}$ provides a normal representation of $\langle M,e_B \rangle$ on $(L^2(M) \ot_B L^2(M))^{\oplus \infty}$. Choose a positive functional $\Om_i \in \langle M,e_B \rangle^*_+$ as a weak$^*$ limit point of the net of functionals $T \mapsto \langle (T \ot_B 1)^{\oplus \infty} \eta_j,\eta_j \rangle$. By \eqref{eq.prop.eta-j}, we get that $\Om_i(x) = \om_i(x)$ for all $x \in M$ and $\|a \Om_i - \Om_i a \| \leq 4 \|a - E_{M_i}(a)\|_2$ for all $a \in A$. Choosing a weak$^*$ limit point of the net $(\Om_i)_{i \in I}$, we find a positive functional $\Om \in \langle M,e_B \rangle^*_+$ such that $\Om(x) = \tau(pxp)$ for all $x \in M$ and $a \Om = \Om a$ for all $a \in A$. This means that $A$ is amenable relative to $B$ inside $M$.
\end{proof}

The following result is a variant of \cite[Lemma 2.6]{IS18} and it is an immediate corollary of Proposition \ref{prop.rel-amen-approximation}. This corollary is also essentially contained in the proof of \cite[Lemma 3.8]{Iso17}, but we use the occasion to correct a mistake occurring in that proof.

\begin{corollary}\label{cor.rel-amen-in-the-limit}
Let $(P,\tau)$ and $(B,\tau)$ be tracial von Neumann algebras, $p \in P \ovt B$ a projection and $Q \subset p (P \ovt B) p$ a von Neumann subalgebra.

Assume that $Q$ is amenable relative to $P \ovt B_i$, where $B_i \subset B$, $i \in I$, is a net of von Neumann subalgebras with the following property: there exists an increasing net $D_i \subset B$, $i \in I$, of von Neumann subalgebras such that $\bigcup_i D_i$ is weakly dense in $B$ and such that $\bim{B_i}{L^2(B)}{D_i}$ is weakly contained in a coarse $B_i$-$D_i$-bimodule.

Then, $Q$ is amenable relative to $P \ot 1$.
\end{corollary}
\begin{proof}
Fix $i \in I$. Define $\cH_i = L^2(P) \ot (L^2(B) \ot_{B_i} L^2(B))$. Since $Q$ is amenable relative to $P \ovt B_i$, we can choose a net $(\xi_j)_{j \in J}$ of vectors in $\cH_i$ such that
$$\lim_{j \in J} \langle x \xi_j,\xi_j\rangle = \tau(pxp) = \lim_{j \in J} \langle \xi_j x , \xi_j \rangle \quad\text{and}\quad \lim_{j \in J} \|a \xi_j - \xi_j a \| = 0 \quad\text{for all $x \in P \ovt B$, $a \in Q$.}$$
Since $\bim{B_i}{L^2(B)}{D_i}$ is weakly contained in a coarse $B_i$-$D_i$-bimodule, also $\bim{D_i}{(L^2(B) \ot_{B_i} L^2(B))}{D_i}$ is weakly contained in a coarse $D_i$-$D_i$-bimodule. It follows that $\cH_i$ is weakly contained in $\bim{P \ovt D_i}{L^2(P \ovt D_i) \ot_P L^2(P \ovt D_i)}{P \ovt D_i}$.

Applying Proposition \ref{prop.rel-amen-approximation} to the net $(P \ovt D_i)_{i \in I}$, it follows that $Q$ is amenable relative to $P \ot 1$.
\end{proof}

We finally cite the following basic properties on relative amenability and intertwining-by-bimodules and gather them in a container result for later reference in the paper.

\begin{proposition}\label{prop.basic-prop}
All capital letters denote tracial von Neumann (sub)algebras and each time, $p$ and $q$ are nonzero projections in the appropriate von Neumann algebras.
\begin{enumlist}
\item\label{prop.basic-prop.1} (\cite[Lemma 2.8]{DHI16}) If $A \subset p(B_0 \ovt B_1 \ovt B_2)p$ satisfies $A \prec^f B_0 \ovt B_1 \ovt 1$ and $A \prec^f B_0 \ovt 1 \ovt B_2$, then $A \prec^f B_0 \ot 1 \ot 1$.
\item\label{prop.basic-prop.2} (\cite[Proposition 2.7]{PV11}) If $A \subset p(B_0 \ovt B_1 \ovt B_2)p$ is amenable relative to $B_0 \ovt B_1 \ovt 1$ and is amenable relative to $B_0 \ovt 1 \ovt B_2$, then $A$ is amenable relative to $B_0 \ot 1 \ot 1$.
\item\label{prop.basic-prop.3} (\cite[Lemma 2.3]{BV12}) Let $\cG_1 \subset \cU(p(B_1 \ovt B_2)p)$ be a subgroup that normalizes $A_2 \subset p(B_1 \ovt B_2)p$. If $\cG_1\dpr \prec^f B_1 \ot 1$ and $A_2 \prec^f B_1 \ot 1$, then $(\cG_1 \cup A_2)\dpr \prec^f B_1 \ot 1$.
\end{enumlist}
Let $\al : A \to p(P \ovt B)p$ be a faithful unital normal $*$-homomorphism. Assume that the bimodule $\bim{\al(A)}{p L^2(P \ovt B)}{P \ovt 1}$ is coarse.
\begin{enumlist}[resume]
\item\label{prop.basic-prop.4} (\cite[Lemma 5.4(i)]{DV24}) If $A$ is diffuse, then $\al(A) \not\prec P \ot 1$.

More generally, if $A_0 \subset q(A_1 \ovt A)q$ and $A_0 \not\prec A_1 \ot 1$, then $(\id \ot \al)(A_0) \not\prec A_1 \ovt P \ovt 1$.

\item\label{prop.basic-prop.5} (\cite[Lemma 5.4(ii)]{DV24}) If $\al(A)$ is amenable relative to $P \ot 1$, then $A$ is amenable.

More generally, if $A_0 \subset q(A_1 \ovt A)q$ and $(\id \ot \al)(A_0)$ is amenable relative to $A_1 \ovt P \ovt 1$, then $A_0$ is amenable relative to $A_1 \ot 1$.
\end{enumlist}
\end{proposition}

\section{\boldmath Stably solid II$_1$ factors}\label{sec.rel-solid}

In \cite{Oza03}, a diffuse tracial von Neumann algebra $(A,\tau)$ is called \emph{solid} if the relative commutant $B' \cap A$ of any diffuse von Neumann subalgebra $B \subset A$ is amenable. Nonamenable solid II$_1$ factors are in particular \emph{prime}: they cannot be decomposed as the tensor product of two II$_1$ factors.

The main theorem of \cite{Oza03} says that the group von Neumann algebra $L(\Gamma)$ of a hyperbolic discrete group $\Gamma$ is solid. In \cite[Proposition 4.1]{Oza04} (see also \cite[Definition 15.1.2]{BO08}), the concept of \emph{biexactness} of a discrete group $\Gamma$ was introduced and it was proven in \cite[Theorem 15.1.5]{BO08} that the group von Neumann algebra $L(\Gamma)$ of a biexact group $\Gamma$ is solid. In \cite{Oza03,Oza04}, numerous families of groups are shown to be biexact: hyperbolic groups, discrete subgroups of connected simple Lie groups of rank one, and wreath product groups $\Lambda \wr \Gamma = \Lambda^{(\Gamma)} \rtimes \Gamma$ with $\Lambda$ amenable and $\Gamma$ biexact.

The group von Neumann algebra $L(\Gamma)$ of a biexact nonamenable icc group $\Gamma$ is not only prime. By \cite{OP03}, these group von Neumann algebras also satisfy a unique prime factorization property, in the sense that their tensor products $L(\Gamma_1) \ovt \cdots \ovt L(\Gamma_k)$ only have one decomposition as a tensor product of $k$ II$_1$ factors, up to natural identifications. To obtain such results, it is quite natural to consider relative commutants $B' \cap P \ovt L(\Gamma)$ when $\Gamma$ is biexact and $B \subset P \ovt L(\Gamma)$ is {\it diffuse relative to $P$}, i.e. $B \not\prec P \ot 1$. This leads us to the following definition.

\begin{definition}\label{def.rel-solid}
We say that a tracial von Neumann algebra $(A,\tau)$ is \emph{stably solid} if for every tracial von Neumann algebra $(P,\tau)$, projection $p \in P \ovt A$ and von Neumann subalgebra $B \subset p (P \ovt A) p$ with $B \not\prec P \ot 1$, we have that $B' \cap p(P \ovt A)p$ is amenable relative to $P \ot 1$.
\end{definition}

For the rest of this section, we prove that large classes of tracial von Neumann algebras are stably solid. In each of these cases, solidity, or even strong solidity, was already proven in the literature. It is rather straightforward to adapt the arguments and obtain stable solidity. For completeness, we include the necessary details here.

By \cite{PV12} (see also \cite[Theorem 3.1(iii)]{DV24}), the group von Neumann algebra $L(\Gamma)$ of any biexact, weakly amenable group $\Gamma$ is stably solid. This includes hyperbolic groups and discrete subgroups of connected simple Lie groups of rank one. We strongly believe that weak amenability is a superfluous assumption. We actually believe that all biexact tracial von Neumann algebras $(A,\tau)$ in the sense of \cite[Definition 6.1]{DP23}, are stably solid.

The converse however does not hold. In Theorem \ref{thm.stably-solid-and-q-Gaussian} below, we prove that the $q$-Gaussian II$_1$ factors are stably solid, as well as several crossed products of the CAR-algebra and CCR-algebra by actions of groups with stably solid group von Neumann algebra. In combination with \cite[Theorem 8.9]{DP23} and the discussion preceding \cite[Lemma 8.5]{DP23}, it then follows that there are numerous stably solid II$_1$ factors that are not biexact.

In Section \ref{sec.q-Gaussian}, we recalled the definition of the $q$-Gaussian von Neumann algebras $M_q(H_\R)$. Whenever $\Gamma$ is a discrete group and $\pi : \Gamma \to \cO(H_\R)$ is an orthogonal representation, we have the associated action $\Gamma \actson^\be M_q(H_\R)$ given by $\be_g(s_q(\xi)) = s_q(\pi(g)\xi)$ for all $g \in \Gamma$ and $\xi \in H_\R$. In \cite[Corollary 8.8 and Theorem 8.11]{DP23}, strong solidity of $M_q(H_\R)$, as well as solidity of $M_1(H_\R) \rtimes_\be \Gamma$ under the appropriate assumptions, were proven. We largely follow their ideas to prove stable solidity.

\begin{theorem}\label{thm.stably-solid-and-q-Gaussian}
Let $-1 \leq q \leq 1$ and let $H_\R$ be a real Hilbert space. Consider the $q$-Gaussian von Neumann algebra $M_q(H_\R)$.
\begin{enumlist}
\item\label{thm.stably-solid-and-q-Gaussian.1} $M_q(H_\R)$ is stably solid.
\item\label{thm.stably-solid-and-q-Gaussian.2} Let $q = \pm 1$ and let $\pi : \Gamma \to \cO(H_\R)$ be an orthogonal representation of a discrete group $\Gamma$, with associated action $\Gamma \actson^\be M_q(H_\R)$. Assume that some tensor power $\pi^{\ot^\kappa}$ is weakly contained in the regular representation, and that $L(\Gamma)$ is stably solid. Then also the crossed product $M_q(H_\R) \rtimes_\be \Gamma$ is stably solid.
\end{enumlist}
\end{theorem}

Before proving Theorem \ref{thm.stably-solid-and-q-Gaussian}, we prove a few lemmas.

\begin{lemma}\label{lem.stably-solid-crossed-product}
Let $(P,\tau)$ be a tracial von Neumann algebra and $\Gamma \actson^\be P$ a trace preserving action of a discrete group $\Gamma$. Assume that $L(\Gamma)$ is stably solid. Let $p \in P \rtimes_\be \Gamma$ be a projection and $B \subset p(P \rtimes_\be \Gamma)p$ a von Neumann subalgebra. If $B \not\prec P$, then the relative commutant $B' \cap p(P \rtimes_\be \Gamma)p$ is amenable relative to $P$.
\end{lemma}

Note that Lemma \ref{lem.stably-solid-crossed-product} for the trivial action $\be$ of $\Gamma$ on $P$ is just the definition of stable solidity of $L(\Gamma)$.

\begin{proof}
Write $N = P \rtimes_\be \Gamma$. Define the normal $*$-homomorphism $\Delta : N \to N \ovt L(\Gamma)$ by $\Delta(a u_g) = a u_g \ot u_g$ for all $a \in P$ and $g \in \Gamma$. Write $q = \Delta(p)$. Since $B \not\prec P$, it follows from \cite[Lemma 2.2]{DV24} that $\Delta(B) \not\prec N \ot 1$ inside $N \ovt L(\Gamma)$. Since $L(\Gamma)$ is stably solid, we get that $\Delta(B)' \cap q(N\ovt L(\Gamma))q$ is amenable relative to $N \ot 1$. In particular, $\Delta(B' \cap p N p)$ is amenable relative to $N \ot 1$. A second application of \cite[Lemma 2.2]{DV24} implies that $B' \cap pNp$ is amenable relative to $P$.
\end{proof}

The next lemma is a relative version of \cite[Lemma 8.5]{DP23}. Given a $q$-Gaussian von Neumann algebra $M_q(H_\R)$, recall from Section \ref{sec.q-Gaussian} the malleable deformation $\al_t \in \Aut M_q(H_\R \oplus H_\R)$ given by the rotations of $H_\R \oplus H_\R$.

\begin{lemma}\label{lem.deformation-vs-intertwining}
Let $-1 \leq q \leq 1$ and let $H_\R$ be a real Hilbert space. Consider the $q$-Gaussian von Neumann algebra $M_q(H_\R)$, viewed as a von Neumann subalgebra of $M_q(H_\R \oplus H_\R)$ by identifying $H_\R$ with $H_\R \oplus \{0\}$. Let $P$ be a tracial von Neumann algebra, $p \in P \ovt M_q(H_\R)$ a projection and $B \subset p(P \ovt M_q(H_\R))p$ a von Neumann subalgebra.

If $B \not\prec P \ot 1$, there exists a nonzero projection $r \in B' \cap p(P \ovt M_q(H_\R))p$ such that for every $t \in (0,\pi/2)$ and every $\eps > 0$, there exists a $u \in \cU(Br)$ such that $\|(\id \ot E_M \circ \al_t)(u)\|_2 < \eps$.
\end{lemma}
\begin{proof}
We write $M = M_q(H_\R)$ and $Q = B' \cap p(P \ovt M)p$. By contraposition, we assume that for every nonzero projection $r \in Q$, there exists a $t \in (0,\pi/2)$ and an $\eps > 0$ such that $\|(\id \ot E_M \circ \al_t)(u)\|_2 \geq \eps$ for all $u \in \cU(Br)$. We have to prove that $B \prec P \ot 1$.

We also write $\Mtil = M_q(H_\R \oplus H_\R)$. Using the notation recalled in Section \ref{sec.q-Gaussian}, we identify $L^2(M)$ with the $q$-Fock space $\cF_q(H)$, and we identify $L^2(\Mtil)$ with $\cF_q(H \oplus H)$.

Applying the assumption to $r=p$, we find $t \in (0,\pi/2)$ and $\eps_0 > 0$ such that
$$\|(\id \ot E_M \circ \al_{t})(u)\|_2 \geq \eps_0 \quad\text{for all $u \in \cU(B)$.}$$

Denote by $\cP_1 : H \oplus H \to H$ the orthogonal projection onto the first summand. Since $\cP_1^{\ot^n}$ commutes with the permutation unitaries $\pi_\si$, it also defines the orthogonal projection of $(H \oplus H)^{\ot_q^n}$ onto its subspace $(H \oplus \{0\}){\ot_q^n}$. So, on the level of $\cF_q(H \oplus H) = L^2(\Mtil)$, the conditional expectation $E_M$ is given by $\bigoplus_{n=0}^\infty \cP_1^{\ot^n}$.

Fix $u \in \cU(B)$. View $u$ as a vector in $L^2(M) = \cF_q(H)$ and write $u$ as the sum of $(u)_n \in H^{\ot_q^n}$. It follows that
$$\eps_0^2 \leq \|(\id \ot E_M \circ \al_{t})(u)\|_2^2 = \sum_{n=0}^\infty (\cos t)^{2 n} \|(u)_n\|_q^2 \leq \sum_{n=0}^\infty (\cos t)^n \|(u)_n\|_q^2 = \tau(u^* (\id \ot \al_{t})(u)) \; .$$
Since this holds for all $u \in \cU(B)$, we get that the unique element of minimal $\|\cdot\|_2$ in the weak closure of the convex hull of $\{u^* (\id \ot \al_{t})(u) \mid u \in \cU(B)\}$ is a nonzero element $v \in P \ovt \Mtil$ satisfying $\|v\| \leq 1$, $v = p v$ and $b v = v (\id \ot \al_{t})(b)$ for all $b \in B$.

Write $\eps = (1/5)(1-\cos t)\|v\|_2$. For every finite dimensional real subspace $H_{0,\R} \subset H_\R$, we have the von Neumann subalgebra $M_q(H_{0,\R} \oplus H_{0,\R})$ of $\Mtil$. They form a directed system of von Neumann subalgebras whose union is dense in $\Mtil$. We can thus choose such a finite dimensional real subspace $H_{0,\R} \subset H_\R$, so that the von Neumann subalgebra $\Mtil_0 := M_q(H_{0,\R} \oplus H_{0,\R})$ of $\Mtil$ satisfies $\|v - (\id \ot E_{\Mtil_0})(v)\|_2 < \eps$. We write $v_0 := (\id \ot E_{\Mtil_0})(v)$. It follows that
\begin{equation}\label{eq.we-will-expect}
\|b v_0 - v_0 (\id \ot \al_{t})(b)\|_2 \leq 2 \eps \quad\text{for all $b \in \cU(B)$.}
\end{equation}
Define $v_1 \in P \ovt \Mtil_0$ by $v_1 := (\id \ot \al_{-t})(v_0)$. Define the von Neumann subalgebra $\Stil \subset \Mtil$ by $\Stil := M_q(H_\R \oplus H_{0,\R})$. By definition, $M \subset \Stil$ and $\Mtil_0 \subset \Stil$. It then follows from \eqref{eq.we-will-expect} that
\begin{equation}\label{eq.already-good}
\|(\id \ot E_{\Stil} \circ \al_{t})(v_1 b)\|_2 \geq \|b v_0\|_2 - 2 \eps \geq \|b v\|_2 - 3\eps = \|v\|_2 - 3 \eps \quad\text{for all $b \in \cU(B)$.}
\end{equation}
Denoting by $\cP_{(H\ominus H_0) \oplus \{0\}}$ and $\cP_{H_0 \oplus H_0}$ the orthogonal projection of $H \oplus H$ onto the indicated closed subspaces, and denoting by $R_t$ the rotation matrix defining $\al_t$, the operator $E_{\Stil} \circ \al_{t} \circ E_{\Stil}$ on $\cF_q(H \oplus H)$ is given by
$$\bigoplus_{n=0}^\infty (R_{t} \cP_{H_0 \oplus H_0} + \cos t \, \cP_{(H\ominus H_0) \oplus \{0\}})^{\ot^n} \; ,$$
so that the operator $a \mapsto (E_{\Stil} \circ \al_{t} \circ E_{\Stil})(a - E_{\Mtil_0}(a))$ is given by
\begin{equation}\label{eq.my-operator}
\bigoplus_{n=0}^\infty (R_{t} \cP_{H_0 \oplus H_0} + \cos t \, \cP_{(H\ominus H_0) \oplus \{0\}})^{\ot^n}(1-\cP_{H_0 \oplus H_0}^{\ot^n}) \; .
\end{equation}
Since the summands of \eqref{eq.my-operator} are operators that commute with the permutation unitaries $\pi_\si$, $\si \in S_n$, the norm of the operator in \eqref{eq.my-operator} in $B(\cF_q(H \oplus H))$ is bounded above by the norm of the same operator on $\cF_0(H \oplus H)$. Since the operators $R_{t} \cP_{H_0 \oplus H_0}$ and $\cP_{(H\ominus H_0) \oplus \{0\}}$ have orthogonal ranges and domains, the norm of the operator in \eqref{eq.my-operator} is bounded by $\cos t$.

Writing $v_1 b$ in \eqref{eq.already-good} as the sum of $v_1 b - E_{\Mtil_0}(v_1 b)$ and $E_{\Mtil_0}(v_1 b)$, it follows from \eqref{eq.already-good} and the $\cos t$ bound of the operator in \eqref{eq.my-operator} that
\begin{align*}
\|(\id \ot E_{\Stil} \circ \al_{t} \circ E_{\Mtil_0})(v_1 b)\|_2 & \geq \|v\|_2 - 3 \eps - (\cos t) \|v_1 b\|_2 \\
& \geq \|v\|_2 - 4 \eps - (\cos t) \|(\id \ot \al_{-t})(v) b\|_2 \\ &= (1-\cos t)\|v\|_2 - 4 \eps = \eps \quad\text{for all $b \in \cU(B)$.}
\end{align*}
Denote $M_0 = M_q(H_{0,\R})$. Since $(\id \ot E_{\Mtil_0})(v_1 b) = v_1 (\id \ot E_{\Mtil_0})(b) = v_1 (\id \ot E_{M_0})(b)$ for all $b \in B$, we conclude that there is no net of unitaries $(u_k)$ in $\cU(B)$ such that $\|(\id \ot E_{M_0})(b_k)\|_2 \to 0$. It follows that $B \prec P \ovt M_0$ inside $P \ovt M$.

Replacing $P$ by $M_m(\C) \ot P$ and $p$ by $e_{11} \ot p$, we find a projection $f \in P \ovt M_0$, a normal unital $*$-homomorphism $\theta : B \to f (P \ovt M_0)f$ and a nonzero partial isometry $w \in f(P \ovt M)p$ such that $w b = \theta(b) w$ for all $b \in B$. Write $r = w^* w$ and note that $r$ is a nonzero projection in $Q$. Also note that $b r = w^* \theta(b) w$ for all $b \in B$.

Applying our assumption to the projection $r$, we find $t_1 \in (0,\pi/2)$ and $\eps_1 > 0$ such that
\begin{equation}\label{eq.tussenstap}
\|(\id \ot E_M \circ \al_{t_1})(u)\|_2 \geq 3 \eps_1 \quad\text{for all $u \in \cU(Br)$.}
\end{equation}
Choose a larger finite dimensional real subspace $H_{0,\R} \subset H_{1,\R} \subset H_\R$, so that the von Neumann algebra $M_1 = M_q(H_{1,\R})$ satisfies $\|w - (\id \ot E_{M_1})(w)\|_2 < \eps_1$. Since $b r = w^* \theta(b) w$ and $\theta(B) \subset P \ovt M_0$, it follows that $\|u - (\id \ot E_{M_1})(u)\|_2 < 2\eps_1$ for all $u \in \cU(Br)$. In combination with \eqref{eq.tussenstap}, it follows that
\begin{equation}\label{eq.almost-final-step}
\|(\id \ot E_M \circ \al_{t_1} \circ E_{M_1})(u)\|_2 \geq \eps_1 \quad\text{for all $u \in \cU(Br)$.}
\end{equation}

Denote by $\cP_1 : H \to H_1$ the orthogonal projection of $H$ onto $H_1$. Similarly as above, as an operator on $\cF_q(H) = L^2(M)$, the map $E_M \circ \al_{t_1} \circ E_{M_1}$ is given by $\bigoplus_{n=0}^\infty (\cos t_1)^n \cP_1^{\ot^n}$, which is compact because $(\cos t_1)^n \to 0$ and $\cP_1$ has finite rank. Then \eqref{eq.almost-final-step} implies that $Br \prec P \ot 1$. A fortiori, $B \prec P \ot 1$ and the lemma is proven.
\end{proof}

For completeness, we also record the following elementary lemma.

\begin{lemma}\label{lem.reduction-of-stable-solidity}
Let $(M,\tau)$ be a tracial von Neumann algebra. Assume that for every tracial von Neumann algebra $(P,\tau)$, nonzero projection $p \in P \ovt M$ and von Neumann subalgebra $B \subset p (P \ovt M) p$ with $B \not\prec P \ot 1$, there exists a nonzero projection $r \in Q := B' \cap p(P \ovt M)p$ such that $rQr$ is amenable relative to $P \ot 1$. Then $M$ is stably solid.
\end{lemma}
\begin{proof}
Take $B \subset p (P \ovt M) p$ with $B \not\prec P \ot 1$, and write $Q = B' \cap p(P \ovt M)p$. We have to prove that $Q$ is amenable relative to $P \ot 1$. Denote by $f \in Q' \cap p(P \ovt M)p$ the largest projection such that $Q f$ is amenable relative to $P \ot 1$. If $f = p$, we are done. We assume that $p - f \neq 0$ and derive a contradiction.

Since $f$ belongs to the center of the normalizer of $Q$, we have in particular that $f$ commutes with $B$. So $f \in \cZ(Q)$. We apply the assumption of the lemma to $B(p-f) \subset (p-f)(P \ovt M)(p-f)$. We thus find a nonzero projection $r \in Q$ such that $r \leq p-f$ and $rQr$ is amenable relative to $P \ot 1$. Denote by $z \in \cZ(Q)$ the central support of the projection $r \in Q$. Since $f \in \cZ(Q)$, we have that $z \leq p-f$. Since $rQr$ is amenable relative to $P \ot 1$, also $Q z$ is amenable relative to $P \ot 1$. Then $Q(z+f)$ is amenable relative to $P \ot 1$, which contradicts the maximality of $f$.
\end{proof}

\begin{proof}[Proof of Theorem \ref{thm.stably-solid-and-q-Gaussian}]
1. Write $M = M_q(H_\R)$. If $q=-1$ or $q=1$, the von Neumann algebra $M$ is amenable and there is nothing to prove. So for the rest of the proof of 1, we assume that $-1 < q < 1$. Fix a tracial von Neumann algebra $P$, a nonzero projection $p \in P \ovt M$ and a von Neumann subalgebra $B \subset p (P \ovt M)p$ such that $B \not\prec P \ot 1$. Denote by $Q = B' \cap p(P \ovt M)p$ its relative commutant. By Lemma \ref{lem.reduction-of-stable-solidity}, it suffices to prove that there exists a nonzero projection $r \in Q$ such that $rQr$ is amenable relative to $P \ot 1$.

By Lemma \ref{lem.deformation-vs-intertwining}, we find a nonzero projection $r \in Q$ and sequences $t_n \in (0,\pi/2)$, $u_n \in \cU(Br)$ such that $t_n \to 0$ and $\|(\id \ot E_M \circ \al_{t_n})(u_n)\|_2 \to 0$.
Define $\eta_n = (\id \ot \al_{t_n})(u_n)$. Since $t_n \to 0$, we get that
\begin{equation}\label{eq.what-we-need}
\lim_{n} \langle a \eta_n , \eta_n \rangle = \tau(rar) = \lim_{n} \langle \eta_n a , \eta_n \rangle \quad\text{and}\quad \lim_{n} \|b \eta_n - \eta_n b\|_2 = 0
\end{equation}
for all $a \in P \ovt M$ and all $b \in rQr$. Define $\zeta_n \in L^2(P) \ot L^2(\Mtil \ominus M)$ by $\zeta_n = \eta_n - (\id \ot E_M)(\eta_n)$. Since $\|(\id \ot E_M)(\eta_n)\|_2 \to 0$, the sequence $(\zeta_n)_n$ satisfies the same properties as in \eqref{eq.what-we-need}.

We choose an increasing net $(H_{i,\R})_{i \in I}$ of finite dimensional real subspaces of $H_\R$ whose union is dense in $H_\R$. We denote by $M_i := M_q(H_{i,\R})$ the corresponding von Neumann subalgebras of $M$. By construction, $\bigcup_{i \in I} M_i$ is dense in $M$.

Fix $i \in I$. Since $H_{i,\R}$ is finite dimensional, it follows from Proposition \ref{prop.good-estimate-q-Gaussian} and the remark following it that we can choose an integer $\kappa_i \geq 1$ such that the $M_i$-$M_i$-bimodule $L^2(\Mtil \ominus M)^{\ot^{\kappa_i}_M}$ is coarse. Write
$$\cH_i := L^2(P) \ot L^2(\Mtil \ominus M)^{\ot^{\kappa_i}_M} = (L^2(P) \ot L^2(\Mtil \ominus M))^{\ot^{\kappa_i}_{P \ovt M}}$$
and define $\xi_n \in \cH_i$ by $\xi_n := \zeta_n \ot_{P \ovt M} \cdots \ot_{P \ovt M} \zeta_n$. For every fixed $i \in I$, the sequence $(\xi_n)_n$ still satisfies the properties in \eqref{eq.what-we-need}. As a $(P \ovt M_i)$-$(P \ovt M_i)$-bimodule, $\cH_i$ is contained in a multiple of $L^2(P \ovt M) \ot_{P \ot 1} L^2(P \ovt M)$. By Proposition \ref{prop.rel-amen-approximation}, we get that $rQr$ is amenable relative to $P \ot 1$.

2. Take $q = \pm 1$. We still write $M = M_q(H_\R)$ and keep the notation from the proof of 1. Write $N = M \rtimes_\be \Gamma$, let $P$ be a tracial von Neumann algebra, $p \in P \ovt N$ a nonzero projection and $B \subset p(P \ovt N)p$ a von Neumann subalgebra. Denote by $Q = B' \cap p(P \ovt N)p$ its relative commutant. By Lemma \ref{lem.reduction-of-stable-solidity}, we may assume that for every nonzero projection $r \in Q$, $rQr$ is not amenable relative to $P \ot 1$, and we have to prove that $B \prec P \ot 1$.

In particular, $Q$ is not amenable relative to $P \ot 1$. Since $M$ is amenable, we also get that $Q$ is not amenable relative to $P \ovt M$. Since $L(\Gamma)$ is stably solid, we can apply Lemma \ref{lem.stably-solid-crossed-product} to the crossed product $P \ovt N = (P \ovt M) \rtimes_{\id \ot \be} \Gamma$. We conclude that $B \prec P \ovt M$. After replacing $P$ by $M_m(\C) \ot P$, and $p$ by $e_{11} \ot p$, we can take a projection $f \in P \ovt M$, a unital normal $*$-homomorphism $\theta : B \to f(P \ovt M)f$ and a nonzero partial isometry $v \in f(P \ovt N)p$ such that $\theta(b) v = v b$ for all $b \in B$. By making $f$ smaller if needed, we may assume that the support of $E_{P \ovt M}(vv^*)$ equals $f$. So if $\theta(B) \prec P \ot 1$ inside $P \ovt M$, it follows that $B \prec P \ot 1$ inside $P \ovt N$, and the theorem is proven.

We thus assume that $\theta(B) \not\prec P \ot 1$ inside $P \ovt M$, and we derive a contradiction. By Lemma \ref{lem.deformation-vs-intertwining}, we find a nonzero projection $r \in \theta(B)' \cap f(P \ovt M)f$ and sequences $t_n \in (0,\pi/2)$, $u_n \in \cU(B)$ such that $t_n \to 0$ and $\|(\id \ot E_M \circ \al_{t_n})(\theta(u_n)r)\|_2 \to 0$.

Using the direct sum representation $\pi \oplus \pi : \Gamma \to \cO(H_\R \oplus H_\R)$, we define the associated action $\Gamma \actson^{\betil} \Mtil$. We denote $\Ntil = \Mtil \rtimes_{\betil} \Gamma$. Since the automorphisms $\al_t$ and $\betil_g$ of $\Mtil$ commute for all $t \in \R$, $g \in \Gamma$, we extend $\al_t$ to automorphisms of $\Ntil$ by acting trivially on $L(\Gamma)$. Define $\xi_n \in L^2(P) \ot L^2(\Mtil \ominus M)$ by
$$\xi_n = (\id \ot \al_{t_n})(\theta(u_n)r) - (\id \ot E_M \circ \al_{t_n})(\theta(u_n)r) \; .$$
We view $\xi_n$ as vectors in $L^2(P) \ot L^2(\Ntil \ominus N)$. Write $Q_1 = \theta(B)' \cap f(P \ovt N)f$. Note that $r \in Q_1$. Take $\kappa \geq 1$ such that $\pi^{\ot^\kappa}$ is weakly contained in the regular representation of $\Gamma$. By \cite[Lemma 3.3]{Bou12} (which is equally applicable in the case $q=-1$, even though only stated for $q=1$), the $N$-$N$-bimodule $L^2(\Ntil \ominus N)^{\ot^\kappa_N}$ is weakly contained in the coarse $N$-$N$-bimodule. Using $\xi_n \ot_{P \ovt N} \cdots \ot_{P \ovt N} \xi_n$ and reasoning as in the proof of 1, we conclude that $r Q_1 r$ is amenable relative to $P \ot 1$.

Since $r$ commutes with $\theta(B)$ and since $r \in f(P \ovt M)f$, we get that $r v \neq 0$ and $\theta(b) rv = rv b$ for all $b \in B$. Denote by $w \in r(P \ovt N)p$ the polar part of $rv$. We still have that $\theta(b) w = w b$ for all $b \in B$. Denote $p_0 = w^* w$, so that $p_0 \in Q$. Since $rQ_1 r$ is amenable relative to $P \ot 1$, we can choose a functional $\Om_1 \in (P \ovt B(L^2(N)))^*_+$ such that $\Om_1(a) = \tau(r a r)$ for all $a \in P \ovt N$, and such that $\Om_1$ is $rQ_1r$-central. Since $w = rw$ and $w Q w^* \subset rQ_1r$, the functional $\Om \in (P \ovt B(L^2(N)))^*_+$ defined by $\Om(T) = \Om_1(w T w^*)$ satisfies $\Om(a) = \tau(p_0 a p_0)$ for all $a \in P \ovt N$, and $\Om$ is $p_0 Q p_0$-central. We have thus found a nonzero projection $p_0 \in Q$ such that $p_0 Q p_0$ is amenable relative to $P \ot 1$, which contradicts our initial assumptions.
\end{proof}

By a minor adaptation of the proof of Theorem \ref{thm.stably-solid-and-q-Gaussian.2}, one also gets the following stable version of \cite[Corollary 4.5]{Oza04}. This time, we leave the proof to the reader.

\begin{proposition}
Let $\Gamma$ be a discrete group for which $L(\Gamma)$ is stably solid. Let $(A_0,\tau_0)$ be an amenable tracial von Neumann algebra. Then the Bernoulli crossed product $(A_0,\tau_0)^\Gamma \rtimes \Gamma$ is stably solid.
\end{proposition}

In the paper \cite{Oza03}, Ozawa introduced the concept of solidity and proved it by combining the Akemann-Ostrand (AO) property with local reflexivity. In \cite[Definition 3.1.1]{Iso12}, a slight strengthening of (AO), coined (AO)$^+$, was introduced. In line with \cite{PV12}, it was then proven in \cite{Iso12} that all factors satisfying (AO)$^+$ and the weak$^*$ completely bounded approximation property (W$^*$CBAP) are strongly solid. The same method can be used as follows to prove stable solidity.

It then follows from \cite[Section 2.4]{Iso12} that the von Neumann algebras $L^\infty(\bG)$ of the Kac type universal orthogonal, resp.\ universal unitary quantum groups are stably solid II$_1$ factors.

\begin{proposition}
Every tracial von Neumann algebra $M$ that satisfies the condition (AO)$^+$ of \cite[Definition 3.1.1]{Iso12} and the W$^*$CBAP is stably solid.
\end{proposition}
\begin{proof}
Fix a tracial von Neumann algebra $P$, a nonzero projection $p \in P \ovt M$ and a von Neumann subalgebra $B \subset p(P \ovt M)p$ such that $B \not\prec P \ot 1$. Denote by $Q = B' \cap p(P \ovt M)p$ its relative commutant. We have to prove that $Q$ is amenable relative to $P \ot 1$.

Since $B \not\prec P \ot 1$, we can choose a net of unitaries $(u_i)_{i \in I}$ in $B$ such that $\|(1 \ot T)(u_i)\|_2 \to 0$ for every $T \in \cK(L^2(M))$, where we view $u_i$ as a vector in $L^2(P \ovt M)$.

Choose $\Om_1 \in B(L^2(P \ovt M))^*_+$ as a weak$^*$ limit point of the net of vector functionals $T \mapsto \langle T(u_i),u_i\rangle$. Define the $*$-algebra $D = (P \ovt M) \otalg (P \ovt M)\op$ and the unital $*$-homomorphisms
\begin{align*}
& \Theta : D \to B(L^2(P \ovt M)) : \Theta(a \ot b\op)(\xi) = a \xi b \quad\text{and}\\
& \Psi : D \to B(L^2(P \ovt M \ovt M)) : \Psi(a \ot b\op)(\xi) = a_{12} \xi b_{13} \; .
\end{align*}
Note that $\Om_1(\Theta(a \ot 1)) = \tau(pap) = \Om_1(1 \ot a\op)$ for all $a \in P \ovt M$ and $\Om_1(\Theta(v \ot \vbar)) = \tau(p)$ for all $v \in \cU(Q)$.

Below, we prove the existence of $C > 0$ such that
\begin{equation}\label{eq.goal-with-C}
|\Om_1(\Theta(d))| \leq C \, \|\Psi(d)\| \quad\text{for all $d \in D$.}
\end{equation}
By the Hahn-Banach theorem, we can then choose $\Om_2 \in B(L^2(P \ovt M \ovt M))^*_+$ such that $\Om_2(\Psi(d)) = \Om_1(\Theta(d))$ for all $d \in D$. In particular, $\Om_2(a \ot 1) = \tau(pap)$ for all $a \in P \ovt M$ and $\Om_2(\Psi(v \ot \vbar)) = \tau(p)$ for all $v \in \cU(Q)$. It follows that $\Om_2 \cdot \Psi(v \ot \vbar) = \Om_2 = \Psi(v \ot \vbar) \cdot \Om_2$ for all $v \in \cU(Q)$.

We define $\Om \in (P \ovt B(L^2(M)))^*_+$ by $\Om(T) = \Om_2(T \ot 1)$. Since $P \ovt B(L^2(M)) \ovt 1$ commutes with $\Psi(1 \ot (P \ovt M)\op)$, it follows that $\Om$ is $Q$-central and satisfies $\Om(a) = \tau(pap)$ for all $a \in P \ovt M$. So, $Q$ is amenable relative to $P \ot 1$.

It remains to prove that there exists a $C > 0$ such that \eqref{eq.goal-with-C} holds. Denote by $\lambda : M \to B(L^2(M))$ and $\rho : M\op \to B(L^2(M))$ the left multiplication, resp.\ right multiplication representation. Since $M$ satisfies (AO)$^+$, there exists a weakly dense, locally reflexive C$^*$-subalgebra $M_0 \subset M$ and a unital completely positive map $\theta : M_0 \otmin M_0\op \to B(L^2(M))$ such that for every $a,b \in M_0$, we have that $\theta(a \ot b\op) - \lambda(a)\rho(b\op) \in \cK(L^2(M))$.
We can then define $\Om_0 \in (B(L^2(P)) \otmin M_0 \otmin M_0\op)^*_+$ by $\Om_0(T) = \Om_1((\id \ot \theta)(T))$.

Define the $*$-subalgebra $D_0 \subset D$ by $D_0 = (P \otalg M_0) \otalg (P \otalg M_0)\op$. By definition, we have that $\Psi(D_0) \subset B(L^2(P)) \otmin M_0 \otmin M_0\op$. Since $\|(1 \ot T)(u_i)\|_2 \to 0$ for every $T \in \cK(L^2(M))$, we get that $\Om_1(B(L^2(P)) \otmin \cK(L^2(M))) = \{0\}$. Therefore $\Om_0(\Psi(d)) = \Om_1(\Theta(d))$ for all $d \in D_0$. With $C_0 := \|\Om_0\|$, it follows that $|\Om_1(\Theta(d))| \leq C_0 \, \|\Psi(d)\|$ for all $d \in D_0$.

Since $M$ satisfies the W$^*$CBAP and $M_0$ is locally reflexive, by \cite[Lemma 2.3.1]{Iso12}, we can choose a net of finite rank, normal, completely bounded maps $\vphi_j : M \to M_0$, $j \in J$, and a $\kappa > 0$ such that $\|\vphi_j\|\cb \leq \kappa$ for all $j \in J$, and such that $\|\vphi_j(b) - b \|_2 \to 0$ for all $b \in M$.

Fix $d \in D$. Define $d_j \in D_0$ by $d_j = ((\id \ot \vphi_j) \ot (\id \ot \vphi_j)\op)(d)$. Since $\Om_1(\Theta(a \ot 1)) = \tau(pap)$ for all $a \in P \ovt M$, we get that $\|\Om_1 \cdot \Theta(a \ot 1)\| \leq \|a\|_2$ for all $a \in P \ovt M$. Similarly, $\|\Theta(1 \ot a\op) \cdot \Om_1\| \leq \|a\|_2$ for all $a \in P \ovt M$. It follows that $\Om_1(\Theta(d_j)) \to \Om_1(\Theta(d))$. Viewing $\Psi(D) \subset B(L^2(P)) \ovt M \ovt M\op$, we also have that $\Psi(d_j) = (\id \ot \vphi_j \ot \vphi_j\op)(\Psi(d))$. Therefore,
\begin{align*}
|\Om_1(\Theta(d))| &= \limsup_{j \in J} |\Om_1(\Theta(d_j))| \leq \limsup_{j \in J} C_0 \, \|\Psi(d_j)\| \\
& = \limsup_{j \in J} C_0 \, \|(\id \ot \vphi_j \ot \vphi_j\op)(\Psi(d))\| \leq C_0 \kappa^2 \, \|\Psi(d)\| \; .
\end{align*}
So, \eqref{eq.goal-with-C} is proven and this concludes the proof of the proposition.
\end{proof}

\section{\boldmath W$^*$-correlations of finite von Neumann algebras}\label{sec.corr}

\subsection{General results}

In \cite[Definition 1.4]{IPR19}, the concept of \emph{von Neumann equivalence} between two finite von Neumann algebras $A$ and $B$ was introduced. A variant of this concept was defined in \cite[Definition 5.1]{BV22} and coined \emph{measure equivalence}. In hindsight, we believe that both terminologies are confusing and therefore propose to use the terminology \emph{W$^*$-correlation} instead of measure equivalence, or von Neumann equivalence.

We recall the basic definitions in this section and also clarify the precise relationship between W$^*$-correlated (i.e.\ measure equivalent) von Neumann algebras and the notion of von Neumann equivalence from \cite{IPR19}.

Although the rest of this paper only deals with countably decomposable von Neumann algebras, we formulate our results in complete generality in this section, for later reference. Also recall that a von Neumann algebra admits a faithful normal tracial state if and only if it is finite and countably decomposable. We use the terminology \emph{tracial von Neumann algebra} for a von Neumann algebra with a fixed faithful normal tracial state.

When $A$ is a finite von Neumann algebra, a right Hilbert $A$-module $\rightmod{K}{A}$ is said to be of \emph{finite type} if $\rightmod{K}{A}$ is isomorphic to a closed right $A$-submodule of a finite direct sum of copies of the standard representation on $H_A$. This means that there exists an $n \in \N$ and a projection $p \in M_n(\C) \ot A$ such that $\rightmod{K}{A} \cong \rightmod{p(\C^n \ot H_A)}{A}$. When $A$ is countably decomposable, in particular when $A$ has separable predual, this is equivalent to saying that $\rightmod{K}{A}$ is \emph{finitely generated} as a right Hilbert $A$-module. One similarly considers the notion of a finite type left Hilbert $A$-module.

Recall that a Hilbert $A$-$P$-bimodule $\bim{A}{H}{P}$ is called \emph{coarse} if the canonically associated $*$-homomorphism $A \otalg P\op \to B(H)$ extends to a normal $*$-homomorphism $A \ovt P\op \to B(H)$. Note that $\bim{A}{H}{P}$ is coarse if and only if $\bim{A}{H}{P}$ is isomorphic with a closed $A$-$P$-subbimodule of a direct sum of copies of the standard $A$-$P$-bimodule $H_A \ot H_P$.

The following definition was given in \cite{BV22}, but using the terminology \emph{measure equivalence} instead of \emph{W$^*$-correlation}.

\begin{definition}[{Definition 5.1 in \cite{BV22}}]\label{def.Wstar-corr}
Let $A$, $B$ and $P$ be finite von Neumann algebras. We say that an $A$-$(P \ovt B)$-bimodule $H$ is a \emph{$P$-embedding of $A$ into $B$} if $H$ is of finite type as a right Hilbert $(P \ovt B)$-module and coarse as an $A$-$P$-bimodule.

By coarseness, we can view $H$ as a $(P\op \ovt A)$-$B$-bimodule. If $H$ is moreover of finite type as a left Hilbert $(P\op \ovt A)$-module, we say that $H$ is a \emph{$P$-equivalence between $A$ and $B$.}

We say that two finite von Neumann algebras $A$ and $B$ are \emph{W$^*$-correlated} if there exists a finite von Neumann algebra $P$ and a $P$-equivalence between $A$ and $B$ that is faithful as a left $A$-module and as a right $B$-module.
\end{definition}

Note that by \cite[Lemma 5.2]{BV22}, being W$^*$-correlated defines an equivalence relation on the class of finite von Neumann algebras. Also note the following when comparing Definition \ref{def.Wstar-corr} to \cite[Definition 5.1]{BV22}: since a $P$-equivalence $H$ is assumed to be a right $(P \ovt B)$-module, it is automatically coarse when viewed as a $P\op$-$B$-bimodule.

\begin{remark}\label{rem.charact-Wstar-corr}
Note that after replacing $P$ by $M_n(\C) \ot P$ for large enough $n$, every $P$-equivalence between finite von Neumann algebras $A$ and $B$ is given by faithful unital normal $*$-homomorphisms $\al : A \to p(P \ovt B)p$ and $\be : B \to q(P\op \ovt A)q$ and a unitary $V : p (H_P \ot H_B) \to (H_{P\op} \ot H_A) q$ satisfying
\begin{equation}\label{eq.charact-Wstar-corr}
V(\al(a) \xi (d \ot b)) = (d\op \ot a)(V(\xi))\be(b) \quad\text{for all $a \in A$, $b \in B$, $d \in P$, $\xi \in p (H_P \ot H_B)$.}
\end{equation}
\end{remark}

The following example is essentially contained in \cite[Proposition 5.10]{BV22}.

\begin{example}\label{ex.twisted-group-vNalg}
Let $G$ be a discrete group and $\om \in Z^2(G,\T)$ a $2$-cocycle. Consider the group von Neumann algebra $L(G)$ with canonical unitaries $u_g \in L(G)$ and the twisted group von Neumann algebra $L_\om(G)$ with canonical unitaries $u^\om_g \in L_\om(G)$. Define the faithful normal $*$-homomorphism $\al : L_\om(G) \to L_\om(G) \ovt L(G) : \al(u^\om_g) = u^\om_g \ot u_g$ for all $g \in G$. With $P = L_\om(G)$, we get a $P$-equivalence between $L(G)$ and $L_\om(G)$.

Taking $G = \Z^2$ and the $2$-cocycle $\om((x_1,y_1),(x_2,y_2)) = \exp(2 \pi i \theta x_1 y_2)$, where $\theta \in \R \setminus \Q$, we find that $L(G)$ is diffuse abelian, while $L_\om(G) \cong R$, the hyperfinite II$_1$ factor. So, $L^\infty([0,1])$ is W$^*$-correlated to $R$. We discuss this in more detail in Section \ref{sec.Wstar-corr-amenable}.
\end{example}

The following elementary result shows that W$^*$-correlations behave well under direct sums and taking corners. The result already gives a first flavor of the flexibility that W$^*$-correlations offer, and shows that being W$^*$-correlated is a fully local property.

\begin{proposition}\label{prop.direct-sums-etc}
Let $A$ and $B$ be finite von Neumann algebras, $I$ a set and $p_i \in A$, $q_i \in B$ projections with central supports $z(p_i)$, resp.\ $z(q_i)$. Assume that $\bigvee_{i \in I} z(p_i) = 1$ and $\bigvee_{i \in I} z(q_i) = 1$. If $p_i A p_i$ is W$^*$-correlated to $q_i B q_i$ for every $i \in I$, then $A$ is W$^*$-correlated to $B$.

In particular, the following results hold.
\begin{enumlist}
\item\label{prop.direct-sums-etc.1} If $I$ is a set and if for every $i \in I$, we have W$^*$-correlated finite von Neumann algebras $A_i$, $B_i$, then $\bigoplus_{i \in I} A_i$ is W$^*$-correlated to $\bigoplus_{i \in I} B_i$.
\item\label{prop.direct-sums-etc.2} If $A$ is a finite von Neumann algebra and $p \in A$ is a projection with central support $1$, then $A$ is W$^*$-correlated to $p A p$.
\item\label{prop.direct-sums-etc.3} If $I$ is a set and if for every $i \in I$, we have a finite von Neumann algebra $A_i$ such that all $A_i$ are W$^*$-correlated to the same finite von Neumann algebra $B$, then also $\bigoplus_{i \in I} A_i$ is W$^*$-correlated to $B$.
\end{enumlist}
\end{proposition}
\begin{proof}
We first prove the particular consequence (i). By Remark \ref{rem.charact-Wstar-corr}, take finite von Neumann algebras $P_i$ and faithful unital normal $*$-homomorphisms $\al_i : A_i \to p_i(P_i \ovt B_i)p_i$ and $\be_i : B_i \to q_i (P_i\op \ovt A_i)q_i$ such that there exists a unitary $V_i$ satisfying \eqref{eq.charact-Wstar-corr}. Denote by $A$, $B$ and $P$ the direct sums of $A_i$, $B_i$ and $P_i$. Defining $p$ as the direct sum of the $p_i$, we may view $p$ as a projection in $P \ovt A$. We similarly define $q \in P\op \ovt B$. Then the direct sum of the $\al_i$ defines a faithful unital normal $*$-homomorphisms $\al : A \to p(P \ovt B)p$. We similarly define $\be : B \to q (P\op \ovt A)q$. The direct sum of the $V_i$ provides a unitary $V$ satisfying \eqref{eq.charact-Wstar-corr}. So, $A$ is W$^*$-correlated to $B$.

We next prove the following statement: if $A$ is a finite von Neumann algebra and $(z_i)_{i \in I}$ is a family of central projections with $\bigvee_{i \in I} z_i = 1$, then $A$ is W$^*$-correlated to $B := \bigoplus_{i \in I} A z_i$. To prove this, take the standard representation $H_A$ of $A$. Consider the Hilbert space $H = \bigoplus_{i \in I} H_A \cdot z_i$. With the natural left action of $A$, right action of $B$ and right action of the finite von Neumann algebra $P := \ell^\infty(I)$, we get a $P$-equivalence between $A$ and $B$.

Now note that by the previous two paragraphs, the main statement of the proposition holds if $p_i \in A$ and $q_i \in B$ are central projections. Indeed, by the previous paragraph $A$ is W$^*$-correlated to $\bigoplus_{i \in I} A p_i$ and $B$ is W$^*$-correlated to $\bigoplus_{i \in I} B q_i$. If $A p_i$ is W$^*$-correlated to $B q_i$ for all $i \in I$, it then follows from the first paragraph that $A$ is W$^*$-correlated to $B$.

We then prove a special case of (ii). Let $A$ be a finite von Neumann algebra and denote by $E : A \to \cZ(A)$ the unique center-valued trace. Let $p \in A$ be a projection and $\eps > 0$ such that $E(p) \geq \eps 1$. We prove that $A$ is W$^*$-correlated to $p A p$. Take an integer $n \geq 1$ such that $1/n \leq \eps$. It follows that in the finite von Neumann algebra $M_n(\C) \ot A$, the center-valued trace of $1 \ot p$ is greater or equal than $1$, while the center-valued trace of $e_{11} \ot 1$ is equal to $1$. We thus find an isometry $V \in \C^n \ot A$ with $q := VV^*$ satisfying $q \leq 1 \ot p$ and $V^* V = 1$. Let $H_A$ be the standard representation of $A$. Then $H_A \cdot p$ is an $A$-$pAp$-bimodule. To see that we get a W$^*$-correlation, with $P = \C$, it suffices to note that left multiplication by $V$ indeed embeds the right $pAp$-module $H_A \cdot p$ in an $n$-fold direct sum of $p \cdot H_A \cdot p = H_{p A p}$.

We can now prove (ii) in general. For every integer $k \geq 1$, denote by $z_k \in \cZ(A)$ the spectral projection of $E(p)$ that corresponds to the interval $[1/k,1]$. Define $p_k \in \cZ(pAp)$ by $p_k = z_k p$. By the previous paragraph, $A z_k$ is W$^*$-correlated to $p_k A p_k$ for all $k \geq 1$. Since the central support of $p$ equals $1$, we have that $z_k \to 1$. So, $\bigvee_k z_k = 1$ and $\bigvee_k p_k = p$. So by the special case of the main statement for central projections, it follows that $A$ is W$^*$-correlated to $pAp$.

Then the main statement of the proposition follows. By (ii), we find that $A z(p_i)$ is W$^*$-correlated to $p_i A p_i$, and that $B z(q_i)$ is W$^*$-correlated to $q_i B q_i$ for all $i \in I$. So, $A z(p_i)$ is W$^*$-correlated to $B z(q_i)$ for all $i \in I$. Since the special case for central projections was already proven, it follows that $A$ is W$^*$-correlated to $B$.

Finally note that (iii) is indeed an immediate consequence. Write $A = \bigoplus_{i \in I} A_i$ and denote by $z_i$ the central projection corresponding to the direct summand $A_i$. Since $A z_i$ is W$^*$-correlated to $B$ and $\bigvee_{i \in I} z_i = 1$, also $A$ is W$^*$-correlated to $B$.
\end{proof}

Also stability of W$^*$-correlations under tensor products is easy to check.

\begin{proposition}\label{prop.Wstar-corr-tensor}
If for $i \in \{1,2\}$, $A_i$ and $B_i$ are W$^*$-correlated finite von Neumann algebras, then also $A_1 \ovt A_2$ is W$^*$-correlated to $B_1 \ovt B_2$.
\end{proposition}
\begin{proof}
If $P_i$ are finite von Neumann algebras and $H_i$ is a $P_i$-equivalence between $A_i$ and $B_i$, one checks that $H_1 \ot H_2$ becomes a $(P_1 \ovt P_2)$-equivalence between $A_1 \ovt A_2$ and $B_1 \ovt B_2$.
\end{proof}

Most of the time, we only consider finite von Neumann algebras $A$ and $B$ that are countably decomposable with a fixed normal faithful tracial state, resp.\ with separable predual. We prove that in that case, to define W$^*$-correlations, we may assume that the auxiliary finite von Neumann algebra $P$ satisfies the same restrictions.

\begin{proposition}\label{prop.countably-decomp-separable}
Let $A$ and $B$ be W$^*$-correlated finite von Neumann algebras.
\begin{enumlist}
\item\label{prop.countably-decomp-separable.1} If $A$ and $B$ are countably decomposable, there exists a II$_1$ factor $P$ and a $P$-equivalence between $A$ and $B$.
\item\label{prop.countably-decomp-separable.2} If $A$ and $B$ have separable predual, there exists a II$_1$ factor $P$ with separable predual and a $P$-equivalence $H$ between $A$ and $B$ such that $H$ is a separable Hilbert space.
\end{enumlist}
\end{proposition}
\begin{proof}
(i) Choose faithful normal tracial states $\tau_A$ and $\tau_B$ on $A$, resp.\ $B$. Take a finite von Neumann algebra $Q$ and a $Q$-equivalence $H$ between $A$ and $B$. Choose an increasing net of central projections $(z_i)_{i \in I}$ in $Q$ such that $Q z_i$ is countably decomposable for every $i \in I$ and $z_i \to 1$. Define $H_i = H \cdot (z_i \ot 1)$. Note that $H_i$ is an $A$-$(Q \ovt B)$-subbimodule of $H$. Define $p_i \in \cZ(A)$ as the support of the left action of $A$ on $H_i$. Define $q_i \in \cZ(B)$ as the support of the right action of $B$ on $H_i$. Since $z_i \to 1$ and since the left action of $A$ on $H$ and the right action of $B$ on $H$ are faithful, we get that $p_i \to 1$ and $q_i \to 1$ strongly.

Choose $i_1 < i_2 < \cdots$ such that $\tau_A(p_{i_k}) \to 1$ and $\tau_B(q_{i_k}) \to 1$. Then also $p_{i_k} \to 1$ and $q_{i_k} \to 1$ strongly. Define $z = \bigvee_k z_{i_k}$. It follows that $Qz$ is countably decomposable and that $H \cdot (z \ot 1)$ is a $Qz$-equivalence between $A$ and $B$.

After replacing $Qz$ by a matrix algebra over $Qz$ and using Remark \ref{rem.charact-Wstar-corr}, we find a countably decomposable finite von Neumann algebra $D$, projections $p \in D \ovt B$, $q \in D \ovt A\op$, normal unital faithful $*$-homomorphisms $\al : A \to p(D \ovt B)p$ and $\be\op : B\op \to q(D \ovt A\op)q$, and an element $W \in D \ovt B(H_B,H_{A\op})$ satisfying
\begin{equation}\label{eq.my-good-formulas}
\begin{split}
& W \, (\id \ot \lambda_B)\al(a) \, (1 \ot \rho_B(b)) = (1\ot \rho_{A\op}(a\op)) \, (\id \ot \lambda_{A\op})\be\op(b\op)\, W \quad\text{and}\\
& W^* W = (\id \ot \lambda_B)(p) \; , \; W W^* =(\id \ot \lambda_{A\op})(q) \; .
\end{split}
\end{equation}
Here, given any standard representation of a finite von Neumann algebra $C$ on $H_C$, we denote by $\lambda_C : C \to B(H_C)$ the canonical $*$-homomorphism given by the left action and by $\rho_C : C \to B(H_C)$ the canonical $*$-antihomomorphism given by the right action.

Choose a faithful normal tracial state $\tau_D$ on $D$. Then define $(P,\tau) = (D,\tau_D) \ast (C,\tau_C)$, where $C$ is a diffuse abelian von Neumann algebra with faithful normal tracial state $\tau_C$. Then, $P$ is a II$_1$ factor and we can view $\al$ and $\be\op$ as taking values in $p (P \ovt B)p$, resp.\ $q(P \ovt A\op)q$. We can also view $W$ as an element of $P \ovt B(H_B,H_{A\op})$. In this way, we find a $P$-equivalence between $A$ and $B$.

(ii) As in the previous paragraphs, since $A$ and $B$ are W$^*$-correlated, we can take a finite von Neumann algebra $D$, projections $p$ and $q$, $*$-homomorphisms $\al$ and $\be\op$ and an element $W$ satisfying \eqref{eq.my-good-formulas}. Since $A$ and $B$ have separable predual, we can choose a von Neumann subalgebra $D_0 \subset D$ such that $D_0$ has separable predual, $p \in D_0 \ovt B$, $q \in D_0 \ovt A\op$, $\al(A) \subset D_0 \ovt B$, $\be\op(B\op) \subset D_0 \ovt A\op$, and $W \in D_0 \ovt B(H_B,H_{A\op})$. Next making a free product construction as in the previous paragraph, we find a II$_1$ factor $P$ with separable predual and a $P$-equivalence between $A$ and $B$.
\end{proof}

\subsection{\boldmath W$^*$-correlations versus von Neumann equivalence}\label{sec.Wstar-corr-vs-vN-equiv}

Recall the following definition of \emph{von Neumann equivalence} from \cite{IPR19}. As above, we denote by $H_C$ the standard Hilbert space of a finite von Neumann algebra $C$.

\begin{definition}[{Definition 1.4 in \cite{IPR19}}]\label{def.vNequiv}
Two finite von Neumann algebras $A$ and $B$ are called \emph{von Neumann equivalent} if there exists a semifinite von Neumann algebra $\cM$ and unital embeddings $A,B\op \subset \cM$ such that the following holds: $A$ commutes with $B\op$, and there exist intermediate embeddings $A \subset B(H_A) \subset \cM$ and $B\op \subset B(H_B) \subset \cM$ such that the finite rank projections in $B(H_A)$ and $B(H_B)$ are finite projections of $\cM$.
\end{definition}

We prove that if $A$ and $B$ are von Neumann equivalent (in the sense of Definition \ref{def.vNequiv}), then $A$ and $B$ are W$^*$-correlated (in the sense of Definition \ref{def.Wstar-corr}). The converse holds if $A$ and $B$ are factors, but does not hold in general (see Example \ref{ex.counterex-vNequiv-Wstar-corr}).

\begin{proposition}\label{prop.Wstar-corr-vs-vN-equiv}
Let $A$ and $B$ be finite von Neumann algebras.
\begin{enumlist}
\item If $A$ and $B$ are von Neumann equivalent, then $A$ and $B$ are W$^*$-correlated.
\item If $A$ and $B$ are finite factors and if $A$ and $B$ are W$^*$-correlated, then $A$ and $B$ are von Neumann equivalent.
\end{enumlist}
\end{proposition}
\begin{proof}
(i) Take $\cM$ as in Definition \ref{def.vNequiv}. Take $A \subset B(H_A) \subset \cM$ and $B\op \subset B(H_B) \subset \cM$. Fix a minimal projection $e \in B(H_A)$, which we also view as a finite projection in $\cM$. Also fix a minimal projection $f \in B(H_B)$, which we also view as a finite projection in $\cM$.

We start by proving that $A$ and $B$ are W$^*$-correlated under the extra assumption that there exists an integer $n \geq 1$ such that in $M_n(\C) \ot \cM$, the projection $e_{11} \ot e$ is subequivalent to the projection $1 \ot f$.

Write $P = f \cM f$ and note that $P$ is a finite von Neumann algebra. Choose a standard representation of $\cM$ on the Hilbert space $H_{\cM}$. Since $B(H_B) \subset \cM$, we get a canonical identification $\cM = P \ovt B(H_B)$. This leads to a canonical identification of $H_{\cM} = H_P \ot H_B \ot \overline{H_B}$, and thus of $H_{\cM} \cdot f = H_P \ot H_B$. Under this last identification, the left action of $\cM$ on $H_{\cM} \cdot f$ corresponds to the natural left action of $P \ovt B(H_B)$ on $H_P \ot H_B$, and the right action of $f \cM f$ on $H_{\cM} \cdot f$ corresponds to the natural right action of $P \ot 1$ on $H_P \ot H_B$.

Since $A \subset \cM$ commutes with $B\op \subset \cM$, the left action of $A$ on $H_{\cM} \cdot f$ translates into a faithful normal unital $*$-homomorphism $\al : A \to P \ovt B \subset P \ovt B(H_B)$. In this way, $H_P \ot H_B$ becomes a Hilbert $A$-$(P \ovt B)$-bimodule. By construction, $H_P \ot H_B$ is of finite type as a right Hilbert $(P \ovt B)$-module. It thus remains to prove that $H_P \ot H_B$ is coarse as an $A$-$P$-bimodule and then of finite type as a left Hilbert $(P\op \ovt A)$-module.

Since $e \in B(H_A)$ is a minimal projection and $B(H_A) \subset \cM$, we identify $H_P \ot H_B = H_{\cM} \cdot f$ with $e \cdot H_{\cM} \cdot f \ot H_A$, in such a way that the left action of $A$ corresponds to the left action of $1 \ot A$ on $e \cdot H_{\cM} \cdot f \ot H_A$, and the right action of $P = f \cM f$ corresponds to the right action of $P \ot 1$ on $e \cdot H_{\cM} \cdot f \ot H_A$. Coarseness thus follows and it suffices to prove that $e \cdot H_{\cM} \cdot f$ is of finite type as a right Hilbert $P$-module.

By assumption, we find $V \in \C^n \ot \cM$ such that $V^* V = e$ and $VV^* \leq 1 \ot f$. Therefore, left multiplication by $V$ embeds the right Hilbert $P$-module $e \cdot H_{\cM} \cdot f$ into $\C^n \ot f \cdot H_{\cM} \cdot f = \C^n \ot H_P$, so that it is of finite type. This concludes the proof of (i) under the extra assumption on $e$ and $f$.

In general, $e$ and $f$ are finite projections in $\cM$ with central support equal to $1$. Define $r = e \vee f$ and write $Q = r \cM r$. Then $Q$ is a finite von Neumann algebra. Denote by $E : Q \to \cZ(Q)$ its unique center-valued trace. As projection in $Q$, the central support of $f$ equals $1$. Choose an integer $n \geq 1$ large enough such that the spectral projection $c \in \cZ(Q)$ of $E(f)$ corresponding to the interval $[1/n,1]$ is nonzero. It follows that in $M_n(\C) \ot Q$, the projection $e_{11} \ot ec$ is subequivalent to the projection $1 \ot fc$. Choose a central projection $z \in \cZ(\cM)$ such that $c = z r$. It follows that in $M_n(\C) \ot \cM z$, the projection $e_{11} \ot e z$ is subequivalent to the projection $1 \ot fz$.

Since $z \in \cZ(\cM)$, we may view $B(H_A) \cong B(H_A) z \subset \cM z$ and $B(H_B) \cong B(H_B) z \subset \cM z$. Since the projections $ez$ and $fz$ satisfy the extra assumption, it follows from the first part of the proof that $A$ and $B$ are W$^*$-correlated.

(ii) By Proposition \ref{prop.countably-decomp-separable.1} and Remark \ref{rem.charact-Wstar-corr}, we can choose a II$_1$ factor $P$, projections $p \in P \ovt B$ and $q \in P \ovt A\op$, unital normal $*$-homomorphisms $\al : A \to p(P \ovt B) p$ and $\be\op : B\op \to q(P \ovt A\op)q$, and a unitary $W : p L^2(P \ovt B) \to q L^2(P \ovt A\op)$ such that $W(\al(a) \xi (d \ot b)) = \be\op(b\op) W(\xi) (d \ot a\op)$ for all $a \in A$, $b \in B$ and $d \in P$. Since $A$ and $B$ are factors and $P$ is a II$_1$ factor, after a conjugacy, we may assume that $p \in P \ot 1$ and $q \in P \ot 1$. We thus write $p \ot 1$ and $q \ot 1$ instead, with $p \in P$ and $q \in P$.

Define the semifinite von Neumann algebra $\cM = p P p \ovt B(L^2(B))$ with the commuting subalgebras $\al(A) \subset pPp \ovt B \subset \cM$ and $1 \ot B\op \subset 1 \ot B(L^2(B)) \subset \cM$. By construction, the finite rank projections in $1 \ot B(L^2(B))$ are finite in $\cM$. Also note that $W \cM W^* = q P q \ovt B(L^2(A\op))$, with $W \al(A) W^*$ given by the right action $1 \ot A \subset 1 \ot B(L^2(A\op))$. The finite projections in $1 \ot B(L^2(A\op))$ are finite in $q P q \ovt B(L^2(A\op))$. Modulo the canonical unitary $L^2(A\op) = L^2(A)$, we thus find the required copy of $A \subset B(L^2(A)) \subset \cM$ given by $W^* (1 \ot B(L^2(A\op)))W$. So, $A$ and $B$ are von Neumann equivalent.
\end{proof}

\begin{example}\label{ex.counterex-vNequiv-Wstar-corr}
Note that the discrete abelian von Neumann algebras $\C$ and $\ell^\infty(\N)$ are W$^*$-correlated by Proposition \ref{prop.direct-sums-etc.3}. These finite von Neumann algebras are however not von Neumann equivalent in the sense of Definition \ref{def.vNequiv}. Indeed, assume that $\cM$ is a semifinite von Neumann algebra, and that $\C \subset \C \subset \cM$ and $\ell^\infty(\N) \subset B(\ell^2(\N)) \subset \cM$ witness von Neumann equivalence. Since the identity $1 \in \cM$ is a minimal projection in $\C$, we first get that $1$ is a finite projection. But that contradicts the inclusion $B(\ell^2(\N)) \subset \cM$.
\end{example}

\subsection{\boldmath W$^*$-correlations, intertwining-by-bimodules and relative amenability}

Let $B$ be a tracial von Neumann algebra and $B_0 \subset B$ a von Neumann subalgebra. In the following proposition, the intertwining relation $B \prec_B B_0$ appears. Note that this is equivalent to the existence of a nonzero projection $p \in B_0' \cap B$ such that $L^2(B)p$ is finitely generated as a right Hilbert $B_0$-module. When $B_0$ and $B$ are factors, this is in turn equivalent to $B_0 p \subset p B p$ having finite Jones index.

Also recall that $B_0 \subset B$ is said to be \emph{co-amenable} if $B$ is amenable relative to $B_0$ inside $B$.

\begin{proposition}\label{prop.not-with-smaller-B}
Let $A$, $B$ and $P$ be tracial von Neumann algebras, $p \in P \ovt B$ a projection and assume that $\al : A \to p (P \ovt B)p$ defines a $P$-equivalence between $A$ and $B$. Let $B_0 \subset B$ be a von Neumann subalgebra.
\begin{enumlist}
\item\label{prop.not-with-smaller-B.1} If $\al(A) \prec_{P \ovt B}^f P \ovt B_0$, then $B \prec^f_B B_0$.
\item\label{prop.not-with-smaller-B.2} If $\al(A)$ is amenable relative to $P \ovt B_0$, then $B_0 \subset B$ is co-amenable.
\end{enumlist}
\end{proposition}

\begin{proof}
After replacing $P$ by $M_n(\C) \ot P$, by Remark \ref{rem.charact-Wstar-corr}, we find a projection $q \in P\op \ovt A$, a faithful unital normal $*$-homomorphism $\be : B \to q(P\op \ovt A)q$ and a unitary $V : p L^2(P \ovt B) \to L^2(P\op \ovt A) q$ satisfying
\begin{equation}\label{eq.good-formula-V}
V(\al(a) \xi (d \ot b) = (d\op \ot a) V(\xi) \be(b) \quad\text{for all $a \in A$, $b \in B$, $d \in P$.}
\end{equation}
Given a tracial von Neumann algebra $(Q,\tau)$, projection $r \in Q$ and von Neumann subalgebras $S \subset rQr$ and $D \subset Q$, we use the following characterization for intertwining and relative amenability, using Jones' basic construction $\langle Q, e_D \rangle$ for the inclusion $D \subset Q$; cf.\ \cite[Remark 2.1]{DV24}. We have that $S$ is amenable relative to $D$ if and only if $r \langle Q, e_D \rangle r$ admits an $S$-central positive functional $\Om$ such that $\Om|_{rQr} = \tau|_{rQr}$. We have that $S \prec^f_Q D$ if and only if moreover $\Om$ is normal.

We fix faithful normal tracial states $\tau$ on $A$, $B$ and $P$, and we equip $P \ovt B$ with the tensor product trace, which we also denote by $\tau$. Identifying the basic construction of $P \ovt B_0 \subset P \ovt B$ with $P \ovt \langle B,e_{B_0}\rangle$, we thus fix an $\al(A)$-central positive functional $\Om$ on $p(P \ovt \langle B,e_{B_0}\rangle)p$ such that $\Om(a) = \tau(a)$ for all $a \in p (P \ovt B)p$. If $\al(A)$ is amenable relative to $P \ovt B_0$, such a functional exists. If $\al(A) \prec_{P \ovt B}^f P \ovt B_0$, we may moreover assume that $\Om$ is normal.

For every $T \in \langle B, e_{B_0} \rangle$, define the sesquilinear form $\Om_T$ on the vector space $p (P \ovt B)$ by $\Om_T(x,y) = \Om(x (1 \ot T) y^*)$. Using the Cauchy-Schwarz inequality and the fact that the restriction of $\Om$ to $p(P \ovt B)p$ equals $\tau$, we find that
\begin{equation}\label{eq.estimate-sesqui}
|\Om_T(x,y)| \leq \|T\| \, \|x\|_2 \, \|y\|_2 \quad\text{for all $T \in \langle B, e_{B_0} \rangle$, $x,y \in p (P \ovt B)$.}
\end{equation}
We can thus uniquely extend $\Om_T$ to a sesquilinear form on $p L^2(P \ovt B)$ such that \eqref{eq.estimate-sesqui} holds for all $T \in \langle B, e_{B_0} \rangle$ and $x,y \in p L^2(P \ovt B)$.

Using the notation of \eqref{eq.good-formula-V}, we define the vector $\eta \in p L^2(P \ovt B)$ by $\eta = V^*(q)$. Since for every $x,y \in p(P \ovt B)$, the map $T \mapsto \Om_T(x,y)$ is linear, this map remains linear for all $x,y \in p L^2(P \ovt B)$. We can thus define the linear map
$$\om : \langle B, e_{B_0} \rangle \to \C : \om(T) = \Om_T(\eta,\eta) \; .$$
When $T \geq 0$, we have that $\Om_T(x,x) \geq 0$ for all $x \in p(P \ovt B)$. By continuity, also $\Om_T(\eta,\eta) \geq 0$, so that $\om$ is a positive functional.

We claim that $\om$ is normal if $\Om$ is normal. Assume that $(T_i)_{i \in I}$ is a net in $\langle B,e_{B_0}\rangle$ such that $T_i \to T$ weakly and $\|T_i\| \leq 1$ for all $i \in I$. It suffices to prove that $\om(T_i) \to \om(T)$. For all $x,y \in p (P \ovt B)$, we have that $\Om_{T_i}(x,y) \to \Om_T(x,y)$ by normality of $\Om$. Since $\|T_i\| \leq 1$ for all $i \in I$, it then follows from \eqref{eq.estimate-sesqui} that also $\om(T_i) = \Om_{T_i}(\eta,\eta) \to \Om_T(\eta,\eta) = \om(T)$, so that the claim is proven.

When $T \in B$, we have that $\Om_T(x,y) = \Om(x (1 \ot T) y^*) = \tau(x (1 \ot T) y^*) = \langle x (1 \ot T) , y \rangle$, for all $x,y \in p(P \ovt B)$. By continuity,
$$\om(T) = \Om_T(\eta,\eta) = \langle \eta (1 \ot T),\eta \rangle = \langle V(\eta (1 \ot T)), V(\eta)\rangle = \langle q \be(T),q \rangle = \tau(\be(T)) \; .$$
Since $\tau \circ \be$ is a faithful normal trace on $B$, to conclude the proof of the proposition, it thus suffices to show that $\om$ is $B$-central.

Fix $b \in B$ and $T \in \langle B , e_{B_0} \rangle$ with $\|T\| \leq 1$. We have to prove that $\om(bT) = \om(Tb)$. For all $x,y \in p(P \ovt B)$, we have that $\Om_{bT}(x,y) = \Om(x (1 \ot bT) y^*) = \Om_T(x(1 \ot b),y)$. By continuity, $\om(bT) = \Om_T(\eta(1 \ot b),\eta)$. We similarly get that $\om(T b) = \Om_T(\eta,\eta(1 \ot b^*))$.

Choose an arbitrary $\eps > 0$. Take $n \in \N$, $d_i \in P$ and $a_i \in A$ such that
$$\Bigl\| \be(b) - \sum_{i=1}^n d_i\op \ot a_i \Bigr\|_2 < \eps \; .$$
A fortiori, we have that $\bigl\| \be(b) - \sum_{i=1}^n (d_i\op \ot a_i)q \bigr\|_2 < \eps$. Applying $V^*$ and noting that $V^*(\be(b)) = \eta(1 \ot b)$ and $V^*((d_i\op \ot a_i)q) = \al(a_i) \eta (d_i \ot 1)$, it follows that
$$\Bigl\| \eta(1 \ot b) - \sum_{i=1}^n \al(a_i) \eta (d_i \ot 1) \Bigr\|_2 < \eps \; .$$
We denote $\zeta = \sum_{i=1}^n \al(a_i) \eta (d_i \ot 1)$. Since $\om(bT) = \Om_T(\eta(1 \ot b),\eta)$ and $\|\eta(1 \ot b) - \zeta\|_2 < \eps$, it follows from \eqref{eq.estimate-sesqui} that $|\om(bT) - \Om_T(\zeta,\eta)| < \eps$. By a similar computation, $|\om(Tb) - \Om_T(\eta,\zeta')| < \eps$, where $\zeta' = \sum_{i=1}^n \al(a_i^*) \eta (d_i^* \ot 1)$.

For all $x,y \in p(P \ovt B)$, $a \in A$ and $d \in P$, since $\Om$ is $\al(A)$-central, we have that
$$\Om_T(\al(a) x (d \ot 1),y) = \Om(\al(a) x (d \ot T) y^*) = \Om(x (d \ot T) y^* \al(a)) = \Om_T(x,\al(a^*) y (d^* \ot 1)) \; .$$
By linearity and continuity, it follows that $\Om_T(\zeta,\eta) = \Om_T(\eta,\zeta')$. In combination with the previous paragraph, we conclude that $|\om(bT) - \om(Tb)| < 2 \eps$. Since $\eps > 0$ was arbitrary, the equality $\om(bT) = \om(Tb)$ is proven.
\end{proof}

\subsection{\boldmath W$^*$-correlations of amenable finite von Neumann algebras}\label{sec.Wstar-corr-amenable}

By \cite[Proposition 5.7]{BV22}, we know that amenability is preserved under W$^*$-correlations. For amenable finite von Neumann algebras with separable predual, we obtain the following classification up to W$^*$-correlation.

\begin{proposition}\label{prop.Wstar-corr-amenable}
Among the amenable finite von Neumann algebras with separable predual, there are exactly three equivalence classes up to W$^*$-correlation: the discrete finite von Neumann algebras, the diffuse finite von Neumann algebras, and the nonzero direct sums of a discrete and a diffuse finite von Neumann algebra.
\end{proposition}

\begin{proof}
Let $A$ and $B$ be W$^*$-correlated tracial von Neumann algebras. Fix a finite von Neumann algebra $P$ and a $P$-equivalence $\bim{A}{H}{P \ovt B}$ between $A$ and $B$. If $A$ is diffuse, it follows from Proposition \ref{prop.basic-prop.4} that $B$ is diffuse. If $A$ has a diffuse direct summand $Az$, where $z \in \cZ(A)$ is a nonzero central projection, we consider the $Az$-$(P \ovt B)$-bimodule $z \cdot H$. Define $c \in \cZ(B)$ as the support of the support of the right action of $B$ on $z \cdot H$. Then $c \neq 0$ and by construction, $Az$ is W$^*$-correlated to $B c$, so that $B c$ is diffuse.

It follows from the previous paragraph that among the tracial von Neumann algebras, the discrete, the diffuse, and the nonzero direct sums of a discrete and a diffuse tracial von Neumann algebra form distinct W$^*$-correlation equivalence classes.

By Proposition \ref{prop.direct-sums-etc}, every discrete finite von Neumann algebra is W$^*$-correlated to $\C$. Let $A$ be a diffuse amenable finite von Neumann algebra with separable predual. Denote $B = L^\infty([0,1])$. Since the hyperfinite II$_1$ factor is the unique amenable II$_1$ factor with separable predual, $A$ is isomorphic to a direct sum of copies of $R$, $M_n(\C) \ot B$ for $n \geq 1$, and $R \ovt B$. Combining Example \ref{ex.twisted-group-vNalg} and Propositions \ref{prop.direct-sums-etc.3} and \ref{prop.Wstar-corr-tensor}, we find that $A$ is W$^*$-correlated to $R$.

Combining the previous two paragraphs and Proposition \ref{prop.direct-sums-etc.1}, every nonzero direct sum of a discrete finite von Neumann algebra and a diffuse finite von Neumann with separable predual is W$^*$-correlated to $\C \oplus R$.
\end{proof}

\subsection{\boldmath Properties preserved under $P$-embeddings}\label{sec.preservation-of-properties}

Let $A$ and $B$ be tracial von Neumann algebras. Assume that there exists a $P$-embedding (in the sense of Definition \ref{def.Wstar-corr}) of $A$ into $B$. Which properties of $B$ are inherited by $A$~? By \cite[Proposition 5.7]{BV22}, both amenability and the Haagerup approximation property are inherited from $B$ to $A$. In this section, we first prove the easy result saying the stable solidity is inherited from $B$ to $A$.

We then turn to $P$-embeddings into free group factors. By \cite{Oza03}, a nonamenable tracial von Neumann algebra with property Gamma cannot be embedded into a free group factor $L(\F_n)$. By \cite{Eff73}, if a group von Neumann algebra $L(\Lambda)$ has property Gamma, then the group $\Lambda$ is inner amenable, but by \cite{Vae09}, the converse need not hold. Recall that a group $\Lambda$ is inner amenable if there exists a sequence $\xi_k\in \ell^2(\Lambda)$ of unit vectors that converges to $0$ weakly and for which $\lim_k\| ({\rm Ad} u_g)(\xi_k) -\xi_k) \|_2=0$ for any $g\in\Lambda$.
Therefore, Popa posed the question if the group von Neumann algebra $L(\Lambda)$ of a nonamenable, inner amenable group $\Lambda$ can be embedded into a free group factor. In \cite[Theorems 1.1 and 6.4]{DKEP22}, it was proven that the answer remains no, and an alternative proof was given in \cite[Theorem E]{Dri22}. In Theorem \ref{thm.no-embed-inner-amenable}, we give a variant of the proof of \cite{Dri22}, and prove that it is not even possible to find a $P$-embedding of $L(\Lambda)$ into $L(\F_n)$.

The following result holds almost by definition. For completeness, we include a short proof.

\begin{proposition}
Let $A$ and $B$ be tracial von Neumann algebras. Assume that $A$ admits a $P$-embedding into $B$. If $B$ is stably solid, then $A$ is stably solid.
\end{proposition}
\begin{proof}
Using the same argument as in the proof of Proposition \ref{prop.countably-decomp-separable.1}, we can take a tracial von Neumann algebra $P$, a projection $p \in P \ovt B$ and a faithful unital normal $*$-homomorphism $\al : A \to p(P \ovt B)p$ such that the bimodule $\bim{\al(A)}{L^2(P \ovt B)}{P \ot 1}$ is coarse.

Assume that $B$ is stably solid. We have to prove that $A$ is stably solid. So, choose a tracial von Neumann algebra $N$, projection $q \in N \ovt A$ and von Neumann subalgebra $D \subset q (N \ovt A)q$ such that $D \not\prec N \ot 1$. Denote by $Q := D' \cap q(N\ovt A)q$ the relative commutant. By Proposition \ref{prop.basic-prop.4}, $(\id \ot \al)(D) \not\prec N \ovt P \ovt 1$. Since $B$ is stably solid, it follows that $(\id \ot \al)(Q)$ is amenable relative to $N \ovt P \ovt 1$. By Proposition \ref{prop.basic-prop.5}, $Q$ is amenable relative to $N \ot 1$. So, $B$ is stably solid.
\end{proof}

To formulate our result on $P$-embeddings into free group factors, we introduce the following ad hoc variant of \cite[Section 6.3]{Pop06}. We say that a finite von Neumann algebra $M$ has a \emph{coarse malleable deformation} if there exists a tracial von Neumann algebra $(\Mtil,\tau)$, a sequence of trace preserving automorphisms $\theta_n \in \Aut(\Mtil,\tau)$ and a unital embedding $M \subset \Mtil$ such that the following holds.
\begin{enumlist}
\item The $M$-$M$-bimodule $\bim{M}{L^2(\Mtil \ominus M,\tau)}{M}$ is weakly contained in the coarse $M$-$M$-bimodule.
\item For every $a \in \Mtil$, we have that $\|\theta_n(a)-a\|_2 \to 0$.
\item For every $n \in \N$, the map $M \to M : a \mapsto E_M(\theta_n(a))$ extends to a compact operator on $L^2(M,\tau)$.
\end{enumlist}

\begin{theorem}\label{thm.no-embed-inner-amenable}
Let $M$ be a finite von Neumann algebra that admits a coarse malleable deformation.
\begin{enumlist}
\item $M$ is stably solid, and even $\omega$-stably solid: if $u_k \in \cU(p (N \ovt M)p)$ is a sequence of unitaries such that $\|E_{N \ot 1}(a u_k b)\|_2 \to 0$ for all $a,b \in N \ovt M$, then the asymptotic relative commutant $Q := \{a \in p(N \ovt M)p \mid \|u_k a - a u_k\|_2 \to 0\}$ is amenable relative to $N \ot 1$.
\item If $\Lambda$ is a nonamenable, inner amenable group, there is no $P$-embedding of $L(\Lambda)$ into $M$.
\end{enumlist}
The free group factors $L(\F_n)$, $2 \leq n \leq +\infty$, admit a coarse malleable deformation.
\end{theorem}

\begin{proof}
Fix a coarse malleable deformation of $M$ given by $M \subset (\Mtil,\tau)$ and $\theta_n \in \Aut(\Mtil,\tau)$ as above. We equip $M$ with the trace $\tau|_M$, and define the unital, trace preserving, completely positive maps $\vphi_n : M \to M : \vphi_n(a) = E_M(\theta_n(a))$. By assumption, $\vphi_n$ uniquely extends to a compact operator on $L^2(M,\tau)$.

(i) Take a sequence $u_k \in \cU(p (N \ovt M)p)$ such that $\|E_{N \ot 1}(a u_k b)\|_2 \to 0$ for all $a,b \in N \ovt M$, and define $Q$ as in statement~1. For every $n,k \in \N$, consider the vector $\zeta_{n,k} = (\id \ot \theta_n)(u_k) \in N \ovt \Mtil$. For all $a \in N \ovt \Mtil$, we have that $\|a\zeta_{n,k}\|_2 \leq \|a\|_2$. Since $\theta_n \to \id$, we get that
\begin{align*}
& \lim_n \bigl(\sup_k \|\zeta_{n,k} - p \zeta_{n,k} p\|_2\bigr) = 0 \;\; , \\
& \lim_n \bigl( \sup_k |\langle a \zeta_{n,k},\zeta_{n,k}\rangle - \tau(pap)|\bigr) = 0 = \lim_n \bigl( \sup_k |\langle \zeta_{n,k} a,\zeta_{n,k}\rangle - \tau(pap)|\bigr)\;\;\text{for all $a \in N \ovt M$,}\\
& \lim_n \bigl(\limsup_k \|b \zeta_{n,k} - \zeta_{n,k} b \|_2\bigr) \;\;\text{for all $b \in Q$.}
\end{align*}
We then consider $\xi_{n,k} = p(\zeta_{n,k} - (\id \ot E_M)(\zeta_{n,k}))p$ in $p(L^2(N) \ot L^2(\Mtil \ominus M))p$. We claim that for every fixed $n$, $\lim_k \|(\id \ot E_M)(\zeta_{n,k})\|_2 = 0$. Since $E_M \circ \theta_n$ defines a compact operator on $L^2(M)$, it suffices to prove that $\lim_k \|(\id \ot \psi)(u_k)\|_2 = 0$ for every $\psi \in \cK(L^2(M))$. In turn, it suffices to check this for a rank one operator $\psi_{a,b}(x) = \tau(b^* x) a$ given by $a,b \in M$. In that case, the result holds because $\|(\id \ot \tau)(u_k(1 \ot b^*))\|_2 \to 0$ by our assumption on the sequence $(u_k)_k$. So the claim is proven.

So we may choose $k_1 < k_2 < \cdots$ such that $\|(\id \ot E_M)(\zeta_{n,k_n})\|_2 < 1/n$ for all $n \geq 1$. Since the $M$-$M$-bimodule $L^2(\Mtil \ominus M)$ is weakly contained in the coarse bimodule, the sequence $\xi_{n,k_n}$ shows that $Q$ is amenable relative to $N \ot 1$ inside $N \ovt M$.

(ii) Let $\Lambda$ be a nonamenable group. Denote by $L(\Lambda)$ its group von Neumann algebra, generated by the unitary operators $(u_g)_{g \in \Lambda}$. Assume that there exists a $P$-embedding of $L(\Lambda)$ into $M$. We have to prove that $\Lambda$ is not inner amenable. Reasoning as in the proof of Proposition \ref{prop.countably-decomp-separable.1}, we may assume that $P$ is tracial. We then choose a faithful unital normal $*$-homomorphism $\al : L(\Lambda) \to p(P \ovt M)p$ such that the bimodule $\bim{\al(L(\Lambda))}{p L^2(P \ovt M)}{P \ot 1}$ is coarse.

We have to prove that $\Lambda$ is not inner amenable. So assume that $\eta_k \in \ell^2(\Lambda)$ is a sequence of unit vectors such that $\lim_k \|(\Ad u_g)(\eta_k) - \eta_k\|_2 =0$ for all $g \in \Lambda$. We have to prove that $(\eta_k)_k$ does not converge to $0$ weakly.

As in \cite{Dri22}, we consider $\be : L(\Lambda) \to L(\Lambda) \ovt p(P \ovt M)p : \be(u_g) = u_g \ot \al(u_g)$. Note that $\be$ is trace preserving, so that $\be(\eta_k) \in L^2(L(\Lambda) \ovt P \ovt M)$ is well-defined. We define $\zeta_{n,k} = (\id \ot \id \ot \theta_n)\be(\eta_k)$ and $\xi_{n,k} = (1 \ot p)(\zeta_{n,k} - (\id \ot \id \ot E_M)(\zeta_{n,k}))(1 \ot p)$. We now reason in a similar way as in the proof of 1, but we have to be more careful, because we no longer have an upper bound on $\|a \zeta_{n,k}\|_2$ in terms of $\|a\|_2$, for arbitrary $a \in L(\Lambda) \ovt P \ovt \Mtil$.

Note that
\begin{equation}\label{eq.formula-zeta-n-k}
\zeta_{n,k} = \sum_{h \in \Lambda} \eta_k(h) \, (u_h \ot (\id \ot \theta_n)\al(u_h)) \; .
\end{equation}
It follows that for all $g \in \Lambda$ and $a \in P \ovt \Mtil$, and all $n,k \in \N$,
$$\|(u_g \ot a)\zeta_{n,k}\|_2^2 = \sum_{h \in \Lambda} |\eta_k(h)|^2 \, \|a(\id \ot \theta_n)\al(u_h)\|_2^2 \leq \sum_{h \in \Lambda} |\eta_k(h)|^2 \, \|a\|_2^2 = \|a\|_2^2 \; ,$$
and similarly for $\|\zeta_{n,k} (u_g \ot a)\|_2$. As in the proof of 1, we thus get that
\begin{equation}\label{eq.with-p}
\lim_n \bigl(\sup_k \|\zeta_{n,k} - (1 \ot p)\zeta_{n,k}(1 \ot p)\|_2\bigr) = 0 \; .
\end{equation}
Since $(\id \ot \id \ot \theta_n)\be(u_g) = u_g \ot (\id \ot \theta_n)\al(u_g)$ and $\theta_n \to \id$, we also get for every $g \in \Lambda$,
\begin{align*}
& \lim_n \bigl(\sup_k \|\be(u_g) \zeta_{n,k} - (\id \ot \id \ot \theta_n)\be(u_g \eta_k)\|_2\bigr) = 0 \quad\text{and}\\
& \lim_n \bigl(\sup_k \|\zeta_{n,k} \be(u_g) - (\id \ot \id \ot \theta_n)\be(\eta_k u_g)\|_2\bigr) \; .
\end{align*}
Since for every $g \in \Lambda$, we have $\lim_k \|u_g \eta_k - \eta_k u_g\|_2 = 0$, we get that
$$\lim_n \bigl(\limsup_k \|\be(u_g) \zeta_{n,k} \be(u_g)^* - \zeta_{n,k}\|_2\bigr) = 0 \quad\text{for all $g \in \Lambda$.}$$
Defining $\xi_{n,k} = (1 \ot p)(\zeta_{n,k} - (\id \ot \id \ot E_M)(\zeta_{n,k}))(1 \ot p)$, we have found the vectors $\xi_{n,k} \in K := (1 \ot p)(\ell^2(\Lambda) \ot L^2(P) \ot L^2(\Mtil \ominus M))(1 \ot p)$ satisfying
\begin{equation}\label{eq.my-estimate-here}
\lim_n \bigl(\limsup_k \|\be(u_g) \xi_{n,k} \be(u_g)^* - \xi_{n,k}\|_2\bigr) = 0 \quad\text{for all $g \in \Lambda$.}
\end{equation}
Since the $M$-$M$-bimodule $L^2(\Mtil \ominus M)$ is weakly contained in the coarse $M$-$M$-bimodule and since the bimodule $\bim{\al(L(\Lambda))}{pL^2(P \ovt M)}{P \ot 1}$ is coarse, the unitary representation $(\Ad \be(u_g))_{g \in \Lambda}$ of $\Lambda$ on $K$ is weakly contained in the regular representation of $\Lambda$. Since $\Lambda$ is nonamenable, it follows from \eqref{eq.my-estimate-here} that $\lim_n \limsup_k \|\xi_{n,k}\|_2 = 0$. Using \eqref{eq.with-p}, this means that
$$\lim_n \bigl( \liminf_k \|(\id \ot \id \ot E_M)(\zeta_{n,k})\|_2\bigr) = \|p\|_2 \; .$$
We can thus fix $n_0,k_0 \in \N$ such that
$$\|(\id \ot \id \ot E_M)(\zeta_{n_0,k})\|^2_2 \geq \tau(p)/2 \quad\text{for all $k \geq k_0$.}$$
Using \eqref{eq.formula-zeta-n-k}, this means that
\begin{equation}\label{eq.almost-end}
\tau(p)/2 \leq \sum_{g \in \Lambda} |\eta_k(g)|^2 \, \|(\id \ot E_M \circ \theta_{n_0})\al(u_g)\|_2^2 \quad\text{for all $k \geq k_0$.}
\end{equation}
We claim that $\lim_{g \to \infty} \|(\id \ot E_M \circ \theta_{n_0})\al(u_g)\|_2 = 0$, where the notation $g \to \infty$ means that $g$ leaves any finite subset of $\Lambda$. For every $b \in M$, define the functional $\vphi_b \in M_*$ by $\vphi_b(x) = \tau(xb)$. Since $E_M \circ \theta_{n_0}$ defines a compact operator on $L^2(M)$, to prove the claim, it suffices to prove that $\lim_{g \to \infty} \|(\id \ot \vphi_b)\al(u_g)\|_2 = 0$ for all $b \in M$. When $g \to \infty$, we have that $u_g \to 0$ weakly in $L(\Lambda)$. Since the bimodule $\bim{\al(L(\Lambda))}{pL^2(P \ovt M)}{P \ot 1}$ is coarse, it is in particular mixing, meaning that for all $\mu_1,\mu_2 \in pL^2(P \ovt M)$,
$$\lim_{g \to \infty} \sup_{a \in P , \|a\|\leq 1} |\langle \al(u_g) \mu_1 (a \ot 1) , \mu_2 \rangle| = 0 \; .$$
Applying this to $\mu_1 = p(1 \ot b)$ and $\mu_2 = p$, we get that $\lim_{g \to \infty} \|(\id \ot \vphi_b)\al(u_g)\|_1 = 0$. Since the net $((\id \ot \vphi_b)\al(u_g))_{g \in \Lambda}$ is bounded, also the $2$-norm converges to zero and the claim is proven.

We can thus choose a finite subset $\cF \subset \Lambda$ such that $\|(\id \ot E_M \circ \theta_{n_0})\al(u_g)\|_2^2 < \tau(p)/4$ for all $g \in \Lambda \setminus \cF$. It then follows from \eqref{eq.almost-end} that
$$\tau(p)/2 \leq \sum_{g \in \cF} |\eta_k(g)|^2 + (1/4) \tau(p) \sum_{g \in \Lambda \setminus \cF} |\eta_k(g)|^2 \leq \sum_{g \in \cF} |\eta_k(g)|^2 + (1/4) \tau(p)$$
for all $k \geq k_0$. This implies that $\sum_{g \in \cF} |\eta_k(g)|^2 \geq \tau(p)/4$ for all $k \geq k_0$. So, the sequence $(\eta_k)_k$ does not converge to zero weakly, which concludes the proof of 2.

It remains to prove that the free group factors $L(\F_n)$ admit a coarse malleable deformation. Using the notation of Section \ref{sec.q-Gaussian}, we realize $L(\F_n)$ via Voiculescu's free Gaussian functor: $L(\F_n) \cong M_0(H_\R)$, where $H_\R$ is a real Hilbert space of dimension $n$, with $H_\R = \ell^2_\R(\N)$ if $n=\infty$. We put $M = M_0(H_\R)$ and $\Mtil = M_0(H_\R \oplus H_\R)$. Using the identification $L^2(\Mtil) = \cF_0(H \oplus H)$, it is easy to check that the $M$-$M$-bimodule $L^2(\Mtil \ominus M)$ is coarse.

Choose a sequence $t_n \in (0,\pi/2)$ such that $t_n \to 0$. Choose an increasing sequence of finite dimensional real subspaces $H_{n,\R} \subset H_\R$ such that the corresponding orthogonal projections $P_n$ converge to $1$ strongly. Define $v_n \in \cO(H_\R \oplus H_\R)$ by
$$v_n = \begin{pmatrix} \cos t_n \, P_n & \sin t_n \, P_n + (1-P_n) \\ -\sin t_n \, P_n + (1-P_n) & \cos t_n \, P_n \end{pmatrix} \; .$$
Denote by $\theta_n \in \Aut \Mtil$ the corresponding automorphisms. Fix $n \in \N$. It remains to prove that the operator $M \to M : a \mapsto E_M(\theta_n(a))$ is compact on $L^2(M) = \cF_0(H)$. Note however that this operator is given by
$$\bigoplus_{k=0}^\infty (\cos t_n \, P_n)^{\ot^k} \; ,$$
which is compact because $\lim_{k \to \infty} (\cos t_n)^k = 0$ and $P_n^{\ot^k}$ has finite rank for all $n,k$.
\end{proof}

\subsection{\boldmath W$^*$-correlations: passage to von Neumann subalgebras}

We prove two technical lemmas that will be used below. The main point will be to deduce, in certain cases, from a W$^*$-correlation between tracial von Neumann algebras $A$ and $B$, a W$^*$-correlation between certain von Neumann subalgebras $A_1 \subset A$ and $B_1 \subset B$. We first formulate the technical lemma and then explain its more intuitive meaning.

\begin{lemma}\label{lem.left-right-embedding}
Let $A$, $B$ and $P$ be tracial von Neumann algebras and $\bim{\al(A)}{pL^2(P \ovt B)}{P \ovt B}$ a $P$-embedding of $A$ into $B$. Let $A_1 \subset A$ and $B_1,B_2 \subset B$ be von Neumann subalgebras.
\begin{itemlist}
\item Let $K \subset p L^2(P \ovt B)$ be a $(P\op \ovt A_1)$-$B_1$-subbimodule that is finitely generated as a left Hilbert $(P\op \ovt A_1)$-module.
\item Let $V \in \C^n \ot (P \ovt B)p$ and $\vphi : A_1 \to q(M_n(\C) \ot (P \ovt B_2))q$ a normal unital $*$-homomorphism such that $V \al(a) = \vphi(a) V$ for all $a \in A_1$.
\item Define $E_r(B_2,B_1)$ as the set of $b \in B$ for which there exist $k \in \N$ and $b_1,\ldots,b_k \in B$ such that $b B_1 \subset B_2 b_1 + \cdots + B_2 b_k$.
\end{itemlist}
Then $V \cdot K$ is contained in the $\|\cdot\|_2$-closure of $\C^n \ot P \ot E_r(B_2,B_1)$.
\end{lemma}

Note that the bimodule $K$ defines a $P\op$-embedding of $B_1$ into $A_1$, while $V$ and $\vphi$ witness an intertwining of $\al(A_1)$ into $P \ovt B_2$. It is therefore quite natural to relate the product $V \cdot K$ to intertwining $B_1$ into $B_2$, as witnessed by $E_r(B_2,B_1)$.

\begin{proof}
Write $H = p L^2(P \ovt B)$. By definition, the bimodule $\bim{\al(A)}{H}{P \ot 1}$ is coarse. We can thus define the normal $*$-homomorphism $\pi : P\op \ovt A \to B(H)$ such that $\pi(d\op \ot a)(\xi) = \al(a)\xi(d\ot 1)$ for all $\xi \in H$, $d \in P$ and $a \in A$.

Choose an integer $m$, projection $r \in M_m(\C) \ovt P\op \ovt A_1$ and a left $(P\op \ovt A_1)$-modular unitary $W : (M_{1,m}(\C) \ot L^2(P\op \ovt A_1))r \to K$. Define the normal unital $*$-homomorphism $\psi : B_1 \to r(M_m(\C) \ot (P\op \ovt A_1))r$ such that $W(\xi \psi(b)) = W(\xi) (1 \ot b)$ for all $b \in B_1$. Define $\eta \in \C^m \ot H$ by $\eta = \sum_{s=1}^m e_s \ot W((e_{1s} \ot 1)r)$. By construction, $\eta(1 \ot b) = (\id \ot \pi)(\psi(b)) \eta$ for all $b \in B_1$. Also, $K$ equals the closed linear span of $\{\al(a) \eta_s (d \ot 1) \mid a \in A_1, d \in P, s \in \{1,\ldots,m\}\}$.

Define $E \subset L^2(B)$ as the $\|\cdot\|_2$-closure of $E_r(B_2,B_1)$. Then $E$ is a Hilbert $B_2$-$B_1$-bimodule. For every $a \in A_1$, $d \in P$ and $s \in \{1,\ldots,m\}$, we have that
$$V \al(a) \eta_s (d \ot 1) = \vphi(a) (V \eta_s)(d \ot 1) \; .$$
It thus suffices to prove that $(1 \ot V) \eta \in \C^m \ot \C^n \ot L^2(P) \ot E$.

Define $\cM \subset B(\C^m \ot \C^n \ot L^2(P \ovt B))$ as the von Neumann algebra generated by the left action of $M_m(\C) \ot \vphi(A_1)$ and the right action of $P \ot 1$. Note that the identity of $\cM$ is given by left multiplication by $1 \ot \vphi(1)$. Also note that the left action of $1 \ot VV^*$ commutes with $\cM$. We denote by $z \in \cZ(M)$ the smallest projection in $\cM$ satisfying $1 \ot VV^* \leq z$. By construction,
$$\gamma : \cM z \to B(\C^m \ot L^2(P \ovt B)) : \gamma(T) = (1 \ot V^*) T (1 \ot V)$$
is a faithful normal (not necessarily unital) $*$-homomorphism. Write $p_1 = V^* V$ and note that $p_1 \in \al(A_1)' \cap p(P \ovt B)p$. Then note that $\gamma(\cM z)$ equals the von Neumann algebra generated by the left action of $M_m(\C) \ot \al(A_1) p_1$ and the right action of $P \ot 1$ inside $B(\C^m \ot L^2(P \ovt B))$. This means that
\begin{equation}\label{eq.good-descr}
\gamma(\cM z) = M_m(\C) \ot \pi(P\op \ovt A_1) p_1 \; ,
\end{equation}
where one should note that the left multiplication action of $p_1$ on $H$ commutes with $\pi(P\op \ovt A_1)$. It follows from \eqref{eq.good-descr} that $(1 \ot p_1) (\id \ot \pi)\psi(B_1)$ is a von Neumann subalgebra of $\gamma(\cM z)$. Since $\gamma$ is faithful, we can uniquely define the normal $*$-homomorphism $\Phi : B_1 \to \cM z$ such that $\gamma(\Phi(b)) = (1 \ot p_1) (\id \ot \pi)\psi(b)$ for all $b \in B_1$.

Since for every $T \in \cM z$, we have that $(1 \ot V) \gamma(T) = T (1 \ot V)$ and since $V = V p_1$, we find that for all $b \in B_1$,
\begin{align*}
(1 \ot V) \eta (1 \ot b) &= (1 \ot V) \cdot ((\id \ot \pi)\psi(b) \eta) = ((1 \ot V) \gamma(\Phi(b)))(\eta) = \Phi(b)((1 \ot V) \cdot \eta) \; .
\end{align*}
Write $\xi = (1 \ot V) \cdot \eta$. Now note that by construction, $\cM \subset M_m(\C) \ovt M_n(\C) \ovt B(L^2(P)) \ovt B_2$. We have thus found the normal, not necessarily unital, $*$-homomorphism
$$\Phi : B_1 \to M_m(\C) \ovt M_n(\C) \ovt B(L^2(P)) \ovt B_2 \quad\text{such that}\quad \xi(1 \ot b) = \Phi(b) \xi \quad\text{for all $b \in B_1$.}$$
Denote by $P_E$ the orthogonal projection of $L^2(B)$ onto $E$. Write
$$\xi_0 = (1 \ot 1 \ot 1 \ot (1-P_E))(\xi) \in \C^m \ot \C^n \ot L^2(P) \ot E^\perp \; .$$
We have to prove that $\xi_0 = 0$. Since $E$ is defined as the closure of $E_r(B_2,B_1)$, by Lemma \ref{lem.intertwining-weak-normalizer.2}, we can choose a net of unitaries $(v_i)_{i \in I}$ in $\cU(B_1)$ such that
\begin{equation}\label{eq.goes-to-zero}
\lim_i \|P_{E_0}(\zeta v_i)\|_2 = 0 \quad\text{\parbox[t]{9cm}{for every $\zeta \in L^2(B) \ominus E$ and every finitely generated left Hilbert $B_2$-submodule $E_0 \subset L^2(B) \ominus E$.}}
\end{equation}
Since $E$ is a Hilbert $B_2$-$B_1$-bimodule, the equality $\xi(1 \ot b) = \Phi(b) \xi$ implies that $\xi_0(1 \ot b) = \Phi(b) \xi_0$ for all $b \in B_1$.

Choose $\eps > 0$. Choose a finite linear combination
$$\xi_1 = \sum_{j=1}^\kappa \nu_j \ot \zeta_j \quad\text{with $\nu_j \in \C^m \ot \C^n \ot L^2(P)$ and $\zeta_j \in L^2(B) \ominus E$}$$
such that $\|\xi_0 - \xi_1\|_2 < \eps$.  For every $i \in I$, we have that $\|\xi_0(1 \ot v_i) - \xi_1 (1 \ot v_i)\|_2 < \eps$ and $\|\Phi(v_i) \xi_0 - \Phi(v_i) \xi_1\|_2 < \eps$, so that
\begin{equation}\label{eq.weer-tussenstap}
\|\Phi(v_i) \xi_1 - \xi_1 (1 \ot v_i)\|_2 < 2 \eps \quad\text{for all $i \in I$.}
\end{equation}
Define $E_0 \subset L^2(B) \ominus E$ as the closed linear span of $\{B_2 \zeta_j \mid j = 1,\ldots,\kappa\}$. Since $E_0$ is a left $B_2$-module, we get that $\Phi(v_i) \xi_1 \in \C^m \ot \C^n \ot L^2(P) \ot E_0$ for all $i \in I$. It thus follows from \eqref{eq.weer-tussenstap} that
\begin{equation}\label{eq.weer-stap}
\|(1 \ot 1 \ot 1 \ot (1-P_{E_0}))(\xi_1(1 \ot v_i))\|_2 < 2\eps \quad\text{for all $i \in I$.}
\end{equation}
By \eqref{eq.goes-to-zero}, $\lim_i \|(1 \ot 1 \ot 1 \ot P_{E_0})(\xi_1(1 \ot v_i))\|_2 = 0$. By \eqref{eq.weer-stap}, we can thus choose an $i \in I$ such that $\|\xi_1 (1 \ot v_i)\|_2 < 2 \eps$. Since $v_i$ is a unitary, it follows that $\|\xi_1\|_2 < 2 \eps$ and thus $\|\xi_0\|_2 < 3 \eps$. Since $\eps > 0$ was arbitrary, we have proven that $\xi_0 = 0$.
\end{proof}

The following lemma is quite technical, but very useful for our purposes. The methods of Popa's deformation/rigidity theory often allow to prove intertwining relations between von Neumann subalgebras. The lemma allows us to deduce W$^*$-correlations from intertwining.

\begin{lemma}\label{lem.find-subequiv}
Let $(A,\tau)$ be a tracial von Neumann algebra with a family of von Neumann subalgebras $(A_i)_{i \in I}$. Denote the normalizer as $\cA_i = \cN_A(A_i)\dpr$ and make the following assumptions:
\begin{itemlist}
\item If $i_1 \neq i_2$, then $A_{i_1} \not\prec_A A_{i_2}$.
\item If $K \subset L^2(A)$ is an $A_i$-$A_i$-subbimodule that is finitely generated as a right Hilbert $A_i$-module, then $K \subset L^2(\cA_i)$.
\end{itemlist}
Let $(B,\tau)$ be a tracial von Neumann algebra with a family of von Neumann subalgebras $(B_j)_{j \in J}$ with normalizer $\cB_j = \cN_B(B_j)\dpr$ satisfying the same properties.

Let $(P,\tau)$ be a tracial von Neumann algebra and $\bim{A}{H}{P \ovt B}$ a $P$-equivalence between $A$ and $B$. We make the following assumptions.
\begin{enumlist}
\item For every $i \in I$ and $j \in J$, denote by $L_{i,j} \subset H$ the closed linear span of all $A_i$-$(P \ovt B_j)$-subbimodules of $H$ that are finitely generated as a right Hilbert $(P \ovt B_j)$-module. Assume that for every $i \in I$, the closed linear span of $\{ L_{i,j}(P \ovt B) \mid j \in J\}$ equals $H$.
\item For every $i \in I$ and $j \in J$, denote by $R_{i,j} \subset H$ the closed linear span of all $(P\op \ovt A_i)$-$B_j$-subbimodules of $H$ that are finitely generated as a left Hilbert $(P\op \ovt A_i)$-module. Assume that for every $j \in J$, the closed linear span of $\{ (P\op \ovt A) R_{i,j} \mid i \in I\}$ equals $H$.
\end{enumlist}
For every $i \in I$ and $j \in J$, denote by $K_{i,j} \subset H$ the closed linear span of all $A_i$-$(P \ovt B_j)$-subbimodules of $H$ that are finitely generated both as a right Hilbert $(P \ovt B_j)$-module and as a left Hilbert $(P\op \ovt A_i)$-module. Then the following holds.
\begin{itemlist}
\item $L_{i,j} = R_{i,j} = K_{i,j}$ for all $i \in I$ and $j \in J$.
\item For every $i \in I$ and distinct $j,j' \in J$, the subspaces $K_{i,j} (P \ovt B)$ and $K_{i,j'} (P \ovt B)$ are orthogonal.
\item For every $j \in J$ and distinct $i,i' \in I$, the subspaces $A K_{i,j}(P \ot 1)$ and $A K_{i',j}(P \ot 1)$ are orthogonal.
\end{itemlist}
\end{lemma}

Note that some $L_{i,j}$ and $R_{i,j}$ may be equal to $\{0\}$. When they are nonzero, the direct summands of $A_i$ and $B_j$ that act faithfully on $L_{i,j} = R_{i,j} = K_{i,j}$ are W$^*$-correlated. Also note that $L_{i,j}$, $R_{i,j}$ and $K_{i,j}$ are automatically $\cA_i$-$(P \ovt \cB_j)$-subbimodules of $H$.

\begin{proof}
After replacing $P$ by a matrix algebra over $P$, we may assume that the $P$-equivalence $\bim{A}{H}{P \ovt B}$ is given by $\bim{\al(A)}{pL^2(P \ovt B)}{P \ovt B}$. For every $i \in I$ and $j \in J$, denote by $p_{i,j} \in \al(A_i)' \cap p(P \ovt B)p$ the largest projection such that $\al(A_i)p_{i,j} \prec^f P \ovt B_j$. Then left multiplication by $p_{i,j}$ is the orthogonal projection of $H$ onto the closure of $L_{i,j} (P \ovt B)$. By assumption (i), we have for every $i \in I$ that $\bigvee_{j \in J} p_{i,j} = p$. Also note that $p_{i,j}$ belongs to the center of $\al(A_i)' \cap p(P \ovt B)p$. In particular, $p_{i,j}$ and $p_{i,k}$ commute for all $i \in I$ and $j,k \in J$.

Denote by $\cF_{i,j}$ the data witnessing $\al(A_i) \prec P \ovt B_j$ whenever $L_{i,j} \neq \{0\}$. More precisely, $\cF_{i,j}$ consists of the tuples $(n,p_0,V,\vphi)$ where $n \in \N$, $p_0 \in M_n(\C) \ot (P \ovt B_j)$ is a projection, $V \in \C^n \ot (P \ovt B)p$ is a partial isometry and $\vphi : A_i \to p_0(M_n \ot (P \ovt B_j))p_0$ is a normal unital $*$-homomorphism satisfying $V \al(a) = \vphi(a) V$ for all $a \in A_i$.

Let $(n,p_0,V,\vphi) \in \cF_{i,j}$. From our two assumptions on the subalgebras $B_j$, we know that $B_k \not\prec_B B_j$ if $k \neq j$ and that $E_r(B_j,B_j) \subset L^2(\cB_j)$. By Lemma \ref{lem.left-right-embedding}, we thus find that
\begin{enumlist}[label=(\alph*)]
\item $V \cdot R_{i,k} = \{0\}$ for all $k \in J$ with $k \neq j$,
\item $V \cdot R_{i,j} \subset \C^n \ot L^2(P \ovt \cB_j)$.
\end{enumlist}
Since the join of all $V^* V$ with $(n,p_0,V,\vphi) \in \cF_{i,j}$ equals $p_{i,j}$ and since for all $(n,p_0,V,\vphi) \in \cF_{i,j}$, we have that
$$V^*(\C^n \ot L^2(P \ovt \cB_j)) \subset L_{i,j} \; ,$$
statements (a) and (b) imply that $p_{i,j} \cdot R_{i,k} = \{0\}$ if $j \neq k$ and $p_{i,j} \cdot R_{i,j} \subset L_{i,j}$.

Fix $i \in I$ and $j \in J$. Define $p_1 = \bigvee_{k \in J \setminus \{j\}} p_{i,k}$. By assumption (i), $p_{i,j} \vee p_1 = p$. Also, $p_{i,j}$ and $p_1$ commute. So, $p = p_{i,j} + (1-p_{i,j})p_1$. Since $p_{i,k} \cdot R_{i,j} = \{0\}$ for all $k \neq j$, also $p_1 \cdot R_{i,j} = \{0\}$ and thus $R_{i,j} = p_{i,j} \cdot R_{i,j} \subset L_{i,j}$.

By symmetry, also the converse inclusion $L_{i,j} \subset R_{i,j}$ holds and we have proven that $L_{i,j} = R_{i,j}$ for all $i \in I$, $j \in J$.

When $i \in I$ and $j,j'$ are distinct elements of $J$, we find that $p_{i,j'} \cdot L_{i,j} = p_{i,j'} \cdot R_{i,j} = \{0\}$, so that also $p_{i,j'} \cdot (L_{i,j}(P \ovt B)) = \{0\}$, which means that $p_{i,j'} p_{i,j} = 0$. So, the subspaces $L_{i,j}(P \ovt B)$ and $L_{i,j'}(P \ovt B)$ are orthogonal.

By symmetry, also $A R_{i,j} (P \ot 1)$ is orthogonal to $A R_{i',j}(P \ot 1)$ whenever $i,i' \in I$ are distinct and $j \in J$.

Fix $i \in I$ and $j \in J$. It remains to prove that $K_{i,j} = L_{i,j} = R_{i,j}$. Choose an increasing net of $A_i$-$(P \ovt B_j)$-subbimodules $L_n \subset L_{i,j}$ that are finitely generated as a right Hilbert $(P \ovt B_j)$-module and whose union is dense in $L_{i,j}$. Similarly choose an increasing net of $(P\op \ovt A_i)$-$B_j$-subbimodules $R_n \subset R_{i,j}$ that are finitely generated as a left Hilbert $(P\op \ovt A_i)$-module and whose union is dense in $R_{i,j}$. Denote by $\pi_n$ the orthogonal projection of $R_{i,j}$ onto $R_n$. Define $K_n \subset L_n$ as the cokernel of the restriction of $\pi_n$ to $L_n$. Using $\pi_n$, we find that $K_n$ is isomorphic with a $(P\op \ovt A_i)$-$B_j$-subbimodule of $R_n$. Since also $K_n \subset L_n$, we get that $K_n \subset K_{i,j}$. Since the linear span of all $K_n$ is dense in $L_{i,j} = R_{i,j}$, we find that $L_{i,j} = R_{i,j} \subset K_{i,j}$. Since the converse inclusion is trivial, the lemma is proven.
\end{proof}

\section{\boldmath W$^*$-correlations of tensor products}

In this section, we prove our main result on W$^*$-correlations with tensor products of stably solid II$_1$ factors, stated as Theorem \ref{thm.main-B} in the introduction. We state and prove here the following more precise version. As mentioned in the introduction, this result is a W$^*$-correlation variant of \cite[Theorem 1]{OP03} for isomorphisms.

\begin{theorem}\label{thm.W-star-corr-tensor-product-any-tensor-product}
Let $A = A_1 \ovt \cdots \ovt A_n$ be a tensor product of nonamenable, stably solid II$_1$ factors. Let $B = B_1 \ovt \cdots \ovt B_r$ be an arbitrary tensor product of II$_1$ factors.
\begin{enumlist}
\item\label{thm.W-star-corr-tensor-product-any-tensor-product.1} $A$ and $B$ are W$^*$-correlated if and only if we can partition $\{1,\ldots,n\}$ into nonempty subsets $S_1 \sqcup \cdots \sqcup S_r$ such that for every $k \in \{1,\ldots,r\}$, the tensor product $A_{S_k} := \ovt_{i \in S_k} A_i$ is W$^*$-correlated to $B_k$.
\item\label{thm.W-star-corr-tensor-product-any-tensor-product.2} For every W$^*$-correlation $\bim{A}{H}{P \ovt B}$ between $A$ and $B$, there exists a partition $\{1,\ldots,n\} = S_1 \sqcup \cdots \sqcup S_r$ and an $A$-$(P \ovt B)$-subbimodule $H_0 \subset H$ such that for every $k \in \{1,\ldots,r\}$, the Hilbert space $H_0$ is densely spanned by its $A_{S_k}$-$(P \ovt B_k)$-subbimodules that define a W$^*$-correlation between $A_{S_k}$ and $B_k$.
\end{enumlist}
\end{theorem}

Before proving Theorem \ref{thm.W-star-corr-tensor-product-any-tensor-product}, we need two preliminary results. In particular, we first need to prove a special version of Theorem \ref{thm.W-star-corr-tensor-product-any-tensor-product}, namely Theorem \ref{thm.corr-tensor-product}, saying that two tensor products $A_1 \ovt \cdots \ovt A_n$ and $B_1 \ovt \cdots \ovt B_m$ of nonamenable, stably solid II$_1$ factors can only be W$^*$-correlated when $n=m$ and $A_j$ is W$^*$-correlated by $B_{\si(j)}$ for all $j$ and a bijection $\si$.

And before that, we need the following lemma, which essentially goes back to \cite{OP03,OP07}.

\begin{lemma}\label{lem.coarse-embedding-tensor-of-solid}
Let $n \geq m \geq 1$ be integers, $A_1,\ldots,A_n$ tracial von Neumann algebras without amenable direct summand and $B_1,\ldots,B_m$ stably solid tracial von Neumann algebras. Write $A = A_1 \ovt \cdots \ovt A_n$ and $B = B_1 \ovt \cdots \ovt B_m$. Let $P$ be a tracial von Neumann algebra and $\theta : A \to p(P \ovt B)p$ a unital normal $*$-homomorphism such that the bimodule $\bim{\theta(A)}{p L^2(P \ovt B)}{P \ot 1}$ is coarse.

Then $n=m$ and for every $i \in \{1,\ldots,n\}$, there exists a $j \in \{1,\ldots,n\}$ such that $\theta(A_i) \prec P \ovt B_j$.
\end{lemma}
\begin{proof}
We prove the lemma by induction on $n$. When $n = 1$, there is nothing to prove. So assume that $n \geq 2$ and that the lemma holds for strictly smaller values of $n$.

For every $i \in \{1,\ldots,n\}$, we denote by $\widehat{A_i} \subset A$ the von Neumann subalgebra given as the tensor product of all $A_k$, $k \neq i$. We similarly define $\widehat{B_j} \subset B$.

Fix $i \in \{1,\ldots,n\}$. By Proposition \ref{prop.basic-prop.5}, $\theta(A_i)$ is not amenable relative to $P \ot 1$. So by Proposition \ref{prop.basic-prop.2}, we can choose $j \in \{1,\ldots,m\}$ such that $\theta(A_i)$ is not amenable relative to $P \ovt \widehat{B_j}$. Since $B_j$ is stably solid, it follows that $\theta(\widehat{A_i}) \prec P \ovt \widehat{B_j}$. Take $k \in \N$, a nonzero partial isometry $V \in \C^k \ot (P \ovt B)p$, a projection $p_0 \in M_k(\C) \ot (P \ovt \widehat{B_j})$ and a unital normal $*$-homomorphism $\vphi : \widehat{A_i} \to p_0(M_k(\C) \ovt P \ovt \widehat{B_j})p_0$ satisfying $\vphi(a) V = V \theta(a)$ for all $a \in \widehat{A_i}$.

Then $VV^* \in \vphi(\widehat{A_i})' \cap p_0(M_k(\C) \ovt P \ovt B)p_0$ and we may assume that the support of $(\id \ot \id \ot E_{\widehat{B_j}})(VV^*)$ equals $p_0$. Using left multiplication by $V^*$, it then follows that
$$\bim{\vphi(\widehat{A_i})}{p_0 (\C^k \ot L^2(P \ovt \widehat{B_j}))}{P \ovt \widehat{B_j}}$$
is isomorphic with a subbimodule of $\bim{\theta(\widehat{A_i})}{p L^2(P \ovt B)}{P \ovt \widehat{B_j}}$. Therefore,
$$\bim{\vphi(\widehat{A_i})}{p_0 (\C^k \ot L^2(P \ovt \widehat{B_j}))}{P \ot 1}$$
is coarse. By the induction hypothesis, $n-1=m-1$ and for every $r \in \{1,\ldots,n\} \setminus \{i\}$, there exists an $s \in \{1,\ldots,n\} \setminus \{j\}$ such that inside $P \ovt \widehat{B_j}$, we have $\vphi(A_r) \prec P \ovt B_s$.

A fortiori, for every $r \in \{1,\ldots,n\} \setminus \{i\}$, there exists an $s \in \{1,\ldots,n\}$ such that inside $P \ovt B$, we have $\theta(A_r) \prec P \ovt B_s$. Since this statement holds for every $i$, the lemma holds for $n$, and is thus proven.
\end{proof}

\begin{theorem}\label{thm.corr-tensor-product}
Let $A_1,\ldots,A_n$ and $B_1,\ldots,B_m$ be nonamenable II$_1$ factors that are stably solid in the sense of Definition \ref{def.rel-solid}. Write $A = A_1 \ovt \cdots \ovt A_n$ and $B = B_1 \ovt \cdots \ovt B_m$.
\begin{enumlist}
\item $A$ and $B$ are W$^*$-correlated if and only if $n=m$ and there exists a permutation $\si$ such that $A_i$ is W$^*$-correlated to $B_{\si(i)}$ for all $i$.
\item For every W$^*$-correlation $\bim{A}{H}{P \ovt B}$ between $A$ and $B$, there exists a permutation $\si$ and an $A$-$(P \ovt B)$-subbimodule $H_0 \subset H$ such that for every $i$, $H_0$ is densely spanned by the $A_i$-$(P \ovt B_{\si(i)})$-subbimodules of $H_0$ that define a W$^*$-correlation between $A_i$ and $B_{\si(i)}$.
\end{enumlist}
\end{theorem}

\begin{proof}
If $n = m$ and $\si$ is a permutation of $\{1,\ldots,n\}$ such that $A_i$ is W$^*$-correlated to $B_{\si(i)}$ for all $i$, it follows from Proposition \ref{prop.Wstar-corr-tensor} that $A$ is W$^*$-correlated to $B$.

To prove the converse, let $\bim{\al(A)}{pL^2(P \ovt B)}{P \ovt B}$ be a $P$-equivalence between $A$ and $B$ that is isomorphic with $\bim{P\op \ovt A}{L^2(P\op \ovt A)q}{\be(B)}$. By Lemma \ref{lem.coarse-embedding-tensor-of-solid}, we get that $n=m$.

Note that $A_i \subset A$ is regular and $A_i \not\prec A_{i'}$ when $i \neq i'$. Similar properties hold for $B_j \subset B$. We apply Lemma \ref{lem.find-subequiv}. Write $H = pL^2(P \ovt B)$ and define the subspaces $L_{i,j} \subset H$ and $R_{i,j} \subset H$ as in Lemma \ref{lem.find-subequiv}. Since $A_i \subset A$ and $B_j \subset B$ are regular, we have that all $L_{i,j}$ and $R_{i,j}$ are $\al(A)$-$(P \ovt B)$-subbimodules of $H$.

By Lemma \ref{lem.coarse-embedding-tensor-of-solid}, for every $i \in \{1,\ldots,n\}$, the linear span of $\{L_{i,j} \mid j = 1,\ldots,n\}$ is dense in $H$. Similarly, for every $j \in \{1,\ldots,n\}$, the linear span of $\{R_{i,j} \mid i = 1,\ldots,n\}$ is dense in $H$. By Lemma \ref{lem.find-subequiv} and using the notation of that lemma, we find that $L_{i,j} = R_{i,j} = K_{i,j}$ for all $i,j \in \{1,\ldots,n\}$. Note that if $K_{i,j} \neq \{0\}$, then $K_{i,j}$ is densely spanned by its $A_i$-$(P \ovt B_{j})$-subbimodules that define a W$^*$-correlation between $A_i$ and $B_{j}$

Since $A_i \subset A$ and $B_j \subset B$ are regular, the projection $z_{i,j}$ of $H$ onto $K_{i,j}$ belongs to the center of $\al(A)' \cap p (P \ovt B)p$. In particular, all projections $z_{i,j}$ commute. By Lemma \ref{lem.find-subequiv}, for every fixed $i$, the projections $z_{i,j}$, $j = 1,\ldots,n$, are orthogonal and sum up to the identity of $H$. Similarly, for every fixed $j$, the projections $z_{i,j}$, $i =1,\ldots,n$, are orthogonal and sum up to the identity of $H$.

Choose $\si(1)$ such that $z_{1,\si(1)} \neq 0$. Then choose $\si(2)$ such that $z_{1,\si(1)} z_{2,\si(2)} \neq 0$. Since $z_{1,j} z_{2,j} = 0$ for all $j$, we must have $\si(1) \neq \si(2)$. Continuing in this way, we find a permutation $\si$ such that the projection $z_0 := z_{1,\si(1)} \cdots z_{n,\si(n)}$ is nonzero. Define $H_0 := z_0 H$. For every $i$, we have by construction that $H_0 \subset K_{i,\si(i)}$, so that $H_0$ is densely spanned by its $A_i$-$(P \ovt B_{\si(i)})$-subbimodules that define a W$^*$-correlation between $A_i$ and $B_{\si(i)}$. In particular for every $i$, $A_i$ is W$^*$-correlated to $B_{\si(i)}$.
\end{proof}

We are now ready to prove Theorem \ref{thm.W-star-corr-tensor-product-any-tensor-product}. We make use of the method introduced in \cite{IM19} to analyze tensor product decompositions $M = M_1 \ovt M_2$ of a II$_1$ factor by analyzing the induced partial flip $\si \in \Aut(M \ovt M)$ given by $\si((a_1 \ot a_2) \ot (b_1 \ot b_2)) = (a_1 \ot b_2) \ot (b_1 \ot a_2)$.

\begin{proof}[Proof of Theorem \ref{thm.W-star-corr-tensor-product-any-tensor-product}]
If $\{1,\ldots,n\} = S_1 \sqcup \cdots \sqcup S_r$ is a partition into nonempty subsets such that $A_{S_k}$ is W$^*$-correlated to $B_k$ for all $k \in \{1,\ldots,r\}$, it follows from a repeated application of Proposition \ref{prop.Wstar-corr-tensor} that $A$ is W$^*$-correlated to $B$.

It then suffices to prove statement (ii) of the theorem. We first assume that $r=2$ and prove (ii). At the end of the proof, we then easily deduce the general case by induction on $r$.

So assume that $B=B_1 \ovt B_2$ is a tensor product of two II$_1$ factors, that $P$ is a tracial von Neumann algebra and that $\bim{\al(A)}{pL^2(P \ovt B)}{P \ovt B}$ is a $P$-equivalence between $A$ and $B$ that is isomorphic with $\bim{P\op \ovt A}{L^2(P\op \ovt A)q}{\be(B)}$.

For every $i \in \{1,\ldots,n\}$ and $k \in \{1,2\}$, denote by $z_{i,k}$ the largest projection in $\al(A_i)' \cap p(P \ovt B)p$ such that $\al(A_i)z_{i,k}$ is amenable relative to $P \ovt B_k$. Since $A_i \subset A$ is regular, we have that all the projections $z_{i,k}$ belong to the center of $\al(A)' \cap p(P \ovt B)p$ and, in particular, commute. Fix $i \in \{1,\ldots,n\}$. Then, $\al(A_i)z_{i,1}z_{i,2}$ is amenable relative to $P \ovt B_1$ and relative to $P \ovt B_2$. By Proposition \ref{prop.basic-prop.2}, $\al(A_i)z_{i,1}z_{i,2}$ is amenable relative to $P \ot 1$. Since $A_i$ is a nonamenable II$_1$ factor, it follows from Proposition \ref{prop.basic-prop.5} that $z_{i,1}z_{i,2} = 0$. So, $(p-z_{i,1}) \vee (p-z_{i,2}) = p$ for all $i \in \{1,\ldots,n\}$.

We can then inductively choose $k_1,\ldots,k_n \in \{1,2\}$ such that $(p-z_{1,k_1}) \cdots (p-z_{j,k_j}) \neq 0$ for all $j = 1,\ldots,n$. Define $z_0 = (p-z_{1,k_1})\cdots (p-z_{n,k_n})$. Then $z_0 \neq 0$ and we replace $p L^2(P \ovt B)$ by $z_0 L^2(P\ovt B)$. Writing $S_1 = \{j \in \{1,\ldots,n\} \mid k_j = 2\}$, $S_2 = \{j \in \{1,\ldots,n\} \mid k_j = 1\}$, we have found a partition of $\{1,\ldots,n\}$ and may assume from the start that $\al(A_i)$ is strongly nonamenable relative to $P \ovt B_2$ for $i \in S_1$, and relative to $P \ovt B_1$ for $i \in S_2$.

For $k \in \{1,2\}$, define $A_{S_k}$ as the tensor product of $A_i$, $i \in S_k$.

{\bf Part 1.} We prove that $\al(A_{S_1}) \prec^f P \ovt B_1$ and $\al(A_{S_2}) \prec^f P \ovt B_2$.

As in \cite{IM19}, we consider the partial flip automorphism
\begin{multline*}
\si \in \Aut((P \ovt B) \ovt (P \ovt B)) : \si((d \ot b_1 \ot b_2) \ot (d' \ot b'_1 \ot b'_2)) \\ = (d \ot b_1 \ot b'_2) \ot (d' \ot b'_1 \ot b_2) \quad\text{for all $d,d' \in P$, $b_1,b'_1 \in B_1$, $b_2,b'_2 \in B_2$.}
\end{multline*}
For every $i \in \{1,\ldots,n\}$, we define $Q_i \subset p(P \ovt B)p \ovt p(P \ovt B)p$ by $Q_i = p \ot \al(A_i)$ if $i \in S_1$ and $Q_i = \al(A_i) \ot p$ if $i \in S_2$. By definition, $\si(Q_i)$ is strongly nonamenable relative to $(P \ovt B) \ovt (P \ot 1)$ for all $i \in \{1,\ldots,n\}$.

Define $N = P \ovt B \ovt P \ovt P\op$ and put $\betil = \id^{\ot 3} \ot \be$, so that
$$\betil : (P \ovt B) \ovt (P \ovt B) \to (1 \ot 1 \ot 1 \ot q)(N \ovt A)(1 \ot 1 \ot 1 \ot q) \; .$$
Then define $p_1 \in N \ovt A$ by $p_1 = \betil(\si(p \ot p))$. By Proposition \ref{prop.basic-prop.5}, $\betil(\si(Q_i))$ are von Neumann subalgebras of $p_1 (N \ovt A) p_1$ that are strongly nonamenable relative to $N \ot 1$. For every $j \in \{1,\ldots,n\}$, denote by $\widehat{A_j}$ the tensor product of all $A_k$, $k \neq j$.

As in the first paragraphs of the proof, define $q_{i,j} \in \betil(\si(Q_i))' \cap p_1(N \ovt A)p_1$ as the largest projection such that $\betil(\si(Q_i)) q_{i,j}$ is amenable relative to $N \ovt \widehat{A_j}$. Since $Q_i$ is regular in $\al(A) \ovt \al(A)$, the projections $q_{i,j}$ all belong to the center of $\betil(\si(\al(A) \ovt \al(A)))' \cap p_1 (N \ovt A)p_1$ and thus all commute. Reasoning as in the first paragraphs of the proof,
\begin{equation}\label{eq.this-is-zero}
q_{i,1} \cdots q_{i,n} = 0 \quad\text{for all $i \in \{1,\ldots,n\}$.}
\end{equation}
For every map $\vphi : \{1,\ldots,n\} \to \{1,\ldots,n\}$, define
$$p_\vphi = (p_1-q_{1,\vphi(1)}) \cdots (p_1-q_{n,\vphi(n)}) \; .$$
It follows from \eqref{eq.this-is-zero} that the join of all $p_\vphi$ equals $p_1$.

We claim that
\begin{equation}\label{eq.nice-claim}
\betil(\si(\al(A_{S_1}) \ovt \al(A_{S_2}))) \prec^f N \ot 1 \; .
\end{equation}
Since the join of all $p_\vphi$ equals $p_1$, to prove the claim, it suffices to prove that
\begin{equation}\label{eq.easier-claim}
\betil(\si(\al(A_{S_1}) \ovt \al(A_{S_2})))p_\vphi \prec^f N \ot 1
\end{equation}
whenever $\vphi : \{1,\ldots,n\} \to \{1,\ldots,n\}$ is a function and $p_\vphi \neq 0$.

So fix a function $\vphi$ with $p_\vphi \neq 0$. For $i \in S_1$, define $R_i = \al(A) \ovt \al(\widehat{A_i})$. For $i \in S_2$, define $R_i = \al(\widehat{A_i}) \ovt \al(A)$. By definition, $R_i$ commutes with $Q_i$. By definition, $\betil(\si(Q_i))p_\vphi$ is strongly nonamenable relative to $N \ovt \widehat{A_{\vphi(i)}}$. Since $A_{\vphi(i)}$ is stably solid, it follows that
\begin{equation}\label{eq.good-embed}
\betil(\si(R_i))p_\vphi \prec^f N \ovt \widehat{A_{\vphi(i)}} \quad\text{for all $i \in \{1,\ldots,n\}$.}
\end{equation}
We claim that $\vphi$ is injective. Otherwise, we find $i \neq i'$ and $k \in \{1,\ldots,n\}$ such that $\vphi(i) = k = \vphi(i')$. Since $R_i$ and $R_{i'}$ normalize each other and together generate $\al(A) \ovt \al(A)$, it then follows from \eqref{eq.good-embed} and Proposition \ref{prop.basic-prop.3} that
\begin{equation}\label{eq.strange-embed}
\betil(\si(\al(A) \ovt \al(A)))p_\vphi \prec^f N \ovt \widehat{A_k} \; .
\end{equation}
Since we can view $\betil \circ \si \circ (\al \ot \al)$ as a $(P \ovt P \ovt P\op)$-equivalence of $A \ovt A$ with $B \ovt A$ and since $B \ovt \widehat{A_k} \subset B \ovt A$ is a subfactor of infinite index, \eqref{eq.strange-embed} contradicts Proposition \ref{prop.not-with-smaller-B.1}. So, $\vphi$ is injective and hence, $\vphi$ is surjective as well.

By definition, $\al(A_{S_1}) \ovt \al(A_{S_2}) \subset R_i$ for all $i \in \{1,\ldots,n\}$. It then follows from \eqref{eq.good-embed}, Proposition \ref{prop.basic-prop.1} and the surjectivity of $\vphi$ that \eqref{eq.easier-claim} holds. So the claim in \eqref{eq.nice-claim} is proven.

Since $\be$ defines a $P\op$-embedding of $B$ into $A$, Proposition \ref{prop.basic-prop.4} and \eqref{eq.nice-claim} imply that
$$\si(\al(A_{S_1}) \ovt \al(A_{S_2})) \prec^f P \ovt B \ovt P \ovt 1 \quad\text{and thus}\quad \al(A_{S_1}) \ovt \al(A_{S_2}) \prec^f (P \ovt B_1) \ovt (P \ovt B_2) \; .$$
This in turn implies that $\al(A_{S_1}) \prec^f P \ovt B_1$ and $\al(A_{S_2}) \prec^f P \ovt B_2$.

{\bf Part 2.} We prove that $\be(B_1) \prec^f P\op \ovt A_{S_1}$ and $\be(B_2) \prec^f P\op \ovt A_{S_2}$.

For all $k \in \{1,2\}$ and $i \in \{1,\ldots,n\}$, define $q_{k,i,1} \in \be(B_k)' \cap q(P\op \ovt A)q$ as the largest projection such that $\be(B_k)q_{k,i,1}$ is amenable relative to $P\op \ovt \widehat{A_i}$. Define $q_{k,i,2} = q-q_{k,i,1}$. Since $B_k \subset B$ is regular, all $q_{k,i,j}$ belong to the center of $\be(B)' \cap q(P\op \ovt A)q$ and, in particular, all commute.

For all functions $\lambda,\rho : \{1,\ldots,n\} \to \{1,2\}$, denote
$$q_{\lambda,\rho} = (q_{1,1,\lambda(1)} \cdots q_{1,n,\lambda(n)}) \cdot (q_{2,1,\rho(1)} \cdots q_{2,n,\rho(n)}) \; .$$
Since by definition, $q_{k,i,1} + q_{k,i,2} = q$ for all $k$ and $i$, the join of all the projections $q_{\lambda,\rho}$ equals $q$. To prove part~2, it thus suffices to fix such functions $\lambda$ and $\rho$ for which $q_{\lambda,\rho} \neq 0$ and to prove that $\be(B_1)q_{\lambda,\rho} \prec^f P\op \ovt A_{S_1}$ and $\be(B_2)q_{\lambda,\rho} \prec^f P\op \ovt A_{S_2}$.

Define $T_1$ as the set of $i$ with $\lambda(i) = 2$. Define $T_2$ as the set of $i$ with $\rho(i) = 2$. Write $q_0 = q_{\lambda,\rho}$. By definition, $\be(B_k) q_0$ is strongly nonamenable relative to $P\op \ovt \widehat{A_i}$ for all $i \in T_k$ and $\be(B_k) q_0$ is amenable relative to $P\op \ovt \widehat{A_i}$ whenever $i \not\in T_k$.

We first prove that $T_1 \cup T_2 = \{1,\ldots,n\}$. Assume the contrary. Take $j \in \{1,\ldots,n\}$ such that $j \not\in T_1$ and $j \not\in T_2$. Define
$$\gamma : A \to P \ovt P\op \ovt A : \gamma(a) = (\id \ot \be)(\al(a)) (1 \ot q_0) \; .$$
Then $\gamma$ defines a $(P \ovt P\op)$-equivalence between $A$ and $A$. By Theorem \ref{thm.corr-tensor-product}, we find an $i \in \{1,\ldots,n\}$ such that
\begin{equation}\label{eq.we-have-this-embedding}
\gamma(A_i) \prec P \ovt P\op \ovt A_j \; .
\end{equation}
Take $k \in \{1,2\}$ such that $i \in S_k$. By part~1 of the proof, $\al(A_{S_k}) \prec^f P \ovt B_k$, so that $\gamma(A_i) \prec^f P \ovt \be(B_k)q_0$. A fortiori, $\gamma(A_i)$ is amenable relative to $P \ovt \be(B_k)q_0$. Since $j \not\in T_k$, we get that $\be(B_k)q_0$ is amenable relative to $P\op \ovt \widehat{A_j}$, so that $\gamma(A_i)$ is amenable relative to $P \ovt P\op \ovt \widehat{A_j}$.

From \eqref{eq.we-have-this-embedding}, we get a nonzero projection $q_1 \leq \gamma(1)$ such that $\gamma(A_i)q_1$ is amenable relative to $P \ovt P\op \ovt A_j$. Since we have just seen that $\gamma(A_i)q_1$ is amenable relative to $P \ovt P\op \ovt \widehat{A_j}$, it follows from Proposition \ref{prop.basic-prop.2} that $\gamma(A_i)q_1$ is amenable relative to $P \ovt P\op \ovt 1$, contradicting Proposition \ref{prop.basic-prop.5}. So we have proven that $T_1 \cup T_2 = \{1,\ldots,n\}$.

Write $T_1^c = \{1,\ldots,n\}\setminus T_1$. Since $\be(B_1)q_0$ is strongly nonamenable relative to $P\op \ovt \widehat{A_i}$ for all $i \in T_1$ and since $A_i$ is stably solid, we get that $\be(B_2)q_0 \prec^f P\op \ovt \widehat{A_i}$ for all $i \in T_1$. By Proposition \ref{prop.basic-prop.1}, we get that
\begin{equation}\label{eq.embed-with-T1c}
\be(B_2)q_0 \prec^f P\op \ovt A_{T_1^c} \; .
\end{equation}
From part~1 of the proof, we know that $\al(A_{S_2}) \prec^f P \ovt B_2$. We now view the image the Hilbert space $L^2(P\op \ovt A)q_0$ as a $P\op$-embedding of $B$ into $A$ and we apply Lemma \ref{lem.left-right-embedding} to this $P\op$-embedding (and thus with the roles of $A$ and $B$ exchanged). With the notation of Lemma \ref{lem.left-right-embedding}, it follows that $E_r(A_{T_1^c},A_{S_2}) \neq \{0\}$, so that $A_{S_2} \prec_A A_{T_1^c}$. This means that $S_2 \subset T_1^c$ and thus $T_1 \subset S_1$.

By symmetry, also $T_2 \subset S_2$. Since $T_1 \cup T_2 = \{1,\ldots,n\}$, we conclude that $T_1 = S_1$ and $T_2 = S_2$. So \eqref{eq.embed-with-T1c} says that $\be(B_2)q_0 \prec^f P\op \ovt A_{S_2}$. By symmetry, also $\be(B_1)q_0 \prec^f P\op \ovt A_{S_1}$. This concludes the proof of part~2.

{\bf\boldmath End of the proof, in case $r=2$.} Write $H = pL^2(P \ovt B)$. Fix $k \in \{1,2\}$. Since $A_{S_k} \subset A$ and $B_k \subset B$ are regular, it follows from part~1 of the proof that $H$ is densely spanned by the $A_{S_k}$-$(P \ovt B_k)$-subbimodules that are finitely generated as a right Hilbert $(P \ovt B_k)$-module. It similarly follows from part~2 of the proof that $H$ is densely spanned by the $(P\op \ovt A_{S_k})$-$B_k$-subbimodules that are finitely generated as a left Hilbert $(P\op \ovt A_{S_k})$-module. As in the last paragraph of the proof of Lemma \ref{lem.find-subequiv}, $H$ is then also densely spanned by the $A_{S_k}$-$(P \ovt B_k)$-subbimodules that are finitely generated both as a left Hilbert $(P\op \ovt A_{S_k})$-module and as a right Hilbert $(P \ovt B_k)$-module. This precisely means that $H$ is densely spanned by its $A_{S_k}$-$(P \ovt B_k)$-subbimodules that define a W$^*$-correlation between $A_{S_k}$ and $B_k$.

{\bf\boldmath End of the proof.} We now prove (ii) in full generality, by induction on $r$. When $r=1$, there is nothing to prove. Let $r \geq 2$ and assume that (ii) holds for $r-1$. Write $B_0 = B_1 \ovt \cdots \ovt B_{r-1}$. Then $H$ is a W$^*$-correlation between $A$ and $B_0 \ovt B_r$. Since we have already proven (ii) in the case $r=2$, we can choose a partition $\{1,\ldots,n\} = S_0 \sqcup S_r$ into nonempty subsets and an $A$-$(P \ovt B)$-subbimodule $H_1 \subset H$ such that for every $k \in \{0,r\}$, $H_1$ is densely spanned by its $A_{S_k}$-$(P \ovt B_k)$-subbimodules that define a W$^*$-correlation between $A_{S_k}$ and $B_k$. Note that then, automatically, every nonzero $A$-$(P \ovt B)$-subbimodule of $H_1$ has the same property.

Choose an $A_{S_0}$-$(P \ovt B_0)$-subbimodule $K \subset H_1$ that defines a W$^*$-correlation between $A_{S_0}$ and $B_0$. By the induction hypothesis, we find a partition $S_0 = S_1 \sqcup \cdots \sqcup S_{r-1}$ of $S_0$ into nonempty subsets and an $A_{S_0}$-$(P \ovt B_0)$-subbimodule $K_0 \subset K$ such that for every $k \in \{1,\ldots,r-1\}$, $K_0$ is densely spanned by its $A_{S_k}$-$(P \ovt B_k)$-subbimodules that define a W$^*$-correlation between $A_{S_k}$ and $B_k$. Define $H_0$ as the closed linear span of $A \cdot K_0 \cdot (P \ovt B)$. By construction, $H_0$ is a nonzero $A$-$(P \ovt B)$-subbimodule of $H_1$. Since $A_{S_k} \subset A$ and $B_k \subset B$ are regular, for every $k \in \{1,\ldots,r-1\}$, $H_0$ is still densely spanned by its $A_{S_k}$-$(P \ovt B_k)$-subbimodules that define a W$^*$-correlation between $A_{S_k}$ and $B_k$. By the last sentence of the previous paragraph, $H_0$ is also densely spanned by its $A_{S_r}$-$(P \ovt B_r)$-subbimodules that define a W$^*$-correlation between $A_{S_r}$ and $B_r$. This concludes the proof of the theorem.
\end{proof}

\section{\boldmath W$^*$-correlations of graph products}\label{sec.Wstar-corr-graph-products}

In this section, we prove our main results on W$^*$-correlations between graph products of stably solid II$_1$ factors. The main technical result is Theorem \ref{thm.corr-graph-product}. We consider graph products given by simple graphs that satisfy the rigidity assumption of \cite{BCC24}: for every vertex $s$, we have $\link(\link s) = \{s\}$. In words, this means that $s$ is the only vertex that is connected to every vertex in $\link s$. We prove that if such graph products of nonamenable, stably solid II$_1$ factors are W$^*$-correlated, already the graph products given by the subgraphs $\link s$ must be W$^*$-correlated.

This suffices to provide in Theorem \ref{thm.uncountable-family-of-non-correlated} the first uncountable family of mutually non W$^*$-correlated II$_1$ factors, and thus prove Theorem \ref{thm.main-A}.

As a consequence of the main technical theorem \ref{thm.corr-graph-product}, we also prove in Theorem \ref{thm.iso-graph-product} a variant of the main result of \cite{BCC24}: if $A_\Gamma$ and $B_\Lambda$ are stably isomorphic graph products given by rigid simple graphs $\Gamma$ and $\Lambda$ and with all the vertex von Neumann algebras being nonamenable, stably solid II$_1$ factors, there exists a graph isomorphism $\si : \Gamma \to \Lambda$ such that $A_s$ and $B_{\si(s)}$ are stably isomorphic for every vertex $s$ of $\Gamma$. In Corollary \ref{cor.counterex}, we prove that a similar result for W$^*$-correlated graph products, and even for virtually isomorphic graph products, fails. We refer to the discussion right after Theorem \ref{thm.main-C} for the precise relation between our result and the main result of \cite{BCC24}.

Before proving these results, we need some preparation. We start by proving a slight generalization of \cite[Theorem 5.6]{BCC24}. Since we can easily deduce the result as a corollary of Proposition \ref{prop.rel-amen-intersection}, we give a detailed proof for completeness.

\begin{corollary}\label{cor.rel-amen-graph-products}
Let $\Gamma = (S,E)$ be a simple graph and $(B_s,\tau_s)_{s \in S}$ a family of tracial von Neumann algebras with graph product $(B,\tau)$. Let $(P,\tau)$ be any tracial von Neumann algebra. For every subset $T \subset S$, denote by $B_T \subset B$ the von Neumann subalgebra generated by $B_s$, $s \in T$.

If $S_1,S_2 \subset S$ and $A \subset p(P \ovt B)p$ is a von Neumann subalgebra that is amenable relative to $P \ovt B_{S_1}$ and relative to $P \ovt B_{S_2}$, then $A$ is amenable relative to $P \ovt B_{S_1 \cap S_2}$.
\end{corollary}
\begin{proof}
We use the notation introduced in Section \ref{sec.graph-products}. Define $J \subset C_\Gamma$ as the set of elements $g \in C_\Gamma$ with the properties that $|s_1 g| = 1 + |g| = |g s_2|$ for all $s_1 \in S_1$ and $s_2 \in S_2$. Note that $C_\Gamma = C_{S_1} J C_{S_2}$. We write $S_0 = S_1 \cap S_2$ and, for every $i \in \{0,1,2\}$, $B_i = B_{S_i}$. We get that
$$B_1 B_2 \cup \bigcup_{g \in J \setminus \{e\}} B_1 B_g^\circ B_2$$
is a total subset of $L^2(M)$.

By Proposition \ref{prop.rel-amen-intersection}, it thus suffices to construct a $B_0$-$M$-bimodule $K$ and prove that
$$\bim{M}{\bigl( L^2(M) \ot_{B_1} [B_1 B_2] \ot_{B_2} L^2(M)\bigr)}{M} \quad\text{and}\quad \bim{M}{\bigl( L^2(M) \ot_{B_1} [B_1 b B_2] \ot_{B_2} L^2(M)\bigr)}{M}$$
are contained in $\bim{M}{\bigl(L^2(M) \ot_{B_0} K\bigr)}{M}$ for all $b \in B_g^\circ$, $g \in J$. The proof of this statement is essentially given in the proof of \cite[Theorem 5.6]{BCC24}. For completeness, we include the following short argument.

As explained at the end of Section \ref{sec.graph-products}, $E_{B_2}(b_1) = E_{B_0}(b_1)$ for all $b_1 \in B_1$. It follows that $x \ot_{B_1} 1 \ot_{B_2} y \mapsto x \ot_{B_0} y$ for all $x,y \in M$, uniquely extends to an $M$-bimodular unitary of $L^2(M) \ot_{B_1} [B_1 B_2] \ot_{B_2} L^2(M)$ onto $L^2(M) \ot_{B_0} L^2(M)$.

Next fix $g \in J$ and write $g = t_1 \cdots t_n$ as a reduced word in the alphabet $S$. Fix $b \in B_g^\circ$. Define $S'_0 = \{s \in S_0 \mid \forall i \in \{1,\ldots,n\} : s \sim t_i\}$. It follows from the definition of $J$ that for $h_1 \in C_{S_1}$ and $h_2 \in C_{S_2}$, the equality $h_1 g = g h_2$ holds if and only if $h_1 = h_2$ and every letter $s$ in a reduced expression of $h_1 = h_2$ belongs to $S'_0$. It also follows from the definition of $J$ that for all $h_1 \in C_{S_1}$ and $h_2 \in C_{S_2}$, we have that $|h_1 g| = |h_1| + |g|$ and $|g h_2| = |g| + |h_2|$. So, $B^\circ_{h_1 g} = B^\circ_{h_1} B^\circ_g$ and $B^\circ_{g h_2} = B^\circ_g B^\circ_{h_2}$.

Altogether, we conclude that $B^\circ_{h_1} B^\circ_g \perp B^\circ_g B_2$ whenever $h_1 \in C_{S_1} \setminus C_{S'_0}$ and, trivially, $B^\circ_{h_1} \subset B_0$ whenever $h_1 \in C_{S'_0}$. It follows that $E_{B_2}(b^* b_1 b) = E_{B_2}(b^* E_{B_0}(b_1) b)$ for all $b_1 \in B_1$, so that $x \ot_{B_1} b \ot_{B_2} y \mapsto x \ot_{B_0} (b \ot_{B_2} y)$ for all $x,y \in M$, uniquely extends to an $M$-bimodular unitary of $L^2(M) \ot_{B_1} [B_1 b B_2] \ot_{B_2} L^2(M)$ onto $L^2(M) \ot_{B_0} ([B_0 b B_2] \ot_{B_2} L^2(M))$. Defining $K$ as the direct sum of $L^2(M)$ and $[B_0 b B_2] \ot_{B_2} L^2(M)$, the corollary is proven.
\end{proof}

For later use, we also record the following lemma. Recall the notation $\ster s = \{s\} \cup \link s$.

\begin{lemma}\label{lem.bimodules-graph-products}
Let $\Gamma = (S,E)$ be a simple graph and $(B_s,\tau_s)_{s \in S}$ a family of diffuse tracial von Neumann algebras with graph product $(B,\tau)$. For every subset $S_0 \subset S$, denote by $B_{S_0} \subset B$ the von Neumann subalgebra generated by $B_s$, $s \in S_0$.
\begin{enumlist}
\item\label{lem.bimodules-graph-products.1} If $S_0,S_1 \subset S$ are disjoint subsets, the bimodule $\bim{B_{S_0}}{L^2(B)}{B_{S_1}}$ is coarse.
\item If $S_0,S_1 \subset S$ and $S_0 \not\subset S_1$, then $B_{S_0} \not\prec B_{S_1}$.
\item\label{lem.bimodules-graph-products.3} If $S_0 \subset S$, $\link S_0 = \bigcap_{s \in S_0} \link s$ and $S_1 = S_0 \cup \link S_0$, then the normalizer of $B_{S_0}$ inside $B$ equals $B_{S_1}$. The bimodule $\bim{B_{S_0}}{L^2(B \ominus B_{S_1})}{B_{S_0}}$ has no nonzero $B_{S_0}$-$B_{S_0}$-subbimodule that is finitely generated as a right Hilbert $B_{S_0}$-module.
\item\label{lem.bimodules-graph-products.4} The von Neumann algebra $B$ is a factor if and only if $B_s$ is a factor for every $s \in S$ satisfying $\ster s = S$.
\end{enumlist}
\end{lemma}
\begin{proof}
Throughout the proof, we use the notations introduced in Section \ref{sec.graph-products}.

(i) Define $J \subset C_\Gamma$ as the set of elements $g \in C_\Gamma$ with the property that $|s_0 g| = 1+|g| = |g s_1|$ for all $s_0 \in S_0$ and $s_1 \in S_1$. As in the proof of Corollary \ref{cor.rel-amen-graph-products}, the Hilbert space $L^2(B)$ is densely spanned by the closed subspaces $[B_{S_0} B_{S_1}]$ and $[B_{S_0} c B_{S_1}]$ with $c \in B_g^\circ$, $g \in J \setminus \{e\}$.

As explained in Section \ref{sec.graph-products}, $E_{B_{S_1}}(b) = \tau(b) 1$ for all $b \in B_{S_0}$. So, $[B_{S_0} B_{S_1}]$ is a coarse $B_{S_0}$-$B_{S_1}$-bimodule. Take $g \in J \setminus \{e\}$. For every $h \in C_{S_0}\setminus \{e\}$, we have that $|hg| = |h| + |g|$, so that $B_{hg}^\circ = B_h^\circ B_g^\circ$. Similarly, $B_{gk}^\circ = B_g^\circ B_k^\circ$ for all $k \in C_{S_1} \setminus \{e\}$.

As in the proof of Corollary \ref{cor.rel-amen-graph-products}, when $h \in C_{S_0} \setminus \{e\}$ and $k \in C_{S_1}$, we have $h g \neq g k$ in $C_\Gamma$. So, $B_h^\circ B_g^\circ \perp B_g^\circ B_{S_1}$ for all $h \in C_{S_0} \setminus \{e\}$. This implies that
$$E_{B_{S_1}}(d^* b c) = \tau(b) E_{B_{S_1}}(d^* c) \quad\text{for all $b \in B_{S_0}$, $c,d \in B_g^\circ$.}$$
It follows that the $B_{S_0}$-$B_{S_1}$-bimodule $[B_{S_0} c B_{S_1}]$ is coarse. In combination with the first paragraph, (i) is proven.

(ii) Take $s \in S_0 \setminus S_1$. By (i), the bimodule $\bim{B_s}{L^2(B)}{B_{S_1}}$ is coarse. Since $B_s$ is diffuse, it follows that $B_s \not\prec B_{S_1}$. A fortiori, $B_{S_0} \not\prec B_{S_1}$.

(iii) Since $B_{\link S_0}$ commutes with $B_{S_0}$, it follows that $B_{S_1}$ is contained in the normalizer of $B_{S_0}$ inside $B$. It then suffices to prove that the bimodule $\bim{B_{S_0}}{L^2(M \ominus B_{S_1})}{B_{S_0}}$ has no nonzero $B_{S_0}$-$B_{S_0}$-subbimodule that is finitely generated as a right Hilbert $B_{S_0}$-module.

Define $L \subset C_\Gamma$ as the set of elements $g \in C_{\Gamma} \setminus C_{S_1}$ with the property that $|s g| = 1+|g| = |gs|$ for all $s \in S_0$. Then $L^2(M \ominus B_{S_1})$ is densely spanned by $B_{S_0} B_g^\circ B_{S_0}$, $g \in L$. It thus suffices to fix $g \in L$ and prove that $[B_{S_0} B_g^\circ B_{S_0}]$ has no nonzero $B_{S_0}$-$B_{S_0}$-subbimodule that is finitely generated as a right Hilbert $B_{S_0}$-module.

Write $g = t_1 \cdots t_n$ as a reduced word in the alphabet $S$. Define $S'_0 = \{s \in S_0 \mid \forall i \in \{1,\ldots,n\} : s \sim t_i \}$. Fix $c \in B_g^\circ$. In the proof of Corollary \ref{cor.rel-amen-graph-products}, we have seen that the map
\begin{equation}\label{eq.my-bimodule}
[B_{S_0} c B_{S_0}] \to L^2(B_{S_0}) \ot_{B_{S'_0}} L^2(B) : a c b \mapsto a \ot_{B_{S'_0}} cb \quad\text{for all $a,b \in B_{S_0}$,}
\end{equation}
is a well-defined $B_{S_0}$-bimodular isometry. Since $g \in L$, we have that $S'_0 \neq S_0$. Choose $s_0 \in S_0 \setminus S'_0$. It then follows from (i) and \eqref{eq.my-bimodule} that the bimodule $\bim{B_s}{[B_{S_0}cB_{S_0}]}{B_{S_0}}$ is coarse. Since $B_s$ is diffuse, this coarse bimodule has no nonzero $B_s$-$B_{S_0}$-subbimodule that is finitely generated as a right Hilbert $B_{S_0}$-module.

A fortiori, $[B_{S_0}cB_{S_0}]$ has no nonzero $B_{S_0}$-$B_{S_0}$-subbimodule that is finitely generated as a right Hilbert $B_{S_0}$-module.

(iv) When $s \in S$ and $\ster s = S$, then $\cZ(B_s) \subset \cZ(B)$ by construction. Conversely, assume that $B_s$ is a factor for every $s \in S$ satisfying $\ster s = S$. Take $a \in \cZ(B)$. Applying (iii) to $S_0 = \{t\}$, it follows that $a \in B_{\ster t}$ for every $t \in S$. Given any family of subsets $(\cF_i)_{i \in I}$ of $S$ with intersection $\bigcap_{i \in I} \cF_i = \cF_0$, we have that $\bigcap_{i \in I} B_{\cF_i} = B_{\cF_0}$, with the convention that $B_\emptyset = \C 1$. Define $S_0 = \{s \in S \mid \ster s = S\}$. Since $\bigcap_{t \in S} \ster t = S_0$, we find that $a \in B_{S_0}$. If $S_0 = \emptyset$, it follows that $a \in \C 1$. If $S_0 \neq \emptyset$, then $B_{S_0}$ is the tensor product of the II$_1$ factors $(B_s)_{s \in S_0}$, which is a factor. Since $a \in \cZ(B_{S_0})$, it again follows that $a \in \C 1$.
\end{proof}

\begin{theorem}\label{thm.corr-graph-product}
Let $\Gamma$, $\Lambda$ be simple graphs with vertex sets $S$, $T$. Assume that both graphs satisfy the rigidity assumption of \cite{BCC24}, meaning that $\link(\link s) = \{s\}$ for every vertex $s$.

Let $(A_s)_{s \in S}$ and $(B_t)_{t \in T}$ be families of nonamenable, stably solid II$_1$ factors. Denote by $A$ and $B$ their graph products w.r.t.\ $\Gamma$ and $\Lambda$.

If $A$ and $B$ are W$^*$-correlated, then for every $s \in S$, there exists a $t \in T$ such that $A_{\link s}$ is W$^*$-correlated with $B_{\link t}$.
\end{theorem}

\begin{proof}
Fix a $P$-equivalence between $A$ and $B$ that is given by $\bim{\al(A)}{pL^2(P \ovt B)}{P \ovt B}$ and that is isomorphic with $\bim{P\op \ovt A}{L^2(P\op \ovt A)q}{\be(B)}$. We apply Lemma \ref{lem.find-subequiv} to the families of subalgebras $A_{\link s} \subset A$, $s \in S$, and $B_{\link t} \subset B$, $t \in T$.

If $s_1,s_2 \in S$ and $\link s_1 \subset \link s_2$, we get that $\link(\link s_2) \subset \link(\link s_1)$, so that $s_1 = s_2$. By Lemma \ref{lem.bimodules-graph-products}, $A_{\link s_1} \not\prec A_{\link s_2}$ when $s_1 \neq s_2$. For every $s \in S$, write $\ster s = \{s\} \cup \link s$. Since $\link(\link s) = \{s\}$, it follows from Lemma \ref{lem.bimodules-graph-products} that the normalizer of $A_{\link s}$ inside $A$ equals $A_{\ster s}$ and that $\bim{A_{\link s}}{L^2(A \ominus A_{\ster s})}{A_{\link s}}$ has no nonzero $A_{\link s}$-subbimodule that is finitely generated as a right Hilbert $A_{\link s}$-module.

By symmetry, we have the same properties for the subalgebras $B_{\link t}$ of $B$.

We now prove that assumption (i) of Lemma \ref{lem.find-subequiv} holds. Fix $s \in S$ and a nonzero projection $p_0 \in \al(A_{\ster s})' \cap p(P \ovt B)p$. We have to prove that there exists a $t \in T$ such that $\al(A_{\link s})p_0 \prec P \ovt B_{\link t}$.

We claim that there exists a $t \in T$ such that $\al(A_s)p_0$ is nonamenable relative to $P \ovt B_{T \setminus \{t\}}$. Assume the contrary. By Corollary \ref{cor.rel-amen-graph-products}, $\al(A_s)p_0$ is amenable relative to $P \ovt B_{T \setminus T_0}$ for every finite subset $T_0 \subset T$. By Lemma \ref{lem.bimodules-graph-products}, the bimodule $\bim{B_{T_0}}{L^2(B)}{B_{T \setminus T_0}}$ is coarse for every finite subset $T_0 \subset T$. Taking $T_0$ larger and larger, it then follows from Corollary \ref{cor.rel-amen-in-the-limit} that $\al(A_s)p_0$ is amenable relative to $P \ot 1$. This contradicts Proposition \ref{prop.basic-prop.5}, so that the claim is proven.

Note that $B$ can be viewed as the amalgamated free product $B = B_{\ster t} \ast_{B_{\link t}} B_{T \setminus \{t\}}$, so that
$$P \ovt B = (P \ovt B_{\ster t}) \ast_{P \ovt B_{\link t}} (P \ovt B_{T \setminus \{t\}}) \; .$$
Define $p_1 \in (\al(A_s)p_0)' \cap p_0(P \ovt B)p_0$ as the largest projection such that $\al(A_s)p_1$ is strongly nonamenable relative to $P \ovt B_{T \setminus \{t\}}$. Note that $p_1 \neq 0$ and that $p_1 \in \al(A_{\ster s})' \cap p (P \ovt B)p$.

Since $\al(A_{\link s}) p_1$ commutes with $\al(A_s)p_1$, it follows from \cite[Theorem 6.4]{Ioa12} that at least one of the following statements holds: $\al(A_{\link s}) p_1 \prec P \ovt B_{\link t}$, or $\al(A_{\ster s}) p_1 \prec P \ovt B_{\ster t}$. In the first case, we have found the required $t \in T$ such that $\al(A_{\link s})p_0 \prec P \ovt B_{\link t}$.

In the second case, we take an integer $n \geq 1$, a projection $q \in M_n(\C) \ovt P \ovt B_{\ster t}$, a nonzero partial isometry $X \in q (\C^n \ovt P \ovt B) p_1$ and a unital normal $*$-homomorphism $\theta : A_{\ster s} \to q(M_n(\C) \ovt P \ovt B_{\ster t})q$ such that $\theta(a) X = X \al(a)$ for all $a \in A_{\ster s}$. We may moreover assume that the conditional expectation of $XX^*$ on $M_n(\C) \ovt P \ovt B_{\ster t}$ has support $q$.
Since $\al(A_s)p_1$ is strongly nonamenable relative to $P \ovt B_{T \setminus \{t\}}$, a fortiori, $\al(A_s)p_1$ is strongly nonamenable relative to $P \ovt B_{\link t}$. By \cite[Lemma 2.3(ii)]{DV24}, we get that $\theta(A_s)$ is strongly nonamenable relative to $P \ovt B_{\link t}$ inside $P \ovt B_{\ster t}$.

Note that $P \ovt B_{\ster t}$ can be viewed as $P \ovt B_t \ovt B_{\link t}$. Since $\theta(A_{\link s})$ commutes with $\theta(A_s)$ and since $B_t$ is stably solid, we conclude that $\theta(A_{\link s}) \prec P \ovt B_{\link t}$ inside $P \ovt B_{\ster t}$. By \cite[Lemma 2.3(ii)]{DV24}, we get that also in the second case, $\al(A_{\link s}) p_1 \prec P \ovt B_{\link t}$.

We have thus proven that assumption (i) of Lemma \ref{lem.find-subequiv} holds. By symmetry also assumption (ii) of that lemma holds. But then Lemma \ref{lem.find-subequiv} provides for every $s \in S$, a $t \in T$ such that $A_{\link s}$ is W$^*$-correlated to $B_{\link t}$.
\end{proof}

Combining Theorems \ref{thm.corr-tensor-product} and \ref{thm.corr-graph-product}, we can now prove Theorem \ref{thm.main-A} and obtain the following uncountable family of non W$^*$-correlated II$_1$ factors.

\begin{theorem}\label{thm.uncountable-family-of-non-correlated}
For every subset $\cF \subset \{2,3,\ldots\}$, define the group
$$G_\cF = \free_{n \in \cF} \bigl(\underbrace{\F_2 \times \cdots \times \F_2}_{n \; \text{times}}\bigr) \; .$$
If $\cF \neq \cF'$, then $L(G_\cF)$ is not W$^*$-correlated with $L(G_{\cF'})$.
\end{theorem}
\begin{proof}
For every subset $\cF \subset \{2,3,\ldots\}$, we define the graph $\Gamma_\cF$ as the disjoint union of the complete graphs with $n \in \cF$ vertices. Taking the free group factor $L(\F_2)$ at every vertex, $L(G_\cF)$ is canonically isomorphic to the graph product. The subalgebras of this graph product corresponding to the subsets $\link s$ are precisely the $(n-1)$-fold tensor powers of $L(\F_2)$ with $n \in \cF$.

If $L(G_\cF)$ is W$^*$-correlated with $L(G_{\cF'})$, it thus follows from Theorem \ref{thm.corr-graph-product} that for every $n \in \cF$, there exists a $k \in \cF'$ such that the $(n-1)$-fold tensor power of $L(\F_2)$ is W$^*$-correlated to the $(k-1)$-fold tensor power. By Theorem \ref{thm.corr-tensor-product}, $n-1=k-1$. We thus conclude that $\cF \subset \cF'$. By symmetry, also the converse inclusion holds.
\end{proof}

When specializing to virtually isomorphic graph products, we can deduce that the underlying graphs must be isomorphic, which is a variant of the main result of \cite{BCC24}. For the precise relation between Theorem \ref{thm.iso-graph-product} and the main result of \cite{BCC24}, we refer to the discussion right after Theorem \ref{thm.main-C} in the introduction.

\begin{theorem}\label{thm.iso-graph-product}
Let $\Gamma$, $\Lambda$ be simple graphs with vertex sets $S$, $T$. Assume that both graphs satisfy the rigidity assumption of \cite{BCC24}, meaning that $\link(\link s) = \{s\}$ for every vertex $s$.

Let $(A_s)_{s \in S}$ and $(B_t)_{t \in T}$ be families nonamenable, stably solid II$_1$ factors. Denote by $A$ and $B$ their graph products w.r.t.\ $\Gamma$ and $\Lambda$.

If $A$ and $B$ are stably isomorphic, there exists a graph isomorphism $\si : S \to T$ such that for every $s \in S$, the II$_1$ factors $A_s$ and $B_{\si(s)}$ are stably isomorphic.
\end{theorem}
\begin{proof}
Take a $*$-isomorphism $\al : A \to p (M_n(\C) \ot B)p$. Define the $A$-$B$-bimodule $\bim{A}{H}{B} = \bim{\al(A)}{p (\C^n \ot L^2(B))}{B}$. In the proof of Theorem \ref{thm.corr-graph-product}, we have seen that the families of von Neumann subalgebras $A_{\link s} \subset A$ and $B_{\link t} \subset B$ satisfy the conditions of Lemma \ref{lem.find-subequiv}. As in Lemma \ref{lem.find-subequiv}, we thus denote for every $s \in S$, $t \in T$ by $K_{s,t}$ the closed linear span of all $A_{\link s}$-$B_{\link t}$-subbimodules of $H$ that are finitely generated as a right Hilbert $B_{\link t}$-module and that are finitely generated as a left Hilbert $A_{\link s}$-bimodule.

Since $A_s$ commutes with $A_{\link s}$ and since $B_t$ commutes with $B_{\link t}$, every $K_{s,t}$ is an $A_{\ster s}$-$B_{\ster t}$-bimodule. Because all $A_s$ and $B_t$ are factors, it follows from Lemma \ref{lem.bimodules-graph-products.3} and \ref{lem.bimodules-graph-products.4} that $A_{\ster s}' \cap A = \C 1$ and that $B_{\ster t}' \cap B = \C 1$. Since $\al$ is an isomorphism, if follows that $K_{s,t} = H$ whenever $K_{s,t} \neq \{0\}$. By Lemma \ref{lem.find-subequiv}, we get that $K_{s,t} \perp K_{s,t'}$ if $t\neq t'$ and that $K_{s,t} \perp K_{s',t}$ if $s \neq s'$. From the proof of Theorem \ref{thm.corr-graph-product}, we know that for every $s$, there exists a $t$ such that $K_{s,t} \neq \{0\}$, and by symmetry, for every $t$, there exists an $s$ such that $K_{s,t} \neq \{0\}$. So, there is a unique bijection $\si : S \to T$ such that for $s \in S$ and $t \in T$, we have that $K_{s,t} \neq \{0\}$ if and only if $t = \si(s)$.

Fix $s \in S$ and write $t = \si(s)$. We first prove that $A_s$ is stably isomorphic to $B_t$. Since $K_{s,t} \neq \{0\}$, we can choose a nonzero bifinite $A_{\link s}$-$B_{\link t}$-subbimodule $K_0 \subset H$, which we may choose to be irreducible. Taking a unital normal $*$-homomorphism $\theta : A_{\link s} \to q (M_k(\C) \ot B_{\link t})q$ such that $\bim{A_{\link s}}{(K_0)}{B_{\link t}}$ is isomorphic with $\bim{\theta(A_{\link s})}{q(\C^k \ot L^2(B_{\link t}))}{B_{\link t}}$, we get that $\theta$ is an irreducible finite index inclusion. From the isomorphism between $q(\C^k \ot L^2(B_{\link t}))$ and $K_0$, we get a nonzero partial isometry $V \in p(M_{n,k}(\C) \ot B)q$ satisfying $\al(a) V = V \theta(a)$ for all $a \in A_{\link s}$. We write $e = V^* V$ and note that $e$ is a projection in $q(M_k(\C) \ot B)q \cap \theta(A_{\link s})'$.

We claim that $\theta(A_{\link s})' \cap q(M_k(\C) \ot B)q = (1 \ot B_t)q$. To prove this claim, take $a \in \theta(A_{\link s})' \cap q(M_k(\C) \ot B)q$. Since $\theta(A_{\link s})$ has finite index in $q (M_k(\C) \ot B_{\link t})q$ and since $a$ commutes with $\theta(A_{\link s})$, the closed linear span $L$ of $(M_{1,k}(\C) \ot B_{\link t}) a (M_{k,1}(\C) \ot B_{\link t})$ is a $B_{\link t}$-$B_{\link t}$-subbimodule of $L^2(B)$ that is finitely generated as a right Hilbert $B_{\link t}$-module. Since $\link(\link t) = \{t\}$, it follows from Lemma \ref{lem.bimodules-graph-products.3} that $L \subset L^2(B_{\ster t})$. So, $a \in q(M_k(\C) \ot B_{\ster t})q$. Viewing $B_{\ster t}$ as $B_{\link t} \ovt B_t$ and using that $\theta$ is an irreducible inclusion, we get that $a \in (1 \ot B_t)q$. So the claim is proven.
In particular, we find a projection $f \in B_t$ such that $e = q(1 \ot f)$.

Since $V^* \al(A_s) V$ commutes with $\theta(A_{\link s})$, we get that $V^* \al(A_s) V \subset (1 \ot f B_t f)q$. Since $\link(\link s) = \{s\}$, it also follows from Lemma \ref{lem.bimodules-graph-products.3} and \ref{lem.bimodules-graph-products.4} that $A_{\link s}' \cap A = A_s$. Since $\al$ is an isomorphism, we find a projection $r \in A_s$ such that $VV^* = \al(r)$. Since $V (1 \ot B_t) V^*$ commutes with $\al(A_{\link s})$, we also find that $V(1 \ot B_t) V^* \subset \al(r A_s r)$. We have thus proven that $V^* \al(r A_s r) V = (1 \ot f B_t f)q$ and find a unital $*$-isomorphism $\psi : r A_s r \to f B_t f$ such that $V^* \al(a) V = (1 \ot \psi(a))q$ for all $a \in r A_s r$. We have thus proven that $A_s$ is stably isomorphic to $B_t$. Note that we have in particular proven that $\al(A_s) \prec B_{\si(s)}$ for every $s \in S$. Since the normalizer $A_{\ster s}$ of $A_s$ has trivial relative commutant, this implies that $\al(A_s) \prec^f B_{\si(s)}$ for all $s \in S$.

It remains to prove that $\si$ is a graph isomorphism. By symmetry, it suffices to fix $s' \in \link s$ and prove that $\si(s') \in \link \si(s)$. Write $t = \si(s)$ and consider the partial isometry $V$ and $*$-homomorphism $\theta$ that we constructed above, with $VV^* \in \al(A_s)$. In particular, $VV^*$ commutes with $\al(A_{s'})$. We have seen that $\al(A_{s'}) \prec^f B_{\si(s')}$. We can thus choose a projection $q' \in M_{k'}(\C) \ot B_{\si(s')}$, a nonzero partial isometry $W \in p(M_{n,k'}(\C) \ot B)q'$ and a normal unital $*$-homomorphism $\theta' : A_{s'} \to q'(M_{k'}(\C) \ot B_{\si(s')})q'$ such that $\al(a) W = W \theta'(a)$ for all $a \in A_{s'}$, and $WW^* \leq VV^*$. Define $X = V^* W$. Because $W$ is nonzero and $WW^* \leq VV^*$, we have that $X$ is nonzero. By construction, $\theta(a) X = X \theta'(a)$ for all $a \in A_{s'}$. Since $A_{s'}$ is diffuse, it follows that the $B_{\link t}$-$B_{s'}$-bimodule $L^2(B)$ is not coarse. By Lemma \ref{lem.bimodules-graph-products.1}, we find that $s' \in \link t$.
\end{proof}

Seeing Theorems \ref{thm.corr-graph-product} and \ref{thm.iso-graph-product}, it is natural to ask if the conclusion of Theorem \ref{thm.corr-graph-product} can be strengthened to the existence of a graph isomorphism $\si : S \to T$ such that $A_s$ is W$^*$-correlated to $B_{\si(s)}$ for every $s \in S$. This even fails for virtually isomorphic graph products, as we prove in Corollary \ref{cor.counterex}, as a consequence of the following general construction of virtually isomorphic graph product groups.

\begin{proposition}\label{prop.counterex}
Let $\Gamma = (S,E)$ be a simple graph and $s_1 \in S$. Let $(G_s)_{s \in S}$ be a family of discrete groups and denote by $G$ their graph product given by $\Gamma$. Let $H < G_{s_1}$ be a subgroup of finite index $k$.

Denote $[k] = \{1,\ldots,k\}$ and define the new simple graph $\Gamma' = (S',E')$ with $S' = (\ster s_1) \sqcup ([k] \times (S \setminus \ster s_1))$ and with the following edges.
\begin{itemlist}
\item For $s,t \in \ster s_1$, we have $(s,t) \in E'$ iff $(s,t) \in E$.
\item For $s,t \in S \setminus \ster s_1$ and $i,j \in [k]$, we have $((i,s),(j,t)) \in E'$ iff $i=j$ and $(s,t) \in E$.
\item For $s \in \ster s_1$, $t \in S \setminus \ster s_1$ and $i \in [k]$, we have $(s,(i,t)) \in E'$ iff $(s,t) \in E$.
\end{itemlist}
Define the discrete groups $G'_{s_1} = H$, $G'_s = G_s$ for all $s \in \link s_1$ and $G'_{(i,s)} = G_s$ for all $i \in [k]$ and $s \in S \setminus \ster s_1$. Denote by $G'$ their graph product given by $\Gamma'$.

Then, $G'$ is isomorphic to an index $k$ subgroup of $G$. If the graph $\Gamma$ is rigid, also the graph $\Gamma'$ is rigid.

More precisely, if $G_{s_1}/H = \{g_1 H,\ldots,g_kH\}$, there is a unique group homomorphism $\vphi : G' \to G$ satisfying $\vphi(g) = g$ for all $g \in G'_s$, $s \in \ster s_1$, and $\vphi(g) = g_i^{-1} g g_i$ for all $g \in G'_{(i,s)}$, $i \in [k]$ and $s \in S \setminus \ster s_1$. This homomorphism $\vphi$ is faithful and $[G:\vphi(G')] = k$.
\end{proposition}
\begin{proof}
Define $K < G$ as the graph product of $(G_s)_{s \in S \setminus \{s_1\}}$ given by restricting $\Gamma$ to $S \setminus \{s_1\}$. Similarly define $L < K$ as the graph product of $(G_s)_{s \in \link s_1}$. Write $H_1 = G_{s_1}$. We have the natural isomorphism $G \cong (H_1 \times L) \ast_L K$. With this notation, $H < H_1$ is a subgroup of index $k$ and $H_1/H = \{g_1 H,\ldots,g_k H\}$.

Then the subgroups of $G$ given by $g_i^{-1} K g_i$, $i \in [k]$, are free with amalgamation over $L$, and this is well-defined because the elements $g_i$ commute with $L$. Denote by $K'$ the subgroup of $G$ generated by $g_i^{-1} K g_i$, $i \in [k]$. Then the subgroups $H \times L$ and $K'$ of $G$ are free with amalgamation over $L$. They generate a subgroup $\cG < G$ such that $\cG \cong (H \times L) \ast_L K'$. By construction, the elements $g_1,\ldots,g_k \in G$ still are coset representatives of $G/\cG$.

By construction, the natural map $\vphi : G' \to \cG$ is a group isomorphism.

Finally, assume that $\Gamma$ is rigid. We prove that $\Gamma'$ remains rigid. Define the map $\pi : S' \to S$ by $\pi(s) = s$ for all $s \in \ster s_1$ and $\pi(i,s) = s$ for all $i \in [k]$ and $s \in S \setminus \ster s_1$. Note that if $(r,r') \in E'$, then $(\pi(r),\pi(r')) \in E$. Also note that $\pi(\link r) = \link \pi(r)$ for all $r \in S'$. Fix $r \in S'$ and take $r' \in \link(\link r)$. We have to prove that $r'= r$. Note that $\pi(r') \in \link(\pi(\link r)) = \link(\link \pi(r)) = \{\pi(r)\}$, by rigidity of $\Gamma$. So, $\pi(r') = \pi(r)$. If $r \in \ster s_1$, it follows that $r' \in \ster s_1$, so that $r'=\pi(r')=\pi(r)=r$. If $r = (i,s)$ with $i \in [k]$ and $s \in S \setminus \ster s_1$, first note that in the rigid graph $\Gamma$, we have that $\{s_1\} = \link(\link s_1) \not\subset \link(\link s) = \{s\}$, so that $\link s \not\subset \link s_1$. Since also $s_1 \not\in \link s$ because $s \not\in \ster s_1$, we get that $\link s \not\subset \ster s_1$. We can thus choose $t \in S \setminus \ster s_1$ such that $(s,t) \in E$. It follows that $(i,t) \in \link r$. But then, $r' \in \link(\link r) \subset (\ster s_1) \cup (\{i\} \times (S \setminus \ster s_1))$. Since the restriction of $\pi$ to $(\ster s_1) \cup (\{i\} \times (S \setminus \ster s_1))$ is injective and $\pi(r') = \pi(r)$, it follows that $r'=r$.
\end{proof}

\begin{corollary}\label{cor.counterex}
Let $\Gamma$ be a finite simple graph that is not complete. Label every vertex with the free group $\F_2$ and denote by $G_\Gamma$ the graph product group. There exists a finite simple graph $\Gamma'$ and a labeling of its vertices with either $\F_2$ or $\F_3$ such that $\Gamma' \not\cong \Gamma$, but the graph product $G_{\Gamma'}$ is isomorphic with a subgroup of $G_\Gamma$ of index $2$.

Moreover, when $\Gamma$ is rigid, also $\Gamma'$ can be chosen rigid. This means that in Theorem \ref{thm.corr-graph-product}, we cannot conclude that the graphs $\Gamma$ and $\Lambda$ are isomorphic, not even if $\Gamma$ and $\Lambda$ are finite and the II$_1$ factors $A$ and $B$ are virtually isomorphic.
\end{corollary}
\begin{proof}
Let $\Gamma = (S,E)$. Since $\Gamma$ is not complete, choose $s_1 \in S$ such that $S \neq \ster s_1$. Viewing $\F_3$ as an index $2$ subgroup of $\F_2$, we perform the construction of Proposition \ref{prop.counterex}. The resulting graph $\Gamma'$ has strictly more vertices than $\Gamma$, so that $\Gamma'\not\cong \Gamma$.
\end{proof}

\end{document}